\documentclass[11pt,a4paper,reqno]{amsart}

\usepackage{pdfsync}
\usepackage{todonotes}
\usepackage{enumerate}

\usepackage{mathabx,epsfig}
\usepackage{tikz}
\usepackage{tikz-cd}
\usetikzlibrary{decorations.pathreplacing,shapes,arrows}

\usetikzlibrary{decorations.markings}
\usetikzlibrary{arrows,matrix}
\usetikzlibrary{snakes}
\usepgflibrary{arrows}
\tikzset{
    >=stealth',
    rDefine style for boxes
    punkt/.style={
           rectangle,
           rounded corners,
           draw=black, very thick,
           text width=6.5em,
           minimum height=2em,
           text centered},
    pil/.style={
           ->,
           thick,
           shorten <=2pt,
           shorten >=2pt,}
}
\tikzset{every loop/.style={min distance=2mm,in=225,out=135,looseness=10}}

\usepackage{tikz}
\usetikzlibrary{decorations.markings}
\usetikzlibrary{arrows,matrix}
\usepgflibrary{arrows}

\usepackage{hyperref}

 \newcommand{\bba}{\mathbb{A}}
 
 \newcommand{\bbe}{\mathbb{E}}
 \newcommand{\bbp}{\mathbb{P}}
 \newcommand{\bbq}{\mathbb{Q}}
 \newcommand{\bbr}{\mathbb{R}}
 \newcommand{\bbs}{\mathbb{S}}
 \newcommand{\bbf}{\mathbb{F}}

\usepackage{clrscode}
\usepackage{mathrsfs,amssymb}

\usepackage{amsmath, amsthm}

\newcommand{\ZM}{{\mathbb{Z}M}}

\newcommand{\ZN}{{\mathbb{Z}N}}
\newcommand{\Z}{\mathbb{Z}}

\newtheorem{theorem}{Theorem}[section]
\newtheorem{Thm}[theorem]{Theorem}
\newtheorem*{theorema}{Theorem A}
\newtheorem*{theoremb}{Theorem B}

\newtheorem{lemma}[theorem]{Lemma}
\newtheorem{Lemma}[theorem]{Lemma}

\newtheorem{Cor}[theorem]{Corollary}

\newtheorem{proposition}[theorem]{Proposition}
\newtheorem{Prop}[theorem]{Proposition}

\theoremstyle{definition}

\newtheorem{remark}[theorem]{Remark}
\newtheorem{definition}[theorem]{Definition}

\newtheorem*{claim}{Claim}

\numberwithin{equation}{section}

\newcommand{\gp}{\mathscr{P}}

\newcommand{\gr}{\mathrel{\mathscr{R}}}
\newcommand{\gl}{\mathrel{\mathscr{L}}}
\newcommand{\gh}{\mathrel{\mathscr{H}}}

\newcommand{\cd}{\mathop{\mathrm{cd}}\nolimits}

\newcommand{\Hom}{\mathop{\mathrm{Hom}}\nolimits}
\newcommand{\til}[1]{\ensuremath{\widetilde {#1}}}
\newcommand{\inv}{^{-1}}

\newcommand{\FPn}{{\rm FP}\sb n}

\newcommand{\F}{{\rm F}}

\newcommand{\FP}{{\rm FP}}
\newcommand{\FPinfty}{{\rm FP}\sb \infty}

\newcommand{\lb}{\langle}
\newcommand{\rb}{\rangle}

\title[Lyndon's identity theorem for monoids]{A Lyndon's identity theorem for \\ one-relator monoids}

\subjclass[2010]{20M50, 20M05, 20J05, 57M07, 20F10, 20F65}

\keywords{
One-relator monoid, 
homological finiteness property,
cohomological dimension,
geometric dimension,
classifying space.
\\
\indent
This work was supported by the
EPSRC grant EP/N033353/1 `Special inverse monoids: subgroups, structure, geometry, rewriting systems and the word problem'.
The second named author was supported by  a PSC-CUNY award and the Fulbright Comission.
This work was supported by the London Mathematical Society Research in Pairs (Scheme 4) grant
(Ref:\ 41844), which funded a 6-day research visit of the second named author to the University of East Anglia (September 2019).
}

\begin{document}
\maketitle

\begin{center}
ROBERT D. GRAY
\footnote{School of Mathematics, University of East Anglia, Norwich NR4 7TJ, England.
Email \texttt{Robert.D.Gray@uea.ac.uk}.
}
and
BENJAMIN STEINBERG\footnote{
Department of Mathematics, City College of New York, Convent Avenue at 138th Street, New York, New York 10031,  USA.
Email \texttt{bsteinberg@ccny.cuny.edu}.}
\\ 
\end{center}

\begin{abstract}
For every one-relator monoid $M = \lb A \mid u=v \rb$ with $u, v \in A^*$
we construct a contractible $M$-CW complex
and use it
to build a projective resolution of the trivial module which is finitely generated in all dimensions.
This proves that all one-relator monoids are of type $\FP_{\infty}$, answering positively a problem posed by Kobayashi in 2000.  
We also apply our results to classify the one-relator monoids of  cohomological dimension at most $2$, and to describe the relation module,
in the sense of Ivanov,
of a torsion-free one-relator monoid presentation as an explicitly given principal left ideal of the monoid ring.  In addition, we prove the topological analogues of these results by showing that all one-relator monoids satisfy the topological finiteness property $\F_\infty$, and classifying the one-relator moniods with geometric dimension at most $2$. These results give a natural monoid analogue of Lyndon's Identity Theorem for one-relator groups.
\end{abstract}

\section{Introduction}
\label{sec:introduction}

Algorithmic problems concerning groups and monoids are a classical topic in algebra and theoretical computer science dating back to pioneering work of Dehn, Thue and Tietze in the early 1900s; see~\cite{BookAndOtto, LyndonAndSchupp, Margolis:1995qo} for references. The most fundamental algorithmic question concerning an algebraic structure is the word problem, which asks whether two expressions over generators represent the same element.  Markov
\cite{Markov1947} and
and Post~\cite{Post1947} proved independently that the word problem for finitely presented monoids is undecidable in general. This result was later extended to cancellative monoids by Turing~\cite{Turing1950} and then later to groups by Novikov~\cite{Novikov1952} and Boone~\cite{Boone1959}.

Since the word problem for finitely presented groups and monoids is in general undecidable, an important theme has been to identify and study classes for which the word problem is decidable.  A classical result of this kind was proved by Magnus~\cite{Magnus1932}
in the 1930s
who showed that all one-relator groups have decidable word problem.
Magnus's work inspired the study of the word problem for other one-relator algebraic structures. For example, Shishov
\cite{Shirshov62} proved that one-relator Lie algebras have decidable word problem.  
Not all one-relator structures are so well behaved; see e.g. the recent result 
\cite{GrayRAAG}. 

In contrast to one-relator groups, far less is currently known about the class of one-relator monoids, that is, monoids defined by presentations of the form  $\lb A \mid u=v \rb$ where $u$ and $v$ are words from the free monoid $A^*$.  Indeed, it is still not known whether the word problem is decidable for one-relator monoids.  This is one of the most fundamental longstanding open problems in combinatorial algebra.  While this problem is open in general, it has been solved in a number of special cases in work of Adjan, Adjan and Oganesyan, and Lallement; see~\cite{Adjan1966,Adyan1987,Lallement1988}. The best known undecidability result for finitely presented monoids with a small number of defining relations is due to  Matiyasevich who proved in~\cite{Matiyasevich1967}
that there is a finitely presented monoid with just three defining relations that has an undecidable word problem.

A problem which is closely related to the word problem for one-relator monoids is the question of whether every one-relator monoid admits a finite complete presentation (meaning a presentation which is confluent and terminating; see~\cite[Chapter~12]{HoltBook}).  A positive answer to this question would solve the word problem for one-relator monoids, since monoids defined by finite complete presentations have decidable word problem.
Another famous open problem in theoretical computer science
related to this
asks if it is decidable whether a one-relator string rewriting system is terminating; see e.g.~\cite{Dershowitz2005, Matiyasevich2005, Shikishima1997}.

The Anick-Groves-Squier theorem~\cite{Anick1986,Brown1989} shows that if a monoid admits a finite complete presentation, then that monoid must satisfy the homological finiteness property $\FP_\infty$.
Given this, and the discussion above, it is natural to ask whether all one-relator monoids are of type $\FP_\infty$.
This problem was investigated by Kobayashi in~\cite{Kobayashi1998,Kobayashi2000} where he proved that all one-relator monoids have finite derivation type in the sense of Squier~\cite{Squier1994}, and hence are of type $\FP_3$.
He then goes on to ask~\cite[Problem~1]{Kobayashi2000} whether all one-relator monoids are of type $\FP_\infty$.
Further motivation for this question comes from Lyndon's Identity theorem for groups~\cite{Lyndon1950} which gives natural resolutions for one-relator groups which, among other things, imply that all one-relator groups are of type $\FP_\infty$.

The first main result of this paper is the following
theorem which gives a positive answer to Kobayashi's problem~\cite[Problem~1]{Kobayashi2000}.

\begin{theorema}\label{theorem:A}
Every one-relator monoid $\lb A \mid u=v \rb$ is of type $\FP_\infty$.
\end{theorema}

This theorem applies to arbitrary one-relator monoids $\lb A \mid u=v \rb$ without any restrictions on the words $u$ and $v$. It should be stressed that there are very few results in the literature that have been proved for arbitrary one-relator monoids.

We note that, motivated in large part by the connection with string rewriting systems and the word problem, there is a extensive body of literature devoted to the study of homological and homotopical finiteness properties of monoids; see e.g.
\cite{Alonso2003, Cohen1992, Guba1998, Pride2004, Squier1987, Squier1994}.
Monoids of type $\FP_n$ are also used to define the higher dimensional BNSR invariants of groups; see~\cite{Bieri1987, Bieri1988}.

As mentioned above, the analogue of Theorem~A for one-relator groups is known to hold as a consequence of results in a highly influential
paper~\cite{Lyndon1950} published by Lyndon in 1950.
As Lyndon states in his paper, his result, which is now known as \emph{Lyndon's Identity Theorem}, may be viewed as a complementary result to Magnus's word problem solution: that of determining all identities among the relations.
More precisely, Lyndon's result 
identifies the relation module of a one-relator group as an explicitly given  
cyclic module. 
When interpreted topologically, Lyndon's results show how to construct certain natural classifying spaces for one-relator groups which are small, in a sense which can be made precise.
Specifically these spaces have finitely many cells in each dimension, which implies that all one-relator groups are of type $\FP_\infty$.
Furthermore, Lyndon's results can be used to prove
that the presentation $2$-complex of a torsion-free one-relator group is aspherical, and thus is a classifying space for the group; see~\cite{Cockcroft1954,
DyerVasquez1973}.
This implies that torsion-free one-relator groups have geometric and cohomological dimension at most two (in fact, exactly two unless they are free).

In this paper we will prove a Lyndon's Identity Theorem for arbitrary one-relator monoids.
We will then apply this to prove Theorem~A.
In addition to this, we shall also apply our results to  classify one-relator monoids of cohomological dimension at most $2$, and to compute explicitly the relation module, in the sense of Ivanov~\cite{Ivanov}, of all torsion-free one-relator monoids.
Our proof uses a range of techniques
and ideas
from algebraic topology,
and
from combinatorial and homological algebra.
In particular, one key ingredient will be some of the results and ideas from our earlier recent papers~\cite{GraySteinberg1, GraySteinberg2} which give a new topological approach to the study of homological finiteness properties of monoids.

One fundamental new insight made in this paper is that we uncover a tree-like structure in the Cayley graph of one-relator monoids, modulo certain subgraphs which have the geometry either of special one-relator monoids, or of strictly aspherical one-relator monoids.
Here a special monoid presentation is one where all of the defining relations are of the form $w=1$, and a monoid  presentation is strictly aspherical if all of its diagram groups,
in the sense of Guba and Sapir~\cite{GubaSapir1997},  are trivial.
These are then both base cases of our induction, in the first instance resolved by general results on finitely presented special monoids from our earlier article~\cite{GraySteinberg2}.

Our topological approach allows us to exploit this geometric information about the structure of the Cayley graph to construct actions of one-relator monoids on suitable contractible CW complexes.  In addition to this, other new tools and innovations introduced in this paper (of independent interest) include:
the development of a generalised compression theory (an induction technique for one-relator monoids) that provides a common framework for both Adjan-Oganesyan compression~\cite{Adjan1966} and Lallement compression~\cite{Lallement1974},
and a new interpretation of this
in terms of
local divisor theory, in the sense of Diekert~\cite{Diekert2016}.

In more detail, for every one-relator monoid $M$ we shall construct a contractible $M$-CW complex, that is, a CW complex with cellular action of $M$, so the augmented cellular chain complex gives rise to a resolution of the trivial module.
In many cases our construction gives an equivariant classifying space for the monoid with finitely many orbits of cells in each dimension.
Here, an equivariant classifying space for a monoid is the monoid-theoretic analogue of the universal cover of a classifying space of a group.
In other cases the space we construct is not an equivariant classifying space for the monoid.
In these cases
additional work is needed, and we shall then need
to combine our topological results with
ideas from the theory of relation modules of moniods in the sense of Ivanov~\cite{Ivanov},
arguments with monoid pictures
in the sense of~\cite{GubaSapir1997},
and a general method from work of Brown ~\cite{Brown1982} and Strebel~\cite{Strebel1983}, in order construct a resolution
of the trivial module
by finitely generated projective modules.

The resolutions that we obtain for one-relator monoids in this paper are in general infinite, even in the case of torsion-free one-relator monoids.  This is necessary since, rather surprisingly, we shall show here that there are torsion-free one-relator monoids with infinite cohomological dimension.  Here a monoid has torsion if it has a non-idempotent element $x$ such that $x$ is equal to a proper power of itself.   Equivalently, a one-relator monoid is torsion-free if and only if all of its subgroups are torsion-free.  Thus in general for a torsion-free one-relator monoid $M$ there need not be a finite dimensional, contractible free $M$-CW complex.  This is in stark contrast to the case of torsion-free one-relator groups which (as discussed above) all have cohomological dimension at most two.
The fact that there are torsion-free one-relator monoids with infinite cohomological dimension is one of things that makes proving Theorem~A far more complicated than the corresponding result for groups.

In addition to proving Theorem~A, we will use the results and constructions described above to prove the following result, which
classifies one-relator monoids with cohomological dimension at most $2$.

\begin{theoremb}\label{theorem:B}
Let $M$ be a monoid defined by a one-relator presentation $\lb A \mid u=v \rb$.
Suppose without loss of generality that $|v| \leq |u|$.
Let $z \in A^*$ be the longest word which is a prefix and a suffix of both $u$ and of $v$.
Then
\begin{itemize}
\item[(i)]
$\mathop{\mathrm{cd}}(M) = \infty$ if $M$ has a maximal subgroup with torsion; and
\item[(ii)]
$\mathop{\mathrm{cd}}(M) \leq 2$ if and only if
$M$ is torsion-free and
either $z$ is the empty word or $z=v$.
\end{itemize}
Furthermore, if $\mathop{\mathrm{cd}}(M) \leq 2$ then $M$ has an equivariant classifying space obtained by attaching $2$-cells to the Cayley graph of $M$.
\end{theoremb}

It is well known that a one-relator group has torsion if and only if the word in the defining relator is itself a proper power.  The analog of this for one-relator monoids is a follows.  The one-relator monoid defined by $\lb A \mid u=v \rb$.  (where $|v| \leq |u|$) has torsion if and only if $u \in vA^* \cap A^* v$ and  upon writing $u = vw \in A^*$ the word $w$ is a proper power.
This will be proved in Lemma~\ref{lem:max:subgroups} below.  This means that whether or not the one-relator monoid $\lb A \mid u=v \rb$ is torsion-free is something which can easily be read off from the defining relation $u=v$.

In Theorem~B we use $\cd(M)$ to denote the left cohomological dimension of the monoid $M$. The obvious dual result for right cohomological dimension also holds. Note that in general the left and right cohmological dimensions of a monoid need not be the same; see~\cite{Guba1998}.

In proving Theorem~B one thing that makes the situation necessarily more complicated than that of one-relator groups is that the equivariant classifying space we construct for one-relator monoids of cohomological dimension
at most two
is not in general the Cayley complex of the monoid. In fact, we shall see that 
in general
for torsion-free one-relator monoids of cohomological dimension
at most two
the Cayley complex is not an equivariant classifying space.

Not all torsion-free one-relator monoids satisfy the conditions in part (ii) of Theorem~B. 
Indeed we shall see in Section~\ref{s:cd} that there are torsion-free one-relator monoids with infinite cohomological dimension.
Indeed, many of them have infinite dimension.
We shall prove (see Proposition~\ref{p:inf.proj.dim}) that, with the same notation as in Theorem~B, we have the following: If  $z$ is non-empty and $z\neq v$, and no other non-empty word is a prefix and suffix of both $u$ and of $v$, then $\mathop{\mathrm{cd}}(M)=\infty$.
This is evidence supporting our (as yet unproved)
suspicion that
every torsion-free one-relator monoid with $\cd(M) > 2$
has $\cd(M) = \infty$. This will be discussed further in Section~\ref{s:cd}.

In terms of relation modules, Lyndon's identity theorem for groups identifies the relation module of a one-relator group as an explicitly given  
cyclic module.  Ivanov~\cite{Ivanov} defined and studied relation modules for semigroups.  Here we shall compute the relation module, in the sense of Ivanov, for all torsion-free one-relator monoids; see Theorem~\ref{t:relation.module.compress}.  Specifically our theorem shows that the relation module is isomorphic to an explicitly given principal left ideal in the monoid ring.  This improves on the result~\cite[Corollary 5.4]{Ivanov}.

The paper is organised as follows. After giving some preliminary definitions and results in Section~\ref{sec:preliminaries}, we then give an outline of the proofs of our main results in Section~\ref{sec:proof-outline}.
In
Section~\ref{s:compression}
we prove results on compression and local divisors, and give a key result Theorem~\ref{t:compress.cayley.graph.general} interpreting compression from a topological point of view.
The proofs of our main results for subspecial monoids are given in
Section~\ref{sec_special}.
We compute the relation module of torsion-free one-relator monoids
in
Section~\ref{s:relmodule}.
We prove a key technical lemma for certain strictly aspherical one-relator monoids using monoid pictures
in Section~\ref{s:injectivity}, and then apply this together with other results
in
Section~\ref{sec:resolving:relation:module}
to resolve the relation module and compete the proof of Theorem~A.
Finally in Section~\ref{s:cd} we
compete the proof of Theorem~B,
and we describe a large class of
torsion-free one-relator monoids with infinite cohomological dimension.

\section{Preliminaries}
\label{sec:preliminaries}

We assume the reader has familiarity with standard notions from group theory (see, e.g.,~\cite{LyndonAndSchupp}), algebraic topology (see, e.g.,~\cite{Hatcher2002}), category theory (see, e.g.,~\cite{MacLaneBook}) and topological and homological methods in group theory (see, e.g.,~\cite{BrownCohomologyBook, GeogheganBook}).
To make the article more self contained, the majority of the necessary concepts from monoid theory will be defined at the points where they are needed.
For more background on semigroup theory we refer the reader to~\cite{Howie}.

\subsection*{$M$-CW complexes}

We now give some background from~\cite{GraySteinberg1} which will be needed later.
Let $M$ be a monoid and let $E(M)$ be the set of idempotents of $M$, i.e., the elements $e\in M$ with $e^2=e$.
A \emph{left $M$-space} is a topological space $X$ with a continuous left action $M \times X \rightarrow X$ where $M$ has the discrete topology.
Throughout the article we shall only work with left $M$-spaces and so we usually omit the word left and simply talk about $M$-spaces.
A \emph{free left $M$-set on $A$} is one isomorphic to $M \times A$ with action $m(m',a) = (mm',a)$.  We call $A$ a \emph{basis} for the free $M$-set.  Notice that a left $M$-set $X$ is free on a basis $A\subseteq X$ if and only if each $x\in X$ can be uniquely expressed in the form $x=ma$ with $m\in M$ and $a\in A$.
A \emph{projective left $M$-set $P$} is one which is isomorphic to a left $M$-set of the form
$\coprod_{a\in A} Me_a$ with $e_a \in E(M)$ for all $a \in A$.
As for $M$-spaces, we shall usually omit the word left when talking about $M$-sets, and so by $M$-set we shall always mean left $M$-set unless explicitly stated otherwise. If $X$ is a left $M$-set, then $M\backslash X$ is the quotient of $X$ by the least equivalence relation identifying $x$ and $mx$ for all $m\in M$ and $x\in X$.  
So $M \backslash X$ denotes the set of equivalence classes of this equivalence relation.  
One defines $X/M$ for a right $M$-set analogously.

A \emph{projective $M$-cell} of dimension $n$ is an $M$-space of the form $Me \times B^n$ where $e \in E(M)$ and $B^n$ is the closed unit ball in $\mathbb{R}^n$ on which $M$ is defined to act trivially; we call it a \emph{free $M$-cell} when $e=1$.
A projective $M$-CW complex is defined inductively where projective $M$-cells $Me \times B^n$ are attached via $M$-equivariant maps from $Me \times S^{n-1}$ to the $(n-1)$-skeleton.
Here $S^{n-1}$ denotes the $(n-1)$-sphere, which is the boundary of the $n$-ball $B^n$.
Formally, a \emph{projective relative $M$-CW complex} is a pair $(X,A)$ of $M$-spaces such that $X=\varinjlim X_n$ with $i_n\colon X_n\to X_{n+1}$ inclusions, $X_{-1}=A$, $X_0 = P_0\cup A$ with $P_0$ a projective $M$-set and where $X_n$ is obtained as a pushout of $M$-spaces
\begin{equation}\label{eq:attach}
\begin{tikzcd}P_n\times S^{n-1}\ar{r}\ar[d,hook] & X_{n-1}\ar[d,hook]\\ P_n\times B^n\ar{r} & X_n \end{tikzcd}
\end{equation}
with $P_n$ a projective $M$-set and $B^n$ having a trivial $M$-action for $n\geq 1$.
 The set $X_n$ is the \emph{$n$-skeleton} of $X$ and if $X_n=X$ and $P_n\neq \emptyset$, then $X$ is said to have \emph{dimension} $n$.
If $A=\emptyset$, we call $X$ a \emph{projective $M$-CW complex}.
Note that a projective $M$-CW complex is a CW complex and the $M$-action takes $n$-cells to $n$-cells.
We say that $X$ is a \emph{free $M$-CW complex} if each $P_n$ is a free $M$-set, i.e., $X$ is built up from attaching free $M$-cells.

More generally we shall say that $X$ is an \emph{$M$-CW complex} if $X$ is built up via a sequence of pushouts as above, but where we drop the requirement that $P_n$ be a projective $M$-set and instead allow it to be an arbitrary $M$-set.  In this case, it is convenient to term an \emph{$M$-cell} to be an $M$-space of the form $Mm\times B^n$ 
for some $m \in M$ 
(with trivial action on $B^n$).  It is no longer true at this level of generality that every $M$-CW complex is built up from attaching $M$-cells, but all the ones that occur in this paper are of this form.

\begin{Prop}\label{prop:chains}
Let $X$ be an $M$-CW complex (not necessarily projective) built via attaching maps as in \eqref{eq:attach}.  Then the $n^{th}$-cellular chain group $C_n(X)$ is isomorphic as a $\mathbb ZM$-module to $\mathbb ZP_n$ where the action of $M$ on $P_n$ is extended linearly to $\mathbb ZP_n$.
\end{Prop}
\begin{proof}
As an abelian group, we can view $C_n(X)=\bigoplus_{x\in P_n} H_n(B^n/S^{n-1})$ where the summand for $x\in P_n$ corresponds to the cell $\{x\}\times B^n$.  Fix an orientation of $B^n$ and let $\eta_x$ be the corresponding generator of $H_n(B^n/S^{n-1})$ in the summand corresponding to $x\in P_n$.  Since the action of $M$ on $P_n\times B^n$ is given by $m(x,y)=(mx,y)$, the action of $m$ takes the copy of $B^n/S^{n-1}$ corresponding to the cell $\{x\}\times B^n$ to the copy of $B^n/S^{n-1}$ corresponding to the cell $\{mx\}\times B^n$ via the identity map in the $B^n$-coordinate.  Hence the action of $M$ is orientation preserving and so $m\eta_x=\eta_{mx}$.  Thus $C_n(X)\cong \mathbb ZP_n$ as a $\mathbb ZM$-module.
\end{proof}

If $A$ is a right $M$-set and $X$ is a left $M$-space, then $A\otimes_M X$ is $(A\times X)/{\sim}$, equipped with the quotient topology, where $\sim$ is the least equivalence relation such that $(am,x)\sim (a,mx)$ for all $a\in A$, $m\in M$ and $x\in X$.  The class of $(a,x)$ is denoted $a\otimes x$. If $A$ is a free right $M$-set with basis $B$, then $A\otimes_M X\cong \coprod_{b\in B}b\otimes X$ and $x\mapsto b\otimes x$ is a homeomorphism from $X$ to $b\otimes X$. See~\cite[Remark~2.3]{GraySteinberg2} for details.

An \emph{$M$-homotopy} between $M$-equivariant continuous maps $f,g\colon X\to Y$, between $M$-spaces $X$ and $Y$, is an $M$-equivariant mapping $H\colon X\times I\to Y$ with $H(x,0)=f(x)$ and $H(x,1)=g(x)$ for $x\in X$ where $I$ is viewed as having the trivial $M$-action.
 We write $f\simeq_M g$ if there is an $M$ homotopy between $f$ and $g$.
 We say that the $M$-spaces $X$ and $Y$ are \emph{$M$-homotopy equivalent}, written $X\simeq_M Y$, if there are $M$-equivariant continuous mappings  $f\colon X\to Y$ and $g\colon Y\to X$ such that $gf\simeq_M 1_X$ and $fg\simeq_M 1_Y$.
In this situation the mappings $f$ and $g$ are called \emph{$M$-homotopy equivalences}.

A (left) \emph{equivariant classifying space} $X$ for a monoid $M$ is a projective $M$-CW complex which is contractible.
This notion generalises from group theory the idea of the universal cover of a classifying space of a group.
Indeed, if $Y$ is an Eilenberg--Mac Lane complex of type $K(G,1)$ then the universal cover $X$ of $Y$ is a free $G$-CW complex which is contractible.
Thus $X$ is an equivariant classifying space for the group.
It is proved in~\cite[Section~6]{GraySteinberg1} that every monoid admits an equivariant classifying space, and it is unique up to $M$-homotopy equivalence.
One of our main interests in this paper will be in certain natural topological finiteness properties for monoids defined in terms of the existence of equivariant classifying spaces satisfying certain conditions.
A finitely generated monoid $M$ is of type \emph{left-$\F_n$} if there is a left equivariant classifying space $X$ for $M$ with \emph{$M$-finite} $n$-skeleton.
This means that for all $k \leq n$ the set of $k$-cells is a finitely generated projective $M$-set.  We say that $M$ is of type \emph{left-$\F_{\infty}$} if there is a left equivariant classifying space $X$ for $M$, 
which has $M$-finite $n$-skeleton for all $n \geq 0$.    
The \emph{left geometric dimension} of a monoid $M$, denoted  $
\mathop{\mathrm{gd}^{\mathrm{(l)}}} (M)$
is the minimum dimension of a left equivariant classifying space for $M$, where
$\mathop{\mathrm{gd}^{\mathrm{(l)}}} (M) =\infty$  if there is no finite dimensional equivariant classifying space.
Both of these notions have corresponding homological counterparts,
namely $\FP_n$ and cohomological dimension,  
which have been well studied in the literature over the past few decades; see, e.g.,~\cite{Cohen1992, Guba1998, Otto1997}.
A monoid $M$ is said to be of type \emph{left-$\FPn$} (for a positive integer $n$) if there is a projective resolution $P_\bullet$ of the trivial left $\ZM$-module $\Z$  such that $P_i$ is finitely generated for $i \leq n$.
We say that $M$ is of type \emph{left-$\FPinfty$} if there is a projective resolution $P_\bullet$ of $\Z$ over $\ZM$ with $P_i$ finitely generated for all $i$.
For any monoid $M$, if $M$ is of type left-$\F_n$ for some $0 \leq n \leq \infty$, then it is of type left-$\FP_n$.
This is because the augmented cellular chain complex of an equivariant classifying space for $M$ provides a projective resolution of the trivial left $\mathbb ZM$-module $\mathbb Z$.
For finitely presented monoids the conditions left-$\F_n$ and left-$\FP_n$ are equivalent; see~\cite{GraySteinberg1}.
The \emph{left cohomological dimension} of a monoid $M$, denoted
$\mathop{\mathrm{cd}^{\mathrm{(l)}}}(M)$, is the smallest non-negative integer $n$ such that there exists a projective resolution $P_\bullet = (P_i)_{i \geq 0}$ of $\Z$ over $\ZM$ of length $\leq n$, i.e., satisfying $P_i = 0$ for $i>n$.
If no such $n$ exists, then we set  $\mathop{\mathrm{cd}^{\mathrm{(l)}}}(M)=\infty$.
The left geometric dimension is clearly an upper bound on the left cohomological dimension of a monoid.
All of these finiteness properties have obvious dual notions of right-$\FP_n$ and right cohomological dimension  $\mathrm{cd}^{\mathrm{(r)}}(M)$  working with right $M$-spaces and right modules.
The properties left-$\F_n$ and right-$\F_n$  do not coincide.
The same is true for left- and right-$\FP_n$, and for the left and right cohomological dimensions; see~\cite{Cohen1992, Guba1998}.  However, since the class of one-relator monoids is left-right dual all our results for left homological properties have analogues for right homological properties.

Recall that an \emph{isomorphism} $f\colon X\to Y$ of CW complexes is a cellular mapping which is a homeomorphism with the property that it sends (open) cells to (open) cells.  We then write $X\cong Y$ if they are isomorphic CW complexes.

\subsection*{The Squier complex of a presentation}

The \emph{free monoid} over the alphabet $A$,
which we denote by $A^*$,
consists of all words over $A$, including the empty word $\varepsilon$, with respect to the operation of concatenation of words.
We use $A^+$ to denote the free semigroup over the alphabet $A$.
A \emph{monoid presentation} is a pair $\langle A \mid R \rangle$ where $A$ is an alphabet and $R \subseteq A^* \times A^*$ is a set of defining relations which we usually write in the form $u_i=v_i$ $(i \in I)$.
The monoid defined by the presentation $\langle A \mid R \rangle$ is the quotient $A^* / \sigma$ of the free monoid $A^*$ by the congruence $\sigma$ generated by the set of defining relations $R$; see~\cite[Section~1.6]{Howie}.
If $M$ is the monoid defined by the presentation $\langle A \mid R \rangle$ we write $[w]_M$ to denote the image in $M$ of $w \in A^*$. When it is clear from context in which monoid we are working, we simply write $[w]$ rather than $[w]_M$. So, with this notation, given two words $w_1, w_2 \in A^*$ we write $w_1 = w_2$ to mean that $w_1$ and $w_2$ are equal as words from $A^*$, while $[w_1]_M = [w_2]_M$ means that $w_1$ and $w_2$ both represent the same element of the monoid $M = \langle A \mid R \rangle$.
We say that $z \in A^*$ is a \emph{factor} of the word $w \in A^*$ if $w = \alpha z \beta$ for some words $\alpha, \beta \in A^*$.

The \emph{Squier complex} of a monoid presentation $\gp = \lb A | R \rb$, where each relation $r \in R$ is written as $r_{+1} = r_{-1}$, is the $2$-complex $\Gamma(\gp)$ defined as follows.   It will be convenient to describe the $1$-skeleton using the convention of Serre~\cite{Serre}, where the edge set is given with a fixed-point-free involution and geometric edges correspond to pairs $\{e,e\inv\}$ of inverse edges.  So $\Gamma(\gp)$ has vertex set $V = A^*$ and edge set
\[
E = \{ (w_1, r, \epsilon, w_2)\mid \ w_1, w_2 \in A^*, r \in R, \ \mbox{and} \ \epsilon \in \{ +1, -1 \} \}.
\]
The incidence  functions $\iota, \tau \colon  E \rightarrow V$  are defined by $\iota \bbe = w_1 r_{\epsilon} w_2$ and $\tau \bbe = w_1 r_{- \epsilon} w_2$ for $\bbe=(w_1,r,\epsilon,w_2)$. We call these the \emph{initial} and \emph{terminal} vertices, respectively, of the edge $\bbe$. The mapping $(\,\,)^{-1} \colon  E \rightarrow E$ associates with each edge $\bbe = (w_1, r, \epsilon, w_2)$ its inverse edge $\bbe^{-1} = (w_1, r, -\epsilon, w_2)$.

Observe that two words are connected by an edge in the Squier complex exactly if one can be transformed into the other by a single application of a relation from the presentation. It follows that two words represent the same element of the monoid defined by the presentation if and only if they belong to the same connected component of the Squier complex.

A path in $\Gamma(\gp)$ is a sequence of edges $\bbp = \bbe_1 \circ \ldots \circ \bbe_n$ where $\tau \bbe_i  \equiv \iota \bbe_{i+1}$ for $i=1, \ldots, {n-1}$.  Here $\bbp$ is a path from $\iota \bbe_1$ to $\tau \bbe_n$ and we extend the mappings $\iota$ and $\tau$ to the set of paths by defining $\iota \bbp \equiv \iota \bbe_1$ and $\tau \bbp \equiv \tau \bbe_n$.  The inverse of a path $\bbp = \bbe_1 \circ \bbe_2 \circ \ldots \circ \bbe_n$ is defined by $\bbp^{-1} = \bbe_n^{-1} \circ \bbe_{n-1}^{-1} \circ \ldots \circ \bbe_1^{-1}$.  A path $\bbp$ is called   \emph{closed} if $\iota \bbp \equiv \tau \bbp$. For two paths $\bbp$ and $\bbq$ with $\tau \bbp \equiv \iota \bbq$ the composition $\bbp \circ \bbq$ is defined in the obvious way by concatenating the paths. 
For the remainder of this section let us abbreviate $\Gamma(\gp)$ to $\Gamma$.  
We denote the set of all paths in $\Gamma$ by $P(\Gamma)$, where for each vertex $w \in V$ we include a path $1_w$ with no edges, called the \emph{empty path} at $w$.

There is a natural two-sided action of the free monoid $A^*$ on the sets of vertices and  edges of $\Gamma$ defined by left and right multiplication on vertices and by
\[
x \cdot \bbe \cdot y = (x w_1, r, \epsilon, w_2 y)
\]
for an edge $\bbe = (w_1, r, \epsilon, w_2)$ and $x, y \in A^*$. This extends to paths, where for a path $\bbp = \bbe_1 \circ \bbe_2 \circ \ldots \circ \bbe_n$, we define
\[
x \cdot \bbp \cdot y = (x \cdot \bbe_1 \cdot y) \circ (x \cdot \bbe_2 \cdot y)
\circ \ldots \circ (x \cdot \bbe_n \cdot y).
\]

For every $r \in R$ and $\epsilon = \pm 1$ define $ \bba_r^\epsilon = (1, r, \epsilon, 1). $ Such edges are called \emph{elementary}. The elementary edges are in obvious correspondence with the relations of the presentation, and every edge of $\Gamma$ can be written uniquely in the form $\alpha \cdot \bba \cdot \beta$ where $\alpha, \beta \in A^*$ and $\bba$ is elementary.  Thus the elementary edges form a basis for the two-sided action on edges.

It is often useful to represent the edge $\bbe=(\alpha, r, \epsilon, \beta)$ geometrically by an object called a \emph{monoid picture} as follows:
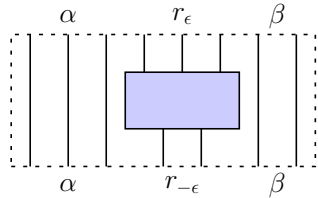
\begin{center}
\scalebox{0.5}
{
\begin{tikzpicture}
[very thick,
blackbox/.style={draw, fill=blue!20, rectangle, minimum height=1.5cm, minimum width=3cm},
dottedbox/.style={draw, rectangle, loosely dashed, minimum height=3.5cm, minimum width=8cm}]
\node at (4.5,2.75) [dottedbox] {};
\node at (5,2.75) [blackbox] {};
\draw (1,1) -- (1,4.5);
\draw (2,1) -- (2,4.5);
\draw (3,1) -- (3,4.5);
\draw (7,1) -- (7,4.5);
\draw (8,1) -- (8,4.5);
\draw (4,3.5) -- (4,4.5);
\draw (5,3.5) -- (5,4.5);
\draw (6,3.5) -- (6,4.5);
\draw (4.5,1) -- (4.5,2);
\draw (5.5,1) -- (5.5,2);
\node at (2,5) {\huge $\alpha$};
\node at (2,0.5) {\huge $\alpha$};
\node at (5,5) {\huge $r_\epsilon$};
\node at (5,0.5) {\huge $r_{-\epsilon}$};
\node at (7.5,5) {\huge $\beta$};
\node at (7.5,0.5) {\huge $\beta$};
\end{tikzpicture}
}
\end{center}

The rectangle in the center of the picture is called a \emph{transistor}, and corresponds to the relation $r$ from the presentation, while the line segments in the diagram are called \emph{wires}, with each wire labelled by a unique letter from $A$. The monoid picture for $\bbe^{-1}$ is obtained by taking the vertical mirror image of the picture of $\bbe$. By stacking such pictures vertically, and joining corresponding wires, we obtain pictures for paths in the graph $\Gamma$. We refer the reader to~\cite{Squier1994, GubaSapir1997, GubaSapirDirected} for a more formal and comprehensive treatment of monoid pictures.  
In order to obtain that homotopy of paths corresponds to isotopy of monoid pictures (up to cancelling mirror image pairs of transistors), it is necessary to adjoin $2$-cells to $\Gamma$.

If $\bbe_1$ and $\bbe_2$ are edges of $\Gamma$, then $(\bbe_1 \cdot \iota \bbe_2) \circ (\tau \bbe_1 \cdot \bbe_2)$ and $(\iota \bbe_1 \cdot \bbe_2) \circ (\bbe_1 \cdot \tau \bbe_2 )$ are paths from $\iota\bbe_1\cdot \iota\bbe_2$ to $\tau\bbe_1\cdot \tau\bbe_2$ and so we may attach a $2$-cell $c_{\bbe_1,\bbe_2}$ with boundary path
\[
(\bbe_1 \cdot \iota \bbe_2) \circ (\tau \bbe_1 \cdot \bbe_2)\cdot
\left((\iota \bbe_1 \cdot \bbe_2) \circ (\bbe_1 \cdot \tau \bbe_2 )\right)\inv.
\]
The $2$-complex obtained 
by attaching all such $2$-cells for every pair of edges in $\Gamma$  
is called the \emph{Squier complex} of the presentation $\gp$.
Note that the Squier complex
is an $A^*\times (A^*)^{op}$-CW complex where the action of a pair $x,y\in A^*$ on cells is given by $xc_{\bbe_1,\bbe_2}y= c_{x\bbe_1,\bbe_2y}$.

If $\bbp$ and $\bbq$ are paths such that $\iota \bbp \equiv \iota \bbq$ and $\tau \bbp \equiv \tau \bbq$ then we say that $\bbp$ and $\bbq$ are \emph{parallel}, and write $\bbp \parallel \bbq$. We use $\parallel\; \subseteq P(\Gamma) \times P(\Gamma)$ to denote the set of all pairs of parallel paths in $\Gamma$. 

From the way we attached $2$-cells to obtain $\Gamma$, homotopy of paths 
in $\Gamma$ (where two paths in $P(\Gamma)$ are homotopic if they are homotopic in the Squier complex) is the smallest equivalence relation on parallel paths such that:

\begin{enumerate}[(H1)]
\item If $\bbe_1$ and $\bbe_2$ are edges of $\Gamma$, then
\[
(\bbe_1 \cdot \iota \bbe_2) \circ (\tau \bbe_1 \cdot \bbe_2) \sim
(\iota \bbe_1 \cdot \bbe_2) \circ (\bbe_1 \cdot \tau \bbe_2 ).
\]

\item For any $\bbp, \bbq \in P(\Gamma)$ and $x,y \in A^*$
\[
\bbp \sim \bbq \ \ \mbox{implies} \ \  x \cdot \bbp \cdot y \sim x \cdot \bbq \cdot y.
\]

\item For any $\bbp, \bbq, \bbr, \bbs \in P(\Gamma)$ with $\tau \bbr \equiv \iota \bbp \equiv \iota \bbq$ and $\iota \bbs \equiv \tau \bbp \equiv \tau \bbq$
\[
\bbp \sim \bbq \ \ \mbox{implies} \ \ \bbr \circ \bbp \circ \bbs \sim \bbr \circ \bbq \circ \bbs.
\]

\item If $\bbp \in P(\Gamma)$ then $\bbp \circ \bbp^{-1} \sim 1_{\iota \bbp}$, where $1_{\iota \bbp}$ denotes the empty path at the vertex $\iota \bbp$.
\end{enumerate}

Condition (H1) captures the idea that when applying non-overlapping relations to a word it does not matter in which order we apply them. In terms of monoid pictures (H1) corresponds to taking two non-overlapping transistors in two adjacent edges in a path, and then pulling the lower transistor up, whilst pushing the higher transistor down. For this reason, this operation is often called pull-up push-down (see below for more on this).  One can show that two monoids pictures with the same upper and lower boundary labels correspond to homotopic paths in $\Gamma$ if and only if they are isotopic (up to cancellation of mirror image pairs of transistors)~\cite{GubaSapir1997}.

Let $\bbp, \bbq \in P(\Gamma)$ with $\bbp \parallel \bbq$. We say that $\bbq$ can be obtained from $\bbp$ by deleting a cancelling pair of edges if there are paths $\bbp_1, \bbp_2$ and an edge $\bbe$ such that
\[
\bbp = \bbp_1 \circ \bbe \circ \bbe^{-1} \circ \bbp_2
\]
and
\[
\bbq = \bbp_1 \circ \bbp_2.
\]
In this situation we also say that $\bbp$ is obtained from $\bbq$ by inserting a cancelling pair of edges.

We say that $\bbq$ can be obtained from $\bbp$ by an application of pull-up push-down if there are paths $\bbr$ and $\bbs$, and edges $\bbe_1$ and $\bbe_2$ such that
\[
\bbp = \bbr \circ
(\bbe_1 \cdot \iota \bbe_2) \circ (\tau \bbe_1 \cdot \bbe_2) \circ
\bbs
\]
and
\[
\bbq = \bbr \circ
(\iota \bbe_1 \cdot \bbe_2) \circ (\bbe_1 \cdot \tau \bbe_2 ) \circ
\bbs.
\]
In this situation we also say that $\bbp$ can be obtained from $\bbq$ by an application of pull-up push-down, and that $\bbq$ can be obtained from $\bbp$ by an application of pull-up push-down.  
Following Kobayashi~\cite{Kobayashi1998}, we will use $\sim_0$ to denote homotopy in $\Gamma$.

The following straightforward lemma follows easily from the definitions above.

\begin{lemma}\label{lem:pupd}
Let $\bbp$ and $\bbq$ be paths in $P(\Gamma)$. Then  $\bbp \sim_0 \bbq$ if and only if there is a sequence of paths
\[
\bbp = \bbp_0, \bbp_1, \ldots, \bbp_k = \bbq
\]
such that, for all $i \in \{1,2,\ldots, k-1\}$, the path $\bbp_{i+1}$ can be obtained from $\bbp_i$ by one of (i) deletion of a cancelling pair of edges (ii) insertion of a cancelling pair of edges, or (iii) an application of pull-up push-down.
\end{lemma}

Following Kobayashi's terminology~\cite{Kobayashi1998}, the presentation $\gp$ is \emph{strictly aspherical} if each connected component of $\Gamma$ is simply connected.  In other words, $\gp$ is strictly aspherical if $\sim_0$ and $\parallel$ coincide as equivalence relations on paths. 
So, $\gp$ is strictly aspherical if and only if 
every closed path is \emph{null-homotopic}, meaning that  
for every closed path $\bbp$ in $\Gamma$ we have $\bbp \sim_0 1_{\iota \bbp}$.
Equivalently, $\gp$ is strictly aspherical if all the diagram groups associated to it, in the sense of Guba and Sapir~\cite{GubaSapir1997}, are trivial.
We note that what we call strictly aspherical presentations here are also referred to simply as \emph{aspherical} presentations in some places in the literature.

\section{Statement of results and proof outline}
\label{sec:proof-outline}

The main objects of study in the paper are monoids defined by presentations of the form $\lb A \mid u=v \rb$ where $A$ is a finite alphabet and $u, v \in A^*$. We call a monoid defined by a presentation $\lb A \mid u=v \rb$ with a single defining relation a \emph{one-relator monoid}. These are also called \emph{one-relation monoids} in the literature.
Following~\cite{Lallement1974} we call one-relator monoids of the form $\lb A \mid u=1 \rb$ \emph{special one-relator monoids}.

As explained in the introduction, our aim is to construct suitable $M$-CW complexes for one-relator monoids which can then be used to prove Theorem~A.  Sometimes these complexes will be equivariant classifying spaces that lead directly to a proof the monoid is of type right- and left-$\F_\infty$.
In other cases we shall have to combine our topological results with algebraic methods, arguments with monoid pictures, and a general method from~\cite{Brown1982} and~\cite{Strebel1983} in order to prove that the monoid is of type $\mathrm{FP}_\infty$.
Recall from the preliminaries section that the properties $\F_\infty$ and $\FP_\infty$ are
are equivalent for finitely presented monoids~\cite{GraySteinberg1}.

Specifically we will prove the following result.

\begin{theorem}[Theorem~A]\label{theorem:A:full}
Every one-relator monoid $\lb A \mid u=v \rb$ is of type left- and right-$\F_\infty$ and of type left- and right-$\FP_\infty$.
\end{theorem}

We can say quite a bit more in the case of the, so-called, subspecial one-relator monoids since,
as we shall see,
these are the monoids for which
the
$M$-CW complexes
we construct
will be equivariant classifying spaces.

\begin{definition}[Subspecial relation]\label{def:subspecial}
Let $u, v \in A^*$ with $|v| \leq |u|$.
The relation $u=v$ is called \emph{subspecial} if  $u \in vA^* \cap A^* v$.
\end{definition}

The terminology `subspecial' is originally due to Kobayashi~\cite{Kobayashi2000}.
Note that in particular the relation $u=1$ is subspecial.
Subspecial one-relator monoids were first considered by Lallement~\cite{Lallement1974} where, among other things, he showed that the word problem is decidable for this class.  He also showed that any one-relator presentation of a monoid containing a non-trivial element of finite order is subspecial.

Let $M$ be the one-relator monoid defined by the presentation
\[
\langle A \mid u=v \rangle.
\]
Without loss of generality we may assume that $|v| \leq |u|$. 
The above presentation of $M$ is \emph{compressible} if there is a non-empty word $r\in A^+$ such that $u,v\in rA^*\cap A^*r$.  Otherwise, the presentation is \emph{incompressible}.  Compression first appeared in the paper of Lallement~\cite{Lallement1974} for subspecial presentations and then, more generally, in the work of Adjan and Oganesyan~\cite{Adyan1987}.  
If $u,v\in rA^*\cap A^*r$ then there  
is an associated compressed one-relator monoid $M_r$, with shorter defining relation, whose word problem is equivalent to that of $M$.  Compression is transitive (each compression of a compression of $M$ is a compression of $M$ in its own right) and confluent, and hence there is a unique incompressible monoid $M'$ to which $M$ compresses.  Moreover, the monoid $M'$ is special if and only if $M$ is subspecial.  More details on compression can be found in Section~\ref{s:compression}.

We shall prove the following result from which Theorem~B will follow. This result actually says a bit more than Theorem~B, since
(2)(ii) identifies a broad family of torsion-free one-relator monoids with infinite cohomological dimension.

\begin{theorem}[Theorem~B]\label{theorem:B:full}
Let $M$ be a monoid defined by a one-relator presentation $\lb A \mid u=v \rb$ with $|v|\leq |u|$.
\begin{enumerate}
\item  Suppose the presentation is subspecial, i.e.,  $u\in vA^*\cap A^*v$.
\begin{itemize}
\item[(i)] If  upon writing $u = vw \in A^*$ the word $w$ is a proper power, then
\[
\mathrm{cd}^{\mathrm{(l)}}(M) =
\mathrm{cd}^{\mathrm{(r)}}(M) =
\mathrm{gd}^{\mathrm{(l)}}(M) =
\mathrm{gd}^{\mathrm{(r)}}(M) =
 \infty.
\]
\item[(ii)] In all other cases
\[
\mathrm{cd}^{\mathrm{(l)}}(M) \leq  \mathrm{gd}^{\mathrm{(l)}}(M) \leq 2,
\]
and
\[
\mathrm{cd}^{\mathrm{(r)}}(M) \leq  \mathrm{gd}^{\mathrm{(r)}}(M) \leq 2.
\]
\end{itemize}
\item Suppose the presentation is not subspecial.
\begin{itemize}
\item[(i)] If the presentation is incompressible, then
\[
\mathrm{cd}^{\mathrm{(l)}}(M) \leq  \mathrm{gd}^{\mathrm{(l)}}(M) \leq 2,
\]
and
\[
\mathrm{cd}^{\mathrm{(r)}}(M) \leq  \mathrm{gd}^{\mathrm{(r)}}(M) \leq 2.
\]
\item[(ii)] If the presentation is compressible by a unique word $r\in A^+$, then
\[
\mathrm{cd}^{\mathrm{(l)}}(M) =
\mathrm{cd}^{\mathrm{(r)}}(M) =
\mathrm{gd}^{\mathrm{(l)}}(M) =
\mathrm{gd}^{\mathrm{(r)}}(M) =
 \infty.
\]
\item[(iii)] In all other cases, the left and right cohomological and geometric dimensions are at least $3$.
\end{itemize}
\end{enumerate}
\end{theorem}

In the rest of this section we will give an outline of our strategy for constructing $M$-CW complexes for one-relator monoids which will be used to prove Theorem~\ref{theorem:A:full} and Theorem~\ref{theorem:B:full}.
The proof of our results divides into consideration of two cases: the \emph{subspecial case} and the \emph{non-subspecial case}.

When $u=v$ is a subspecial relation, that is
$u \in vA^*  \cap A^*v$,
then clearly the
relation $u=v$ is incompressible if and only if $v=1$, that is,
if and only if
the presentation
$\langle A \mid u=v \rangle$
is special.
For special one-relator monoids
$\langle A \mid u=1 \rangle$
the two main results of this paper,
Theorems~\ref{theorem:A:full} and~\ref{theorem:B:full}, are both consequences of more general results for finitely presented special monoids proved in our earlier paper; see~\cite[Section~3]{GraySteinberg2}.

In the non-subspecial incompressible case we shall see that a natural (at most) $2$-dimensional $M$-CW complex, called the Cayley complex, is an equivariant classifying space for the monoid.
The Cayley complex, for a monoid given by a presentation, is defined in the following way.
Let $S$ be a monoid given by a presentation with generators $A$ and defining relations $u_i = v_i$ with $i\in I$.
The right \emph{Cayley (di)graph} $\Gamma(S,A)$ of $S$ with respect to $A$ is the graph with vertex set $S$ and with edges in bijection with $S\times A$ where the directed edge corresponding to $(s,a)$ starts at $s$ and ends at $sa$.
Throughout this paper we shall always work with right Cayley graphs of finitely generated monoids, and so by the Cayley graph of a finitely generated monoid, we will always mean the right Cayley graph.
The \emph{Cayley complex} $X=\Gamma(S,A)^{(2)}$ of $S$,
with respect to the presentation $\lb A \mid u_i = v_i \; (i\in I) \rb$,
is the $2$-dimensional, free $S$-CW complex with $1$-skeleton $\Gamma(S,A)$, and with a free $S$-cell $S \times B^2$ attached for each defining relation in the following way.
For each relation $u_i=v_i$, with $i\in I$, let $p_i$, $q_i$ be paths from $1$ to $m_i$ labelled by $u_i$ and $v_i$ respectively where $m_i$ is the image of $u_i$ (and $v_i$) in $M$.
We glue in $S \times B^2$ so that $\{s \} \times B^2$ is attached via the translate of the loop $p_iq_i^{-1}$ by $s$ (see~\cite[Proposition~2.3]{GraySteinberg1} for more details).
The Cayley complex $X$ is a simply connected, free $S$-CW complex of dimension at most $2$ and it is $S$-finite if the presentation is finite.
In particular, if the Cayley complex is contractible then it is an equivariant classifying space for the monoid.
Note that the Cayley complex depends on the choice of presentation for $S$.
A key fact for our proof below is that the Cayley complex of a strictly aspherical presentation turns out to be contractible; see Lemmas~\ref{lem:9} and~\ref{l:free.contract} below.

Given two words $x, y \in A^*$ we define
\[
\mathrm{OVL}(x,y) = \{ w \in A^+ :
x = x_1 w, \; y = w y_1 \ \mbox{for some} \ x_1, y_1 \in A^* \}.
\]
\begin{lemma}\label{lem:kob98}
\cite[Corollary~5.6]{Kobayashi1998}
\label{lem_Kob}  Let $\lb A \mid u=v \rb$ be a non-subspecial presentation with $|u| \geq |v| \geq 1$. Let $\lambda$ be the longest common prefix of $u$ and $v$, and let $\rho$ be their longest common suffix.  If $\mathrm{OVL}(\rho, \lambda) = \emptyset$, then $\lb A \mid u=v \rb$ is strictly aspherical.
\end{lemma}

\begin{proposition}\label{prop:one}
Let $M= \lb A \mid u=v \rb$ be a one-relator monoid with $|u| \geq |v|$. Then one of the following must hold:
\begin{enumerate}
\item[(a)] $\lb A \mid u=v \rb$ is compressible; or
\item[(b)] $\lb A \mid u=v \rb$ is special, i.e., $v=1$; or else
\item[(c)] $\lb A \mid u=v \rb$ is strictly aspherical.
\end{enumerate}
\end{proposition}
\begin{proof}
Suppose $\lb A \mid u=v \rb$ is not special and is not strictly aspherical.  A subspecial, but not special, presentation is obviously compressible.  If the presentation is not subspecial and not strictly aspherical, then by Lemma~\ref{lem_Kob} it follows that $\mathrm{OVL}(\rho,\lambda) \neq \emptyset$, with $\lambda$ and $\rho$ as in Lemma~\ref{lem:kob98}. This implies that there is a word $\alpha \in A^+$ such that
\begin{align*}
u &\in \alpha A^* \cap A^* \alpha \; \mbox{and} \; v \in \alpha A^* \cap A^* \alpha
\end{align*}
and so the presentation is compressible.
\end{proof}

It follows that every one-relator monoid can be compressed either to a special one-relator monoid or to a strictly aspherical one-relator monoid. The first case happens if and only if the original monoid is subspecial (see~\cite[Lemma~5.4]{Kobayashi2000}).

In this way, Proposition~\ref{prop:one}  allows us to divide the proofs of our main results into considering the non-subspecial and subspecial cases separately.

In both the subspecial and non-subspecial cases our general approach is the same. Let $M=\langle A \mid u=v \rangle$ be a one-relator monoid, and let $M'$ be the incompressible monoid obtained from $M$ using compression. Note that in general $M'$ will be an infinitely generated one-relator monoid. Moreover, the one-relator monoid $M'$ will be special if $M$ is subspecial, and otherwise $M'$ will be strictly aspherical by Lemma~\ref{lem:kob98}. In Theorem~\ref{t:compress.cayley.graph.general} we shall prove a key result which shows how the structure of the right Cayley graphs of $M$ and $M'$ are related to each other. Specifically, that result shows that the Cayley graph $\Gamma$ of $M$ is homotopy equivalent to an infinite star graph with disjoint copies of the Cayley graph $\Gamma'$ of $M'$ attached at each leaf of the star. We use this result to build $M$-CW complexes from equivariant classifying spaces of the compressed monoid $M'$.

While the general approach is the same, the specific constructions that we give are different in the subspecial and non-subspecial cases, as are the properties enjoyed by the resulting $M$-CW complexes. Let us now explain in more detail how the proofs will proceed in each of these two cases. This description will also explain why the result  that torsion-free subspecial one-relator monoids have cohomological dimension at most two (see Theorem~B) does not generalise to torsion-free one-relator monoids in general.

\subsubsection*{The subspecial case}

This case is dealt with in Section~\ref{sec_special}.
The idea in this case is to reduce the problem to the maximal subgroups of the monoid.
Recall that, for any idempotent $e$ in a monoid $N$, the set $eNe$ is a subsemigroup of $N$, and $eNe$ is a monoid with identity element $e$.
The group of units of $eNe$ is a subgroup
(i.e., a subsemigroup which is a group)
of $N$ which is called the \emph{maximal subgroup} of $N$ containing $e$.
The maximal subgroup containing $e$ is the largest subgroup of $N$ with identity element $e$.

We shall see that if $M$ is a subspecial, but not special, monoid then it contains non-trivial idempotents, and all of its maximal subgroups (except the group of units, which is trivial) are isomorphic to a single group $G$.  Moreover, this group $G$ is a one-relator group.  In the subspecial case we give a method for constructing an equivariant classifying space for the monoid from an equivariant classifying space for this group.  Our construction preserves dimension, and the property $\F_n$, and this will allow us to use it to prove our main theorems in the subspecial case, 
by inputting to our construction the classifying spaces for $G$ 
given by Lyndon's Identity Theorem for one-relator groups \cite{Lyndon1950}.

In more detail, the key steps for constructing the equivariant classifying space in the subspecial case are as follows.
Let $M$ be the monoid $\langle A \mid u = v \rangle$ where $|u| \geq |v|$ and $u \in vA^* \cap A^*v$.
Let $M'$ be the one-relator monoid to which $M$ compresses.
In general $M'$ will be an infinitely generated one-relator monoid, and so $M'$ is a free product of a finitely generated special monoid $S$ and a free monoid of infinite rank (see Section~\ref{s:compression} for details).
Let $G$ be the group of units of the monoid $S$.
Then $G$ is a finitely generated one-relator group.
Then the following conditions are equivalent:
\begin{itemize}
\item upon writing $u = vw$ the word $w$ is a proper power;
\item the group $G$ is a one-relator group with torsion;
\item all maximal subgroups of $M$ are one-relator groups with torsion (in fact, all of these groups are isomorphic to $G$).
\end{itemize}

To construct an equivariant classifying space for $M$, with the properties satisfying the main theorems, we begin with an equivariant classifying space for $G$ given by Lyndon's results, from this we construct an equivariant classifying space for the one-relator special monoid $S$ using the method given in our paper~\cite{GraySteinberg2}.
This is then extended to give an equivariant classifying space for the compressed monoid $M'$.
We then use this space together with Theorem~\ref{t:compress.cayley.graph.general},
described above, which relates the structure of the Cayley graph of $M$ to that of $M'$, to
build an equivariant classifying space for $M$.
The constructions we give preserve $M$-finiteness, and they do not increase the dimension of the space.
This will then suffice to prove our main results
Theorem~\ref{theorem:A:full} and Theorem~\ref{theorem:B:full}
in the subspecial case.

\subsubsection*{The non-subspecial case}

In this case the compressed monoid $M'$ is strictly aspherical (or, more precisely, the finitely generated one-relator free factor is).
It follows
(see Lemma~\ref{lem:9})
that the Cayley complex $Y$ of $M'$ with respect to the compressed presentation is contractible.
We use this fact together with Theorem~\ref{t:compress.cayley.graph.general}
 to construct a contractible $M$-CW complex $K$ defined in the following way.
Let $z \in A^*$ be the longest word which is a prefix and a suffix of both $u$ and $v$, and write $u = zu'$ and $v=zv'$.
Then $K$ is the $2$-complex with $1$-skeleton the Cayley graph $\Gamma$ of $M$, and a $2$-cell adjoined at every vertex in $M[z]_M$ with boundary path labelled by $u'(v')^{-1}$.
(See Section~\ref{s:compression}  for full details.) Now the $M$-CW complex $K$ is contractible by Theorem~\ref{t:compress.cayley.graph.general},
since $Y$ is contractible.
However, $K$ will not in general be a projective $M$-CW complex and hence will not be an $M$-equivariant classifying space in this case.
Despite this, since $K$ is contractible its augmented cellular chain complex gives an exact sequence of $\mathbb{Z}M$-modules.
In Theorem~\ref{t:relation.module.compress}
this exact sequence is used to reduce the problem of proving Theorem A in the non-subspecial case to the problem of showing that the left $\mathbb{Z}M$-module  $\mathbb{Z}M[z]_M$ is of type $\mathrm{FP}_\infty$.
In this case we shall see that $\mathbb{Z}M[z]_M$ is in fact isomorphic to the relation module of $M$ in the sense of Ivanov~\cite{Ivanov}.
Using algebraic methods, arguments with monoid pictures, and a general result from~\cite{Brown1982} and~\cite{Strebel1983}, which gives a useful general criterion for a module to be of type $\mathrm{FP}_\infty$, we shall prove in Theorem~\ref{t:complete.fp.infinity}
that the left $\mathbb{Z}M$-module $\mathbb{Z}M[z]_M$ is of type $\mathrm{FP}_\infty$.
This is then combined with Theorem~\ref{t:relation.module.compress} to deduce that $M$ is of type left- and right-$\mathrm{FP}_\infty$; see Theorem~\ref{t:nonsub}.

As noted above, the $M$-CW $2$-complex $K$ which is used to prove that $M$ is of type $\mathrm{FP}_\infty$ is contractible and has dimension at most $2$, but is not projective and hence not an $M$-equivariant classifying space.
More generally, we shall see in Section~\ref{s:cd} that not only does $K$ fail to be an equivariant classifying space for $M$, but in fact it will not be possible to construct any two-dimensional equivariant classifying space for a non-subspecial compressible one-relator monoid.  And in some cases, it will not even be possible to construct a finite dimensional one. For example we will see in Section~\ref{s:cd} that the monoid $\langle a,b,c \mid aba = aca \rangle$ is a torsion-free non-subspecial one-relator monoid of infinite cohomological dimension.
This example is in fact a special case of a much more general result about the cohomological dimension of compressible non-subspecial one-relator monoids which we prove in Section~\ref{s:cd}; see Proposition~\ref{p:inf.proj.dim}.
A result of independent interest that we prove in the process of studying cohomological dimension is that the monoid ring $\mathbb ZM$ of a non-subspecial one-relator monoid has no idempotents except $0$ and $1$; see Theorem~\ref{t:no.idems}.
This strongly generalises Lallement's result~\cite{Lallement1974} that
a non-subspecial one-relator monoid
has no idempotents other than $1$.

This completes the outline of the proof of our main results.

\section{Compression and local divisors}\label{s:compression}

\subsection*{Compression}
We simultaneously generalise here some results from Lallement~\cite{Lallement1974} and from Adjan and Oganesyan~\cite{Adyan1987}.  Note that we frequently use the dual formulation of Lallement's results since we are working with left actions.
Recall that a submonoid $N$  of $A^*$ is \emph{left unitary} if $z,z w \in N$ implies $w \in N$.
Let $N$ be a left unitary submonoid of $A^*$.
A word $z \in N$ is said to be \emph{irreducible} if it belongs to the set
$(N\setminus \{\varepsilon\})\setminus (N\setminus \{\varepsilon\})^2$.
Let $P$ be the set of
irreducible elements of $N$.
Then $P$ is a prefix code (i.e., no element of $P$ is a prefix of another element of $P$); see~\cite[Proposition~7.2.2]{Howie}.  Indeed, if $y\in P$ and $y=z w $ with $z\in P$ and $w\neq \varepsilon$, then $w \in N$ and so $y$ is not irreducible.  Thus $N$ is a free monoid, freely generated by the prefix code $P$. Indeed, if $w\in N$ is non-empty, then its unique factorization as a product of elements of $P$ can be found as follows.  Since $P$ is a prefix code, there is a unique prefix $u$ of $w$ belonging to $P$.  If $w=uv$, then $v\in N$ by the left unitary property and has a unique factorization as a product of elements of $P$ by induction on word length.

Let $r\in A^+$. We say that $r$ \emph{seals} the word $w\in A^*$ if $w\in rA^*\cap A^*r$.  For example, $aba$ seals $ababa$.  In particular, the initial and final occurrences of $r$ may overlap.  If both $r,s\in A^+$ seal $w$ and $|r|\leq |s|$, then $r$ seals $s$.  Indeed, $w=rw'=sw''$ and $w=v'r=v''s$ shows that
 $s\in rA^*\cap A^*r$.  Conversely, if $s$ seals $w$ and $r$ seals $s$, then $r$ seals $w$ as $s=rx=yr$ and $w=sw'=w''s$ implies $rxw'=w=w''yr$.  Thus sealing is transitive.

Throughout this section $M$ will denote the one-relator monoid defined the presentation  $\lb A\mid u=v\rb$.
The presentation
$\lb A \mid u=v \rb$
is said to be \emph{compressible} if there is a non-empty word $r\in A^+$  sealing $u$ and $v$.
In this case we say that \emph{$r$ compresses $u=v$}.  Otherwise, we say that the presentation is \emph{incompressible}.

If both $r,s\in A^+$  compress $u=v$ and $|r|\leq |s|$, then $r$ seals $s$, as observed above.
Note that if a presentation is compressible there is a unique longest word and a unique shortest word compressing $u=v$.   A word $r\in A^+$ is \emph{self-overlap-free} (SOF) if no proper non-empty prefix of $r$ is also a suffix, that is, $r$ is not sealed by any non-empty word except itself.  It follows immediately that if $x$ is the shortest word compressing a relation $u=v$,  then $x$ is SOF and it is the only SOF word compressing the relation.

Let $r\in A^+$ and put
\begin{align*}
T(r)&=\{w \in A^*\mid rw \in A^*r\}.
\end{align*}

Observe that $T(r)$ is a submonoid of $A^*$.  Indeed, $r\varepsilon =r\in A^*r$, so $\varepsilon\in T(r)$, and if $z,w \in T(r)$, then $rz w =z'rw =z'w'r$ for some $z',w'\in A^*$.  So $z w \in T(r)$.  Also $T(r)$ is left unitary.  If $z,z w \in T(r)$, then $r z=z'r$ and so $rz w =z'rw \in A^*r$.  It follows that $rw \in A^*r$ by length considerations.
So $w \in T(r)$ and this completes the proof the $T(r)$ is left unitary.

Since $T(r)$ is freely generated by the prefix code $\Delta_r$ of its irreducible elements, we shall often identify $T(r)$ with $\Delta_r^*$.
Throughout this section we shall endeavor to use lower case Latin letters for elements of $A$ and lower case Greek letters for elements of $\Delta_r$.  In general, the prefix code $\Delta_r$ is infinite.  The code $\Delta_r$ is easiest to describe when $r$ is SOF. In this case, $T(r) = A^*r\cup \{\varepsilon\}$ because if $rw=w'r$ with $w$ non-empty, then the suffix $r$ of $rw$ cannot overlap the prefix $r$ and so $w\in A^*r$.
Therefore, $\Delta_r = (A^*\setminus A^*rA^*)r$ as all occurrences of $r$ in an element of $A^*r$ are disjoint and so an element of $(T(r)\setminus \{\varepsilon\})^2$ is a word of the form $wr$ where $w$ contains a factor $r$.
The description of the prefix code $\Delta_r$ is more complicated when $r$ is not SOF, see~\cite[Section~3]{Lallement1974} for details.
We shall not require a detailed description of the set $\Delta_r$ here.
In our proofs we will only need to refer to the existence of the set $\Delta_r$ and some of its basic properties, e.g., the fact that it is a prefix code which freely generates $T(r)$.
Notice that membership in $T(r)$ is decidable 
and given a word $w\in T(r)$, one can effectively find its factorization into elements of $\Delta_r$ by finding the first non-empty prefix of $w$ in $T(r)$, which will necessarily belong to $\Delta_r$, and then repeating the process on the remainder of the word.

If $r$ compresses $u=v$, then we can write $u=ru'=u_2r$ and $v=rv'=v_2r$ and so $u',v'\in T(r)$.   Thus $u'=\alpha_1\cdots \alpha_k$  and $v'=\beta_1\cdots \beta_\ell$ with $\alpha_i,\beta_j\in \Delta_r$ for all $i,j$.
We can thus form the infinitely generated one-relator monoid
\begin{equation}
L = \lb \Delta_r\mid \alpha_1\cdots \alpha_k=\beta_1\cdots \beta_\ell\rb.
\label{eqn:L}
\end{equation}
Note that the $\alpha_i$ and $\beta_j$ appearing in $u'$ and $v'$ need not be distinct.
Also, note that if $r=u$ or $r=v$, then one of the sides of the defining relation of $L$ can be the empty word (and that is the only way that this can happen).  If we identify $T(r)$ with $\Delta_r^*$, we can, abusing notation, write the presentation $L=\lb \Delta_r\mid u'=v'\rb$.
Notice that the word length of $u'$ and $v'$ over $\Delta_r$ is bounded above by the corresponding word length over $A$ and hence $k<|u|$ and $\ell<|v|$ as $r$ is non-empty.

Let
\[
\Lambda_r =
\left(
\bigcup_{1 \leq i \leq k} \{\alpha_i \}
\right)
\cup
\left(
\bigcup_{1 \leq j \leq l} \{\beta_j \}
\right)
\]
noting that $\Lambda_r$ is a finite subset of $\Delta_r$,
and put $\Psi_r=\Delta_r\setminus \Lambda_r$.
Then $L$ has a free product decomposition $L=\Psi_r^*\ast S$ where $S$ is the finitely presented one-relator monoid
\begin{equation}
S=\lb \Lambda_r\mid \alpha_1\cdots \alpha_k=\beta_1\cdots \beta_\ell\rb.
\label{eqn:S}
\end{equation}
We call $S$ the \emph{compression of $M$ with respect to $r$}.  Notice that $S$ has a shorter defining relation than $M$ since $k<|u|$ and $\ell<|v|$.

\begin{Lemma}\label{l:useful}
Suppose that $r$ seals $s$, and that $s$ seals $u$.  Write $s=rs'$ and $u=ru'$.  Then $s',u'\in T(r)$ and $s'$ seals $u'$ in $\Delta_r^*$.
\end{Lemma}
\begin{proof}
Since $r$ seals $s$ and $u$, we have that $rs'=s\in A^*r$ and $ru'=u\in A^*r$, whence $s',u'\in T(r)$.  Also we have that $ru'=u=su''=rs'u''$ and $u=u_0s=u_0rs'$ for some $u'',u_0\in A^*$ as $s$ seals $u$.  Thus we have that $u'=s'u''$ and hence, since $u',s'\in T(r)$, we have that $u''\in T(r)$ because $T(r)$ is left unitary.  Also from $u_0rs'=u=ru'$ we have that $r$ is a prefix of $u_0r$, so write $u_0r=rw$.  Then $rws'=u_0rs'=ru'$ and so $ws'=u'$.  Since $rw=u_0r$, clearly $w\in T(r)$.
Since $ru'=rs'u''$ it follows that $u' = s'u''$ with $u'' \in T(r)$.
Also, $u' = ws'$ with $w \in T(r)$.
Combining these yields
$u'\in s'\Delta_r^*\cap \Delta_r^*s'$, as required.
\end{proof}

The following proposition expresses a transitivity property of compression that, in particular, implies that every iterated compression of $M$ is a compression of $M$.

\begin{Prop}\label{p:compress.transitivity}
Let $M=\lb A\mid u=v\rb$ be compressed by $r$ with compression $M'=\lb A'\mid u'=v'\rb$.
 Then there is a prefix preserving bijection $\varphi$ from the set of words $s$
from $A^+$
compressing $u=v$ with $|s|>|r|$ and the set of words
from $(A')^+$
compressing the presentation of $M'$.
 If $s$ is such a word and $M''=\lb A''\mid u''=v''\rb$ is the compression of
$M$ with respect to $s$, then the compression $S$ of $M'$ with respect to $\varphi(s)$ is given by $\lb B\mid x=y\rb$ where there is a bijection $\Phi\colon A''\to B$ such that $\Phi(u'')=x$ and $\Phi(v'')=y$.
 In other words, up to relabelling the alphabet, the compressions of $M$ with respect to $s$ and of $M'$ with respect to $\varphi(s)$ are the same.
\end{Prop}
\begin{proof}
Throughout this proof we identify $\Delta_r^*$ with $T(r)$.  Then $u=ru'$ and $v=rv'$.
Let $s \in A^*$ be a word compressing $u=v$ with $|s| > |r|$.
We have already observed that $r$ seals $s$ and so $s=rs'=s''r$ for some $s',s''\in A^*$.   Lemma~\ref{l:useful} implies that $u',v',s'\in T(r)$ and that $s'$ seals
both $u'$ and $v'$ over $\Delta_r$ and so $s'$ compresses $u'=v'$.
For every word $s \in A^*$ compressing $u=v$ with $|s| > |r|$
define
 $\varphi(s)=s'$ where
$s = rs'$.
Then $\varphi$ defines a map from the set of all such words $s$ to the set of words from $(A')^+$ compressing $u'=v'$.
Our aim is to prove that $\varphi$ is a
prefix preserving
bijection.

Clearly $\varphi$ is injective on the set of all such words $s$. Moreover, cancelling $r$ is clearly prefix preserving from $rA^*\to A^*$ and so $\varphi$ is prefix preserving.   We show that $\varphi$ is surjective onto the set of words compressing $u'=v'$.

Suppose that $s'\in \Delta_r^*$ with $u',v'\in s'\Delta_r^*\cap \Delta_r^*s'$.
We show that $s=rs'$ seals $u$.  A similar argument will show that it seals $v$ and hence $s$ compresses $r$ with $\varphi(s)=s'$.
Write $u'=u_1s'=s'u_2$ in $\Delta_r^*$ and note that we can view these as factorizations over $A^*$, as well.  Since $ru'=u$, we have that if $s=rs'$, then $u=ru' = rs'u_2\in sA^*$ and also $u=ru' = ru_1s'$.  But $u_1\in T(r)$ and so $ru_1s'\in A^*rs'=A^*s$.  This shows that $s=rs'$ seals $u$.
This completes the proof that $\varphi$ is surjective.

Fix now a word $s$ compressing $u=v$ with $|s|>|r|$ and let $s'=\varphi(s)$.  First note that since $s$ is sealed by $r$, and by definition of $\varphi$, we can write $s=rs'=s''r$ with $s''\in A^*$.
Let us put
\[T'(s') = \{w\in \Delta_r^*\mid s'w'\in \Delta_r^*s'\}.\]
We show that under the identification of $T(r)$ with $\Delta_r^*$, we have that $T'(s')=T(s)$. If $w\in T'(s')$, then $s'w\in \Delta_r^*s'$ and  hence $sw=rs'w\in r\Delta_r^*s'\subseteq A^*rs'=A^*s$, as $\Delta_r^*=T(r)$.  Therefore, $w\in T(s)$.  Conversely, let $w\in T(s)$.  So $sw=w's$.  Thus $s$ seals $sw=rs'w$ and so, by Lemma~\ref{l:useful}, we have that $s'$ seals $s'w$ over $\Delta_r^*$, whence $w\in \Delta_r^*$ and $s'w\in \Delta_r^*s'$.  Thus $w\in T'(s')$.

 It now follows that the free basis $\Delta_s$ of $T(s)$, as a free submonoid of $A^*$, can be identified with the free basis of $T'(s')$ over $\Delta_r^*$.  Note that $ru'=u=su''=rs'u''$ and $rv'=v=sv''=rs'v''$ with $u'',v''\in T(s)=T'(s')$.  Thus $M$ compresses to $\lb \Lambda_s\mid u''=v''\rb$ under $s$.   Also note that  $u'=s'u''$ and $v'=s'v''$ with $s',u'',v''\in \Delta_r^*$ and hence $s',u'',v''\in \Lambda_r^*$.  Then under our identification of $T(s)$ with $T'(s')$ (and hence $\Delta_s$ with $\Delta_{s'}\subseteq \Delta_r^*$), we have that $u'=v'$ rewrites under compression by $\varphi(s) = s'\in \Lambda_r^*$ to $u''=v''$.

Hence the compressions of
$M$ with respect to $s$, and of $M'$ with respect to $\varphi(s)$, are the same.
This completes the proof of the proposition.
\end{proof}

The following corollary will be essential.

\begin{Cor}\label{c:get.to.lalle}
Let $M=\lb A\mid u=v\rb$ be a compressible one-relator monoid.  Suppose that $r$ compresses $M$ and let $M'=\lb A'\mid u'=v'\rb$ be the compression of $M'$ with respect to $r$.  Then $M'$ is incompressible if and only if $r$ is the unique maximal length word $z$ compressing $u=v$. Moreover, any iterated compression of $M$ that is incompressible is, up to relabelling, the compression of $M$ by $z$.
\end{Cor}

 We call the compression of $u=v$ with respect to the maximum length word compressing $u=v$ the \emph{Lallement} compression of $u=v$, as it was first considered by Lallement~\cite{Lallement1974} in the case of a subspecial presentation.  We call the compression of $u=v$ with respect to the unique SOF word compressing it the Adjan-Oganesyan compression of $u=v$, as this compression was first considered by Adjan and Oganesyan in~\cite{Adyan1987}. Corollary~\ref{c:get.to.lalle} implies that the Lallement compression of the Adjan-Oganesyan compression of $u=v$ is the Lallement compression of $u=v$.

Notice that a compressible one-relator presentation is subspecial if and only if its Lallement compression is special.  Thus Corollary~\ref{c:get.to.lalle} has the following consequence.

\begin{Cor}\label{c:stay.sub}
Let $M=\lb A\mid u=v\rb$ be a compressible one-relator monoid presentation and let $N$ be any compression of $M$.  Then $M$ is subspecial if and only if $N$ is subspecial.
\end{Cor}

Let us consider an example.
 Let $M=\lb a,b\mid ababa=ababbaba\rb$.
 The relation $ababa=ababbaba$ can be compressed by $y=a$ and $z=aba$.
 If we compress with respect to $y=a$ we obtain $M'=\lb
\beta, \gamma \mid
\beta^2=\beta\gamma\beta\rb$ where $\beta=ba$ and $\gamma = bba$.
 If we compress with respect to $z=aba$, then we obtain $M'' = \lb c,d\mid c=d \rb$ where $c=ba$ and $d=bbaba$.
Here it is easy to verify that we have $\Lambda_z = \{c,d \} = \{ ba, bbaba \}$.
Notice that $M'$ can be compressed with respect to $\beta$ and the resulting presentation is $\lb\rho,\tau\mid \rho =\tau\rb$ where $\rho = \beta$ and $\tau = \gamma\beta$, which is the same as the presentation of $M''$ up to relabelling.

Throughout the rest of this section we shall suppose that the word $r \in A^+$ compresses the defining relation $u=v$ in the presentation $\lb A \mid u=v \rb$ of the monoid $M$.
As usual, we write $[w]_M$ to denote the image in $M$ of $w \in A^*$.
Furthermore, with $L$ and $S$ defined as above
in equations \eqref{eqn:L} and \eqref{eqn:S},
we use $[w]_L$ for the image of $w\in \Delta_r^*$ in $L$, and $[w]_S$ for the image of $w\in \Lambda_r^*$ in $S$.
 Because it plays a very important role, we put $x=[r]_M$.
All of this notation will remain in force throughout the rest of this section.

Let us now generalise some lemmas from Lallement~\cite{Lallement1974} reducing the word problem of $M$ to its compression $S$.

For any word $w \in A^* r A^*$, by the \emph{first occurrence of $r$ in $w$} we mean the leftmost occurrence of the word $r$ in $w$, reading the word from left to right. Dually we talk about the \emph{last occurrence of the word $r$ in $w$}, which is the rightmost occurrence of $r$ in $w$.

If $s\in A^*$, then we can uniquely factor $s=y w$ where $y$ is the longest prefix of $s$ with $y \in A^*r\cup \{\varepsilon\}$ (i.e., either $s$ has no occurrence of $r$, in which case $y=\varepsilon$, or $y$ ends in the last occurrence of $r$ in $s$).  We call this the \emph{right canonical factorization} of $s$.

The following generalises the dual of~\cite[Lemma~3.1]{Lallement1974}.

\begin{Lemma}\label{lem:lalle3.1}
Let $s,s'\in A^*$ with right canonical factorizations $s=y w$ and $s'=y'w'$, respectively.  Then $[s]_M=[s']_M$ if and only if $w =w'$ and $[y]_M=[y']_M$.
\end{Lemma}
\begin{proof}
Trivially, if $w=w'$ and $[y]_M=[y']_M$, then $[s]_M=[s']_M$.
For the converse, assume that $[s]_M=[s']_M$.
 If $s=s'$, then clearly $w=w'$ and $y=y'$, whence $[y]_M=[y']_M$.
 Thus we may assume that $s\neq s'$.
 In particular, it must be possibly to apply the defining relation to both $s$ and $s'$.
 Since $u,v\in A^*r$, it follows that $r$ occurs as a factor of both $s$ and $s'$ and so $y\neq \varepsilon$ and $y'\neq \varepsilon$.
It clearly suffices to prove the result when $s'$ may be obtained from $s$ by just a single application of the defining relation $u=v$ in the presentation defining $M$.
Thus, by symmetry, it will suffice to show that if $s=x_1ux_2$ has right canonical factorization $s=yw$ with $y\in A^*r$ and if $s'=x_1vx_2$, then $s'$ has right canonical factorization of the form $y'w$ with $y'\in A^*r$ and $[y']_M=[y]_M$.
 Write $y=y_0r$, $u=u_0r$ and $v=v_0r$.
 Then $s=x_1u_0rx_2 = y_0rw$ where the $r$ displayed on the right hand side is the last $r$ in $s$.
 It follows that $u_0$ is a factor of $y_0$ and $rw$ is a suffix of $rx_2$, and so $s'=x_1vx_2=x_1v_0rx_2=y_0'rw$,
for some $y_0' \in A^*$,
and so the displayed $r$ in the right hand side is the last occurrence of $r$ in $s'$ as $rw$ contains no other occurrence of $r$ as $yw$ is a right canonical factorization of $s$.
 Thus, putting $y'=y_0'r$, we have that $s'$ has right canonical factorization $y'w$ and $[y']_M=[y]_M$ as $u=u_0r$ was a factor of $y_0r=y$ and $y'$ is obtained by replacing $u$ in $y$ by $v$.
\end{proof}

\begin{remark}\label{r:last.part}
It follows from Lemma~\ref{lem:lalle3.1} that if $w\in A^*r$, then every word equivalent to $w$ in $M$ belongs to $A^*r$.
\end{remark}

Lemma~\ref{lem:lalle3.1} reduces the word problem for $M$ to consideration of words in $A^*r$, which leads to the second lemma that we shall need, generalizing the dual of~\cite[Lemma~3.2]{Lallement1974}.  Suppose that $s\in A^*r$.  Then $s$ has a unique factorization $s=y rw$ with the displayed $r$ the first occurrence of $r$ in $s$ (it has at least one since it ends in $r$).  Note that $w \in T(r)$, since $s,yr\in A^*r\subseteq T(r)$ and $T(r)$ is left unitary. Thus $w \in \Delta^*_r$ and so $[w]_L$ makes sense. We call $s=y rw$ the \emph{left canonical factorization} of $s$.

\begin{Lemma}\label{lem:lalle3.2}
Let $s,s'\in A^*r$ with respective left canonical factorizations $s=y rw$ and $s'=y'rw'$.  Then $[s]_M=[s']_M$ if and only if $y=y'$ and $[w]_L=[w']_L$.
\end{Lemma}
\begin{proof}
By Remark~\ref{r:last.part} every word equivalent in $M$ to $s$ or $s'$ belongs to $A^*r$.
It then clearly suffices to show that $s'$ can be obtained from $s$ via one application of the defining relation
$u=v$
if and only if $y=y'$ and $w'$ can be obtained from $w$ by one application of the defining relation for $L$.
Assume first that $s$ is obtained from $s'$ via one application of the defining relation $u=v$ in the presentation of $M$.
Up to symmetry we may assume that $yrw=x_1ux_2$ and $y'rw'=x_1vx_2$.  Since $u,v\in rA^*$ and the left canonical factorization displays the first occurrence of $r$, we must have that $y=y'$ and that $rw=x_1'ux_2$ and $rw'=x_1'vx_2$ where $x_1=yx_1'$.  Recall that $w,w'\in T(r)$ and $u=ru'$, $v=rv'$.  From $rw=x_1'ru'x_2$ and $rw'=x_1'rv'x_2$, we deduce that $r$ is a prefix of $x_1'r$.
Therefore,
writing $x_1' r= rx'$, substituting into the previous two equations, and deleting the $r$-prefixes, we see that
we can find $x'\in A^*$ with $w=x'u'x_2$ and $w'=x'v'x_2$ and $rx'=x_1'r$, whence $x'\in T(r)$.  Since $w,x',u'\in T(r)$, we deduce that $x_2\in T(r)$ as $T(r)$ is left unitary.  Thus we have factorizations $w=x'u'x_2$ and $w'=x'v'x_2$ with $x',x_2\in \Delta_r^*$ and so $[w]_L=[w']_L$ via one application of the defining relation
$u'=v'$ in the presentation of
$L$.

Conversely, suppose that $y=y'$ and $w'$ can be obtained from $w$ by one application of the defining relation
$u'=v'$ in the presentation of
$L$.  Without loss of generality, we may assume that $w=x'u'x_2$ and $w'=x'v'x_2$ with $x',x_2\in \Delta_r^*$.  Write $rx' = x''r$.  Then $s=yrw=yrx'u'x_2= yx''ru'x_2=yx''ux_2$ and $s'=y'rw'=yrx'v'x_2=yx''rv'x_2=yx''vx_2$ and so $s'$ can be derived from $s$ via one application of the defining relation
$u=v$ in the presentation
of $M$.  This completes the proof.
\end{proof}

Notice that these two lemmas reduce the word problem from $M$ to $S$ (as $L$ is a free product of $S$ with a free monoid on a recursive set).
In~\cite{Lallement1974} Lallement proved the above two lemmas in the particular case that $r=v$, where $u=v$ is a subspecial relation.
In this situation $S$ is a speical one-relator monoid.
Since
Adjan~\cite{Adjan1966} had previously proved that word problem is decidable for special one-relator monoids, this allowed Lallement to conclude that subspecial one-relator monoids also have decidable word problem.

Let $P_{\Delta_r}$ be the collection of proper prefixes of elements of $\Delta_r$.   Because $\Delta_r$ is a prefix code, no element of $\Delta_r$ belongs to $P_{\Delta_r}$.
Let us provide an alternate description of $P_{\Delta_r}$.

\begin{Lemma}\label{l:char.Delta.r}
Let $w\in A^*$.  Then $w\in P_{\Delta_r}$ if and only if no non-empty prefix of $w$ belongs to $T(r)$.
\end{Lemma}
\begin{proof}
Suppose first that $w\in P_{\Delta_r}$ has a non-empty prefix $w'$ with $w'\in T(r)$.  Then $w'\in \Delta_r^+$ and so $w'$ has a prefix $w''$ belonging to $\Delta_r$.   But then $w''\in P_{\Delta_r}\cap \Delta_r=\emptyset$, a contradiction.  Thus no non-empty prefix of $w$ belongs to $T(r)$.

Conversely, suppose that no non-empty prefix of $w$ belongs to $T(r)$.   Consider $wr$. Trivially, $wr\in A^*r\subseteq T(r)$ and so $wr\in \Delta_r^*$.  Write $wr=\alpha z$ with $\alpha\in \Delta_r$ and $z\in \Delta_r^*$.  Then since no non-empty prefix of $w$ is in $T(r)$, we must have that $w$ is a proper prefix of $\alpha$.  Thus $w \in P_{\Delta_r}$.
\end{proof}

The following two lemmas will be used
below
when we analyse compression from a topological perspective.

\begin{Lemma}\label{l:last.occurrence}
Suppose that $z\in A^*$ with $z=t r w$.  Then this factorization is right canonical if and only if $w \in P_{\Delta_r}$.
\end{Lemma}
\begin{proof}
Suppose first that the factorization is right canonical.  Then the displayed occurrence of $r$ is the last one in $z$ and so  no non-empty prefix of $w$ belongs to $T(r)$.  Therefore, $w\in P_{\Delta_r}$ by Lemma~\ref{l:char.Delta.r}.  Conversely, if $w\in P_{\Delta_r}$ and if $z=tr w=t'r w'$ with $tr$ a proper prefix of $t'r$, then writing $t'r=trw_0$, we see that $r$ is a suffix of $rw_0$ and so $w_0\in T(r)$.  But $w_0$ is a non-empty prefix of $w$.  This contradicts Lemma~\ref{l:char.Delta.r} and so we conclude that  the factorization $z=tr w$ is right canonical.
\end{proof}

We shall also need the following lemma.

\begin{Lemma}\label{l:embed.prefix}
Let $t,t'\in A^*$ and $w ,w'\in P_{\Delta_r}$. Then $[t rw]_M=[t'rw']_M$ if and only if $w =w'$ and $[t r]_M=[t'r]_M$.
\end{Lemma}
\begin{proof}
By Lemma~\ref{l:last.occurrence}, we have that $t r w$ and $t'r w'$ are right canonical factorizations and hence if $[t rw]_M=[t'vw']_M$, then $w =w'$ and $[t r]_M=[t'r]_M$ by Lemma~\ref{lem:lalle3.1}.
\end{proof}

\subsection*{The compressed monoid as a local divisor}

If $n\in M$, then there is a well-defined mapping $\circ \colon Mn\times nM\to MnM$ given by $m\circ m' = anb$ for $m=an$ and $m'=nb$.
 This mapping restricts to an associative multiplication on the  \emph{local divisor} $M_n=nM\cap Mn$  with identity element $n$ (if $an=na'$ and $nb=b'n$ then $na'b=anb=ab'n$).
So, by the \emph{local divisor} $M_n$ we mean the monoid $(nM \cap Mn, \circ)$ where is the multiplication $\circ$ is defined as above.
Local divisors were introduced by Diekert and studied further
in~\cite{Diekert2006, Diekert2016, CS15}.
It is easily checked that  by restricting $\circ$ we obtain a left action $M_n\times nM\to nM$ of $M_n$ on $nM$ extending the regular action, 
and dually a right action $Mn\times M_n\to Mn$ of $M_n$ on $Mn$. 
We aim to identify $M_x$
with $L$.
Recall here that $x = [r]_M$ where $r$ is a word compressing the defining relation $u=v$ of the monoid $M$.
This will provide an algebraic interpretation of compression that was lacking in previous work
of others on this topic.

\begin{Prop}\label{p:define.map.to.divisor}
Define $\varphi\colon T(r)\to M_x$ by $\varphi(w )=[rw]_M$.  Then $\varphi$ is a surjective homomorphism factoring through $L$ as an isomorphism $\Phi\colon L\to M_x$.
\end{Prop}
\begin{proof}
First note that if $w \in T(r)$, then $rw \in rA^*\cap A^*r$ and so $[rw ]_M\in M_x$.   Let $w_1,w_2\in T(r)$.  Then $rw_1=w_1'r$ for some $w_1'\in A^*$.  We compute $\varphi(w_1)\varphi(w_2) = [rw_1]_M\circ [rw_2]_M = [w_1'r]_M\circ [rw_2]_M = [w_1'rw_2]_M=[rw_1w_2]_M=\varphi(w_1w_2)$.  Also $\varphi(\varepsilon) =[r]_M=x$.  Thus $\varphi$ is a homomorphism.

Let $m\in M_x=xM\cap Mx$.
Then there exist $y,w \in A^*$ with $m=[ry]_M=[w r]_M$. Then  $ry\in A^*r$ by Remark~\ref{r:last.part}.  Thus $y\in T(r)$ and $m=[ry]_M=\varphi(y)$. Therefore, $\varphi$ is onto.

Note that $\varphi(u') = [ru']_M =[u]_M=[v]_M=[rv']_M=\varphi(v')$ and so $\varphi$ factors through $L$.  Suppose that $\varphi(s)=\varphi(s')$ with $s,s'\in T(r)$.  Then $[rs]_M=[rs']_M$.  Now $rs,rs'\in A^*r$ and obviously their left canonical factorizations are $rs,rs'$, respectively.  Thus $[s]_L=[s']_L$ by Lemma~\ref{lem:lalle3.2}.  This establishes that $\Phi$ is an isomorphism.
\end{proof}

Since $M_x$ acts on the left of $xM$ and the right of $Mx$, it follows that $L$ does as well, where we retain the above notation.  We now describe that action.

\begin{Prop}\label{p:action.of.L}
If $y\in \Delta_r^*$ and $m\in M$, then $\Phi([y]_L)\circ xm = [ry]_Mm=x[y]_Mm$ and $mx\circ \Phi([y]_L) = mx[y]_M$.
\end{Prop}
\begin{proof}
We have that $ry=y'r$ for some $y'\in A^*$.  Thus $\Phi([y]_L)\circ xm = [ry]_M\circ xm=[y'r]_M\circ xm = [y'r]_Mm=[ry]_Mm$.  On the other hand, $mx\circ \Phi([y]_L) = mx\circ [ry]_M = mx[y]_M$.  This completes the proof.
\end{proof}

Now we show that $Mx$ is a free right $M_x$-set and $xM$ is a free left $M_x$-set.  Let $\mathcal B$ consists of all $y\in A^*$ such that $y r$ contains no occurrence of $r$ except as a suffix.
Note that in particular if $r$ is a SOF word then
$\mathcal B = A^*\setminus A^*rA^*$.

\begin{Prop}\label{p:free.over.local}
The set $Mx$ is a free right $M_x$-set with basis $B$ the set of elements $[y r]_M$ with $y\in \mathcal B$.  Moreover, if $y,y'\in \mathcal B$, then $[y r]_M=[y'r]_M$ if and only if $y=y'$.
\end{Prop}
\begin{proof}
The final statement follows from Lemma~\ref{lem:lalle3.2} because $y r\varepsilon$ and $y'r\varepsilon$ are left canonical factorizations of $y r$ and $y'r$.
For the other part, first note that if $y,y'\in \mathcal B$ and $w ,w'\in T(r)$ with $[y r]_M\circ \Phi([w ]_L) = [y'r]_M\circ \Phi([w']_L)$, then $[y rw ]_M=[y'rw']_M$ by Proposition~\ref{p:action.of.L}.  Since $w ,w'\in T(r)$, we have $y rw,y'rw'\in A^*r$ and, by definition of $\mathcal B$, clearly $y rw $ and $y'rw'$ are left canonical factorizations.  Thus $y=y'$ and $[w ]_L=[w']_L$ by Proposition~\ref{lem:lalle3.2}.  Therefore, $B=\{[y r]_M\mid y\in \mathcal B\}$ is a basis for a free sub-$M_x$-set of $Mx$.

Let $m\in Mx$; say $m=[zr]_M$.  Let the left canonical factorization of $zr$ be $y rw $.  Then $y\in \mathcal B$ and $w \in T(r)$.  Then we have, by Proposition~\ref{p:action.of.L}, that $[y r]_M\circ \Phi([w ]_L) = [y rw ]_M = [zr]_M=m$.  This completes the proof that $B$ is a basis for $Mx$.
\end{proof}

Of course the dual of Proposition~\ref{p:free.over.local} holds: the left action of $M_x$ on $xM$ is free.

We next want to check that $Mx$ is a free right $S$-set, where $S\leq L$ acts via $\Phi|_S$ and the action of $M_x$.  For this, the following elementary observation will be useful.

\begin{Lemma}\label{l:free.iterate}
Let $N$ be a monoid and $T\leq N$ a submonoid such that $N$ is a free right $T$-set
with basis $B$,
under the right multiplicative action of $T$ on $N$.
Let $X$ be a free right $N$-set with basis $B'$.  Then $X$ is a free right $T$-set with basis $B'B$.
\end{Lemma}
\begin{proof}
If $y\in X$, then $y=b'n$ with $b'\in B'$ and $n\in N$.  But $n=bt$ with $b\in B$ and $t\in T$.  Thus $y=b'bt$.  If $b_1'b_1t_1=b_2'b_2t_2$ with $b_i'\in B'$, $b_i\in B$ and $t_i\in T$, for $i=1,2$, then $b_1'=b_2'$ and $b_1t_1=b_2t_2$ by freeness of the $N$-action.  But then $b_1=b_2$ and $t_1=t_2$ by freeness of the $T$-action.  This completes the proof.
\end{proof}

\begin{Cor}\label{c:free.action}
The compressed one-relator monoid $S$ acts freely on the right of $Mx$ via $(m,[q]_S)\mapsto m[q]_M$,
for $q \in \Lambda_r^*$,
with basis $C$ the elements of the form $[t rw]_M$ with $t\in \mathcal B$ and $w \in \Delta_r^*$ with $w =\varepsilon$ or the last $\Delta_r$-letter of $w $ belonging to $\Psi_r$. Moreover, $[t_1 r w_1]_M = [t_2 r w_2]_M$ for $t_1,t_2\in \mathcal B$ and $w_1,w_2\in \Delta_r^*\Psi_r\cup \{\varepsilon\}$ if and only if $t_1=t_2$ and $[w_1]_L=[w_2]_L$.
\end{Cor}
\begin{proof}
By Proposition~\ref{p:free.over.local} and Proposition~\ref{p:action.of.L}, we have that $L$ acts freely on the right of $Mx$ via $(m,[q]_L)\mapsto m[q]_M$ with basis elements of the form $[t r]_M$ with $t\in \mathcal B$ (and these are distinct).  But $L=\Psi_r^*\ast S$ and hence is a free right $S$-set on $1$ and those elements of $L$ ending in a $\Psi_r^*$-syllable in their free product normal form.  The freeness of the action now follows from Lemma~\ref{l:free.iterate}.  The final statement follows from Lemma~\ref{lem:lalle3.2}.
\end{proof}

Of course, $S$ acts freely on the left of $xM$ by the dual of Corollary~\ref{c:free.action}.

\subsection*{Compression from a topological point of view}
We now interpret compression from a topological viewpoint.  We retain the previous notation of this section. Let $\Gamma$ be the Cayley graph of $M$ with respect to $A$.  Notice that if $m\in Mx$, say $m=nx$, then $m[u']_M=n[ru']_M=n[rv']_M=m[v']_M$.  Let $p_{m,u'}$ and $p_{m,v'}$ be the paths starting at $m\in Mx$ labelled by $u'$ and $v'$, respectively; they both end at the same vertex by the previous discussion.  Let $K$ be the result of attaching $2$-cells $c_m$ to $\Gamma$, for each $m\in Mx$,  with boundary path $p_{m,u'}p_{m,v'}\inv$.   Notice that $M$ acts on the left of $K$ by cellular maps, extending the action on $\Gamma$, with $nc_m = c_{nm}$ for $n\in M$ and  $m\in Mx$ and that $K$ is, in fact, obtained from $\Gamma$ by adjoining a not necessarily projective $M$-cell $Mx\times B^2$.  In particular, the $M$-CW complex $K$ need not be free, nor projective.
By Proposition~\ref{prop:chains}, the group of $2$-chains $C_2(K)$ of $K$ is isomorphic to $\mathbb ZMx$ as a left $\mathbb ZM$-module.

Let $\Gamma'$ be the graph with vertex set $\{\ast\}\cup (B\times L)$ with $B$ as in Proposition~\ref{p:free.over.local}.  The edges of $\Gamma'$ are as follows.  There is an edge from $\ast$ to $(b,1)$ labelled by $b$, for each $b\in B$.  Also, for each $\alpha\in \Delta_r$ and $(b,n)\in B\times L$, there is an edge labelled by $\alpha$ from $(b,n)$ to $(b,n[\alpha]_L)$.  So $\Gamma'$ consists of a distinguished vertex $\ast$, which we have connected by edges to $|B|$ disjoint copies of the Cayley graph of $L$ with respect to $\Delta_r$.

\begin{remark}
\label{rmk:before:thm}
Alternatively, the vertices of $\Gamma'$ can be identified with $\{\ast\}\cup Mx$ via the mapping $(b,[w]_L)\mapsto b[w]_M$ by Proposition~\ref{p:free.over.local}.  From this point of view, there is an edge from $\ast$ to $b$ labelled by each element of $B$ and an edge from $m$ to $m[\alpha]_M$ labelled by $\alpha$ for each $m\in Mx$ and $\alpha\in \Delta_r$.
\end{remark}

Let $q_{(b,n),u'}$ be the
unique
path
in $\Gamma'$
starting at $(b,n)$ labelled by $u'\in \Delta_r^*$ and $q_{(b,n),v'}$ be the
unique
path in $\Gamma'$
starting at $(b,n)$
labelled by $v'$ over the alphabet $\Delta_r$. These paths are coterminal since $L$ satisfies the relation $u'=v'$, so we may attach a $2$-cell $c_{(b,n)}$ with boundary path $q_{(b,n),u'}q_{(b,n),v'}\inv$ for each $(b,n)\in B\times L$.
We denote the resulting $2$-complex by $K'$.
Thus $K'$ consists of a distinguished vertex $\ast$ attached by edges to $|B|$ disjoint copies of the Cayley complex of $L$ with respect to its presentation $\lb \Delta_r\mid u'=v'\rb$.  We aim to show that there is a forest $F\subseteq \Gamma$ such that if we contract each component of $F$ to a point to form $\Gamma/F$ and $K/F$, then
$\Gamma/F\cong \Gamma'$ and $K/F\cong K'$.

\begin{Thm}\label{t:compress.cayley.graph.general}
Let $\Gamma$ be the Cayley graph of the one-relator monoid $M=\lb A\mid u=v\rb$.  Let $r\in A^+$ compress $u=v$ and $L=\lb \Delta_r\mid u'=v'\rb$ where $u=ru'$ and $v=rv'$. Let $\Gamma'$ consist of $|B|$ copies of the Cayley graph of $L$ together with a new vertex attached by an edge to each identity vertex and $K'$ consist of $|B|$ copies of the Cayley complex of $L$ together with a new vertex attached by an edge to each identity vertex, where $B$ is as in  Proposition~\ref{p:free.over.local}.
Then there is a forest $F\subseteq \Gamma$ such that $\Gamma/F\cong \Gamma'$ and
$K/F\cong K'$. 
Denote the composition of the projection $\Gamma\to \Gamma/F$ with the isomorphism $\Gamma/F\to \Gamma'$ by $\psi$ (and similarly for the extension $K\to K'$).  Then $\psi$ enjoys the following properties (where we retain the above notation).
\begin{enumerate}
\item $\psi(M\setminus MxM) = \{\ast\}$ where $x=[r]_M$.
\item   If $w\in A^*rA^*$ has right canonical factorization $w=y z$ with $y\in A^*r$ and $y$ has left canonical factorization $y=w_1rw_2$, then $\psi([w]_M)=([w_1r]_M,[w_2]_L)$.
\item An edge of $\Gamma$ is sent to the same point as its initial vertex unless either
(i) it is an edge from an element of $M\setminus MxM$ to an element $b\in B$, in which case it is mapped to the edge $\ast\to (b,1)$ labelled by $b$, or (ii) it is an edge $m[w]_M\xrightarrow{\,\,a\,\,} m[wa]_M$ with $m\in Mx$ and $wa\in \Delta_r$, in which case it is mapped to the edge labelled by $\alpha=wa\in \Delta_r$ from $\psi(m[w]_M)=\psi(m)$ to $\psi(m[wa]_M)$.
\item A $2$-cell $c_m$ of $K$ with $m\in Mx$ is sent to $c_{\psi(m)}$. 
\item  $\psi$ is injective on the subset of vertices belonging to $Mx$ and bijective on the set of $2$-cells. 
\end{enumerate}
In particular,
the $M$-CW complex
$K$ is homotopy equivalent to $K'$.
\end{Thm}
\begin{proof}
As before, denote by $P_{\Delta_r}$ the collection of proper prefixes of elements of $\Delta_r$.  This is a prefix closed subset of $A^*$ and hence induces a subtree $T_{\Delta_r}$ of the Cayley graph of $A^*$
with respect to generating set $A$.
If $m\in Mx$, then there is a graph morphism $\rho_m\colon T_{\Delta_r}\to \Gamma$ sending a vertex $t$ to $m[t]_M$ and an edge $t\xrightarrow{\,\,a\,\,}t a$ to the edge $m[t]_M\xrightarrow{\,\,a\,\,}m[ta]_M$.

\begin{claim}
For each $m\in Mx$, the graph morphism $\rho_m\colon T_{\Delta_r}\to \Gamma$ is injective.
Furthermore, if we put $T_m=\rho_m(T_{\Delta_r})$, then $T_m\cap T_{m'}=\emptyset$ unless $m=m'$.
\end{claim}
\begin{proof}[Proof of Claim]
It is enough to show that $\rho_m$ is injective on vertices since $T_{\Delta_r}$ is a tree.  Write $m=[tr]_M$ with $t\in A^*$ and let $w ,w'\in P_{\Delta_r}$.  Suppose that $m[w]_M=m[w']_M$.  Then $[trw]_M=[trw']_M$ and so $w =w'$ by Lemma~\ref{l:embed.prefix}.  Thus $\rho_m$ is injective.  Next suppose that $T_m\cap T_{m'}\neq \emptyset$ with $m,m'\in X$.  Then we can find $t,t'\in A^*$ and $w ,w'\in P_{\Delta_r}$ with $[t r]_M=m$, $[t'r]_M=m'$ and $m[w]_M=m'[w']_M$. Then $[t rw]_M=[t'rw']_M$ and so $m=[t r]_M=[t'r]_M=m'$ by Lemma~\ref{l:embed.prefix}.  This establishes the claim.
\end{proof}

There is another subtree of $\Gamma$ that we shall need, disjoint from the previous ones.  Let $Q= A^*\setminus A^*rA^*$.  Then $Q$ is prefix closed and hence induces a subtree $T_Q$ for the Cayley graph of $A^*$.  Elements of $Q$ are obviously in singleton classes for the congruence associated to $M$,
since both the words $u$ and $v$ in the defining relation $u=v$ in the presentation for $M$ contain $r$ as a subword.  These congruence classes make up precisely the elements of $M\setminus MxM$.   Thus there is a natural embedding $T_Q\to \Gamma$ sending $t$ to $[t]_M$ and $t\xrightarrow{\,\,a\,\,}t a$ to $[t]_M\xrightarrow{\,\,a\,\,}[ta]_M$.  We denote the image tree by $T'\subseteq \Gamma$.  Notice that $T'$ is disjoint from $\coprod_{m\in Mx} T_m$ by construction as the vertices of $T_m$ are contained in $MxM$ and $T'$ is disjoint from $MxM$.  Thus $F=T'\amalg \coprod_{m\in Mx}T_m$ is a forest in $\Gamma$ (and hence a subcomplex of $K$) and so we can contract each of the component trees to a point (we prefer to think of the trees as being contracted to their roots) without changing the homotopy type of $\Gamma$ (respectively $K$).

The first thing to observe is that contracting these trees identifies no two vertices
from $Mx$ (since the various trees of $F$ are disjoint and $m\in T_m$ for $m\in Mx$). All vertices in $M\setminus MxM$ are identified with the vertex $1$, let us call this class $\ast$.  We claim that all other vertices of $\Gamma$ get identified with a unique element of $Mx$.  Indeed, if $m\in MxM$, we can write $m=[t rw]_M$ and without loss of generality we can assume that the factorization exhibits the last occurrence of $r$. Then $w \in P_{\Delta_r}$ by Lemma~\ref{l:last.occurrence}.  Therefore, $m$ belongs to the tree $T_{[t r]_M}$ and hence gets identified with $[t r]_M$ in the quotient.  Thus $\Gamma/F$ has vertices (the classes of elements of) $Mx$, all of which are distinct, together with an additional vertex $\ast$ to which $T'$ was contracted.  Note that $Mx$ can be identified with $B\times L$ by Proposition~\ref{p:free.over.local} by sending $[w]_M$ to $([yr]_M,[z]_L)$ where $w\in A^*r$ has left canonical factorization $w=yrz$ and this identification is one of right $L$-sets (see the last paragraph of the proof of Proposition~\ref{p:free.over.local}). This allows us to define $\psi$ on vertices and establishes (1) and (2).

Let us next see what happens to the edges of $\Gamma$.  We claim that each edge of $\Gamma$ belongs to some $T_m$ or $T'$ except edges of the form $m[w ]\xrightarrow{\,\,a\,\,}m[w a]$ with $m\in Mx$ and $w a\in \Delta_r$, which we label in the quotient by $\alpha=w a\in \Delta_r$, or edges of the form $[t]\xrightarrow{\,\,a\,\,}[t a]$ where $t \notin A^*rA^*$ and $t a\in A^*r$, which we label in the quotient by $[ta]_M\in B$.
So the edges of $\Gamma/F$ will be labelled over the infinite alphabet $\Delta_r\cup B$.

Indeed, first note that these edges do not belong to any of the trees and so survive contracting the trees.  For instance, for an edge of the form $m[w ]\xrightarrow{\,\,a\,\,}m[w a]_M$ with $m\in Mx$ and $w a\in \Delta_r$, the geodesic in $T_m$ from $m$ to $m[w]_M$ is labelled by $w\in P_{\Delta_r}$ and there is no edge in $T_{\Delta_r}$ from $w $ labelled by $a$.  Note that in the quotient $\Gamma/F$, this edge goes from (the class of) $m$ to (the class of) $m[wa]_M$ and is labelled by $wa\in \Delta_r$. Similarly, if we have an edge of the form $[t]_M\xrightarrow{\,\,a\,\,}[t a]_M$ where $t\notin A^*rA^*$ and $t a\in A^*r$, then the geodesic in $T'$ from $1$ to $[t]_M$ is labelled by $t$ and there is no edge in $T_Q$ from $t$ with label $a$.    Moreover, $t a\in \mathcal Br$ since $t$ does not have $r$ as a factor and so $[ta]_M\in B$. In $\Gamma/F$ this edge goes from $\ast$ to (the class of) $[ta]_M\in B$ and is labelled by $[ta]_M$.  Under $\psi$ it maps to $([ta]_M,1)$.

  The remaining edges of $\Gamma$ are either of the form $[t]_M\xrightarrow{\,\,a\,\,}[t a]_M$ with $t,t a\notin A^*rA^*$, and hence are contracted to $1$, i.e., the class $\ast$, or of the form $m\xrightarrow{\,\,a\,\,}m[a]_M$ with $m\in MxM$. Writing $m=[t rw]_M$ with the displayed $r$ the last occurrence, we have that $w \in P_{\Delta_r}$ by Lemma~\ref{l:last.occurrence}.  If $w a\in P_{\Delta_r}$, then the edge    $m\xrightarrow{\,\,a\,\,}m[a]_M$  belongs to $T_{[t r]_M}$ and is contracted to $[tr]_M$.  Otherwise, $wa$ must have a non-empty prefix in $T(r)$ by Lemma~\ref{l:char.Delta.r}. But no proper non-empty prefix of $wa$ belongs to $T(r)$ since $w\in P_{\Delta_r}$ (again by Lemma~\ref{l:char.Delta.r}).  Thus we conclude that $wa\in T(r)$ and is irreducible in $T(r)$, that is,   $w a\in \Delta_r$.  Therefore, our edge is of the form $[t r]_M[w ]_M\xrightarrow{\,\,a\,\,}[t r]_M[w a]_M$ with $[t r]_M\in Mx$ and $w a\in \Delta_r$, as desired.

 We have thus far established items (1)--(3).

Let $u'=\alpha_1\cdots \alpha_k$ and $v'=\beta_1\cdots \beta_{\ell}$ with the $\alpha_i,\beta_j\in \Delta_r$ and note that each proper prefix of $\alpha_i,\beta_j$ belongs to $P_{\Delta_r}$. If $m\in Mx$, then each edge of $p_{m,u'}$ is contracted except for the last edge of the subpath labelled by $\alpha_i$ from $m[\alpha_1\cdots\alpha_{i-1}]_M$ to $m[\alpha_1\cdots\alpha_i]_M$, and this edge is labelled by $\alpha_i$ in the quotient and goes from (the class of) $m[\alpha_1\cdots \alpha_{i-1}]_M$ to (the class of) $m[\alpha_1\cdots \alpha_i]_M$. Thus the path $p_{m,u'}$ is mapped under $\psi$ to a reparameterization of $q_{\psi(m),u'}$.    Similarly, $p_{m,v'}$ is mapped under $\psi$ to a reparameterization of $q_{\psi(m),v'}$.  Thus the $2$-cell $c_m$ maps to the cell $c_{\psi(m)}$ on the level of sets, but due to the contraction of edges, the attaching map for $c_{\psi(m)}$ in $K'$ is a reparameterization of the attaching map for the image of $c_m$ in $K/F$.  Since $\psi$ is injective on $Mx$, we have the $\psi$ is injective on the set of $2$-cells.  On the other hand, $c_{(b,[w]_L)} = \psi(c_{b[w]_M})$ for $b\in B$ and $w\in \Delta_r^*$ and so $\psi$ is also surjective on the set of $2$-cells.  This completes the proof.
\end{proof}

Note that
in the above theorem
$K$ is homotopy equivalent to $K'$, which is homotopy equivalent to a wedge of copies of the Cayley complex of $L$.

Theorem~\ref{t:compress.cayley.graph.general} forms the topological underpinnings of our approach, but to proceed further we need to distinguish the subspecial and the non-subspecial cases.

\section{The case of subspecial monoids}\label{sec_special}

Let $M=\lb A\mid u=v\rb$ be a finite one-relator presentation such that the relation $u=v$ is subspecial with $|v|<|u|$.
This means we can write
$u=u''v=vu'$ with $u'',u'\in A^+$.
 Since the case of special monoids, when $v$ is empty, was handled in~\cite[Section~3]{GraySteinberg2}, we shall tacitly assume that $v$ is non-empty throughout the rest of this section.  Also, this notation shall be fixed for the section.

\subsection*{Preliminaries on subspecial monoids}

As already explained in Section~\ref{s:compression} above, the word problem for subspecial one-relator monoids was solved by Lallement~\cite{Lallement1974} via a reduction to the case of special one-relator monoids, which was solved by Adjan~\cite{Adjan1966}. Note that $v$ is the maximum length word compressing $u=v$ and so the corresponding compression is the Lallement compression of $M$.  Put $\Delta=\Delta_v$, $\Lambda=\Lambda_v$ and $\Psi=\Psi_v$, where we retain the notation of
Section~\ref{s:compression}.
If $u'=\alpha_1\cdots \alpha_k$ with $\alpha_i\in \Delta$, then we have that
\[L=\lb \Delta\mid \alpha_1\cdots \alpha_k=1\rb,\] $\Lambda=\{\alpha_1,\ldots,\alpha_k\}$, $\Psi=\Delta\setminus \Lambda$ and the Lallement compression of $M$ is
\[
S=\lb \Lambda\mid \alpha_1\cdots \alpha_k=1\rb
\]
which is a special one-relator monoid.  To decongest notation
throughout this section
we write $[w]$, rather than $[w]_M$, to denote the image in $M$ of $w \in A^*$.
However, we will continue to use the notation $[w]_S$ and $[w]_L$ for the images in $S$ and $L$ of words over the alphabets $\Lambda$ and $\Delta$, respectively.
Following our previous notation (where $r=v$), we  put $x=[v]$.

We recall that an element $m \in M$ is said to be (von Neumann) \emph{regular} if there exists an element $n \in M$ such that $mnm=m$.
We shall need later that $x$ is a regular element of $M$.

\begin{Prop}\label{p:is.regular}
The element $x\in M$ is regular. Thus $Mx$ is a projective left $M$-set.
\end{Prop}
\begin{proof}
We have $[v]=[u]=[vu']$ and so $[v]=[v][u']^n$ for all $n\geq 0$.  Choose $k>|v|$.  Since $u'\in T(v)$, we have that $(u')^k\in T(v)$ and so $v(u')^k\in A^*v$.  But since $k>|v|$, this means the suffix $v$ of $v(u')^k$ is a proper suffix of $(u')^k$.  So $v(u')^k = vy v$ with $y\in A^+$.  Thus $x=[v]=[v][u']^k=[vy v] =x[y]x$ and so $x$ is regular.  The final statement follows because $Mx=M[y]x$ and $[y]x$ is an idempotent.
\end{proof}

Note that it follows, cf.~\cite{CS15}, that $M_x\cong [y]xM[y]x$ (with $y$ as in the proof above) and hence the group of units of $S$ (which equals the group of units of $L$) is a maximal subgroup of $M$.
In fact, it is a one-relator group and has torsion if and only if $u'$ is a proper power, as the following lemma shows.

\begin{lemma}\label{lem:max:subgroups}
Let $M=\lb A\mid u=v\rb$ be a subspecial, but not special, one-relator monoid with $u = u'' v = v u'$ where $u'', u' \in A^+$
and let $S$ be the associated finitely generated special one-relator monoid obtained by compressing $M$ with respect to $v$.
Let $G$ be the group of units of $S$.
Then, for every non-identity idempotent of $M$, the maximal subgroup of $M$ containing that idempotent is isomorphic to $G$.
Moreover, the following are equivalent:
\begin{enumerate}
\item the word $u'$ is a proper power in $A^+$;
\item the group of units $G$ of  $S=\lb \Lambda\mid \alpha_1\cdots \alpha_k=1\rb$ is a one-relator group with torsion;
\item all maximal subgroups of $M$ are one-relator groups with torsion (except the group of units, which is trivial).
\end{enumerate}
\end{lemma}
\begin{proof}
Since $u,v$ are non-empty, the group of units of $M$ is clearly trivial.
We begin by proving that all other maximal subgroups of $M$ are isomorphic to $G$.
To prove this we use an argument which is similar to~\cite[Corollary~3.11]{Lallement1974}.
Once proved, the equivalence of (2) and (3) will then immediately follow.
In this proof we will make use of Green's relations  $\gl$, $\gr$, and $\mathscr{D}$ on a semigroup $S$.
These are natural equivalence relations defined on a semigroup defined in terms of the ideal structure.
See~\cite[Chapter~2]{Howie} for definitions of these relations.

By Proposition~\ref{p:is.regular}, we have that $x$ is regular and so $xyx=x$ for some $y\in M$.  Then $e=xy$ is an idempotent with $ex=x$ and $e\mathrel{\mathscr D} x$.  It follows from~\cite{CS15} that the local divisor $M_x$ is isomorphic to the local divisor $M_e$, which in turn is just $eMe$ with its usual product as a subsemigroup of $M$.  Thus $eMe\cong \Psi^*\ast S$ by Proposition~\ref{p:define.map.to.divisor} and so the group of units of $eMe$ is the group $G$ of units of $S$.

Suppose that $f\in M\setminus \{1\}$ is an idempotent.
We show that $f\mathrel{\mathscr D}e_0$ with $e_0$ an idempotent of $eMe$.
Since $eMe \cong \Psi^* \ast S$ where $S$ is a finitely presented special monoid, it follows from~\cite[Theorem~4.6]{Malheiro2005} that every idempotent in $eMe$ is $\mathscr{D}$-related in the monoid $eMe$ to the idempotent $e$, and hence every maximal subgroup of $eMe$ is isomorphic to its group of units.
From this it will then follow that in the monoid $M$ we have $f\mathrel{\mathscr D}e_0\mathrel{\mathscr D}e$.
Since the maximal subgroups at $\mathscr D$-equivalent idempotents are isomorphic, this will complete the proof that the maximal subgroup of $M$ containing $f$ is isomorphic to $G$.
Let $p \in A^+$ with $f=[p]$.
If $p$ does not contain $v$ as a subword, then neither does $p^2$ and hence $[p] \neq [p^2]$ since no relations can be applied to either of the words $p$ or $p^2$.  Thus $p = p_1 v p_2$ and so $f=m_1e m_2$ for some $m_1,m_2\in M$ using that $ex=x$.  Let $e_0 = em_2fm_1e\in eMe$.  Then $e_0^2=em_2fm_1e m_2fm_1e= em_2f^3m_1e=em_2fm_1e=e_0$.  Also note that $e_0\gr em_2f\gl f$.  Indeed, $e_0\in em_2fM$ trivially.  But also $e_0e m_2f = em_2fm_1e m_2f = em_2f^3=em_2f$ and so $e_0\gr em_2f$.  Trivially, $em_2f\in Mf$.  But $f=f^2=m_1e m_2f$ and so $f\gl em_2f$.  Thus $e_0\mathrel{\mathscr D} f$, as required.

To finish the proof we shall show the equivalence of (1) and (2).
We have $M=\lb A\mid u=v\rb$ with $u = u'' v = v u'$ where $u'', u' \in A^+$.
Suppose that $u'$ is a proper power, say $u' = z^n$ for some word $z \in A^+$ and some $n > 1$.
As observed above we have $z^n = u' \in T(v) = \Delta^*$.
We claim that in fact $z \in T(v)$.
To prove this observe that for all $k \in \mathbb{N}$ we have
\[
v(u')^k = u'' v (u')^{k-1} = \ldots = (u'')^k v.
\]
Choosing $k > |v|$ this implies that $v = u_2 (u')^m$ for some $m \geq 0$ and some suffix $u_2$ of $u'$.
Since $u' = z^n$, this in turn implies that $v = z'z^t$ for some $t \geq 0$ and some suffix $z'$ of $z$.
It now follows, writing $z = z''z'$, that
\[
vz = z'z^tz = z'zz^t = z'z''z'z^t = z'z''v,
\]
which proves that $z \in T(v)$.
Now $z \in T(v) = \Delta^*$ so we can write $z = \gamma_1 \ldots \gamma_q$ with $\gamma_i \in \Delta^*$.
Hence
\[
u' = z^n = (\gamma_1 \ldots \gamma_q)^n
\]
with $n>1$. In this case this implies that
\[
S=\lb \Lambda\mid
(\gamma_1 \ldots \gamma_q)^n=1
\rb
\]
is a special one-relator monoid with torsion.
It then follows
that the group of units $G$ of $S$ is a one-relator group with torsion (see, e.g.,  the argument in~\cite[Section~3]{GraySteinberg2}).
This proves that if $u'$ is a proper power, then $M$ has maximal subgroups with torsion.
The converse is also clearly true: if the group of units of $S$ has torsion then the word $u' = \alpha_1 \ldots \alpha_k$ must be a proper power in $\Lambda^+$ and hence in $A^+$; see~\cite[Section~3]{GraySteinberg2}.
This completes the proof of the lemma.
\end{proof}

Note that the above argument shows that all non-trivial idempotents in a subspecial one-relator monoid are $\mathscr D$-equivalent.

If $V$ is a left $\mathbb ZM$-module, then $xV$ is naturally a left $\mathbb ZM_x$-module, and hence a $\mathbb ZS$-module, via $m\circ w  = rxy=rw $ where $m=rx\in M_x$ and $w =xy$.  Moreover, we have $x\mathbb ZM\cong \mathbb ZxM$.  As $xM$ is a free left $S$-set by the dual of Proposition~\ref{p:free.over.local}, it follows that  if $F$ is a free $\mathbb ZM$-module, then $xF$ is a free $\mathbb ZS$-module and if $P$ is a projective $\mathbb ZM$-module, then $xP$ is a projective $\mathbb ZS$-module.

\begin{Prop}\label{p:nat.iso}
Let $V$ be a $\mathbb ZM$-module.  Then $xV$ is naturally isomorphic to $\Hom_{\mathbb ZM}(\mathbb ZMx,V)$ as a left $\mathbb ZS$-module.  Hence the functor $V\mapsto xV$ is exact.
\end{Prop}
\begin{proof}
Map $\psi\in \Hom_{\mathbb ZM}(\mathbb ZMx,V)$ to $\psi(x)$.  Note that since $x$ is regular, $x=xyx$ with $y\in M$. So $\psi(x) = xy\psi(x)\in xV$.  If $s\in S$, write $s=rx$.  Then $(s\psi)(x) =\psi(x\circ s) = \psi(s)=\psi(rx) = r\psi(x)= s\circ \psi(x)$. Thus we have defined a $\mathbb ZS$-module homomorphism $\Hom_{\mathbb ZM}(\mathbb ZMx,V)\to xV$.  To see that it is injective, if $\psi(x)=0$, then $\psi(\mathbb ZMx)=\mathbb ZM\psi(x)=0$ and so $\psi=0$.

Suppose that $w =xz\in xV$.  Define $\psi\colon \mathbb ZMx\to V$ by $\psi(m) = myw $ (where $xyx=x$).  Then $\psi(x)=xyw =xyxz=xz=w $ and $\psi$ is clearly a $\mathbb ZM$-module homomorphism.  The final statement follows because $\mathbb ZMx$ is projective by Proposition~\ref{p:is.regular} and we have just shown that the functors $\Hom_{\mathbb ZM}(\mathbb ZMx,-)$ and $x(-)$ are isomorphic.
\end{proof}

\subsection*{Constructing an equivariant classifying space}
In this subsection we start with an equivariant classifying space $Y$ for the special one-relator monoid $S$.  By one of the main results of our previous paper (see~\cite[Section~3]{GraySteinberg2}), we can assume that $Y$ is of $S$-finite type and that $\dim Y\leq 2$ unless $u'$ is a proper power (by Lemma~\ref{lem:max:subgroups}).

Since, as $S$ is defined by a special presentation, every projective $S$-set is free by the results of~\cite[Section~3]{GraySteinberg2}, $Y$ is a free $S$-CW complex. Fix a vertex $y_0\in Y$ that is part of the basis for the $0$-skeleton $Y^0$.  Put
\[\til Y= Mx\otimes_S Y.\]
Here $Y$ is a left $S$-space, $Mx$ is an $M$-$S$-biset, and $\til Y$ is a left $M$-set with action $m(a\otimes b) = ma\otimes b$. Also  $\til Y$ is a topological space with the quotient topology, and furthermore is a projective $M$-CW complex of $M$-finite type by~\cite[Corollary~3.2]{GraySteinberg1} and Proposition~\ref{p:is.regular}, and $\dim \til Y=\dim Y$.  Moreover since
$Mx$ is a free right $S$-set,
by Corollary~\ref{c:free.action},
 we have as a topological space that
\[\til Y = \coprod_{[t vw ]\in C} [t vw] \otimes Y\cong \coprod_{Mx/S} Y\]
where $C$ is as in Corollary~\ref{c:free.action} with $r=v$.
So, $C$ is the set of elements of the form $[t vw]_M$ with $t\in \mathcal B$ and $w \in \Delta^*$ with $w =\varepsilon$ or the last $\Delta$-letter of $w $ belonging to $\Psi$.
Recall that the $Mx/S$ denotes the weak orbits of $Mx$ as as right $S$-set, as defined in Section~\ref{sec:preliminaries} above.
Observe that $x\in C$.  Note that $M(x\otimes y_0)\cong (Mx\otimes_S S)\cong Mx$ via an isomorphism  $\gamma\colon M(x\otimes y_0)\to Mx$ sending $mx\otimes y_0$ to $mx$.
However, $\til Y$ is not connected; it is homotopy equivalent to $|C|$ points.

Let $\Gamma$ be the Cayley graph of $M$ with respect to $A$.  It is a free connected $M$-CW complex of $M$-finite type.  Note that $M(x\otimes y_0)\cong Mx$ is a projective $M$-CW
subcomplex of the $0$-skeleton of $\til Y$,
and $Mx$ embeds into the $0$-skeleton of $\Gamma$ in the obvious way. It follows that
using $\gamma$ we can form the pushout $X=\til Y\coprod_{M(x\otimes y_0)} \Gamma$ and this will be a connected projective $M$-CW complex by~\cite[Lemma~2.1]{GraySteinberg1} and clearly is of $M$-finite type.  It will be convenient to identify $Y$ with the copy $x\otimes Y$ in $X$.  It is important to note that two vertices of $\Gamma$ in $Mx$ belong to the same component of $\til Y$ (under the identification $\gamma$) if and only if they belong to the same class in $Mx/S$, or equivalently are $S$-translates of the same basis element from $C$.

Unfortunately, the complex $X$ is not contractible.  The problem is that if $\alpha\in \Lambda$, then there is a path  labelled by $\alpha$ from $mx$ to $mx[\alpha]$ in $\Gamma$, and also a path in $mx\otimes Y$ from $mx\otimes y_0$ to $mx\otimes [\alpha]_Sy_0=mx[\alpha]\otimes y_0$, and the endpoints of these two paths are identified in $X$.  To rectify this problem, we need to adjoin some $2$-cells.  Fix, for each $\alpha\in \Lambda$, a shortest length path $p_{\alpha}$ in the $1$-skeleton $Y^1$ of $Y$ from $y_0$ to $[\alpha]_Sy_0$ and let $q_{\alpha}$ be the path from $x$ to $x[\alpha]$ labelled by $\alpha$ in $\Gamma$.  Note that the vertices visited by the proper prefixes of $mq_{\alpha}$, with $m\in Mx$, are distinct and do not belong to $Mx$ (except the initial vertex $m$) by Lemma~\ref{l:embed.prefix}. Recall that we are identifying $Y$ with $x\otimes Y$ and so we can think of $p_{\alpha}$ as a path from $x\otimes y_0$ to $x\otimes [\alpha]_Sy_0 = x[\alpha]\otimes y_0$.  Thus $q_{\alpha}p_{\alpha}\inv$ is a loop at $x\otimes y_0$ in $X$.   We now attach a projective $M$-cell $Mx\times B^2\to X$ (projective by Proposition~\ref{p:is.regular}) so that the $2$-cell $\{mx\}\times B^2$ has boundary path $m\left(q_{\alpha}p_{\alpha}\inv\right)$ to obtain a projective $M$-CW complex $Z$ of $M$-finite type and $\dim Z\leq \max\{2,\dim Y\}$.  Also  notice that the attaching map of no higher dimensional cell maps into the open $2$-cells corresponding to the $2$-cells we added in this final stage, nor into the paths $mq_{\alpha}$ except perhaps at the endpoints.  This will allow us later to collapse these $2$-cells into $\til Y$.

\begin{Thm}\label{t:subspecial.equiv}
Let $M=\lb A\mid u=v\rb$ with $u\in vA^*\cap A^*v$ and $u\neq v$.  Write $u=vu'$ and assume that $v\neq 1$.
The $M$-finite projective $M$-CW complex $Z$ constructed above is an equivariant classifying space for $M$.  It has dimension at most $2$ unless $u'$ is a proper power.
\end{Thm}
\begin{proof}
To prove the contractibility of $Z$, let $F\subseteq \Gamma$ be the forest from Theorem~\ref{t:compress.cayley.graph.general} with $r=v$.  Then $F$ is a subcomplex of $Z$ and we can form $Z/F$ by contracting each of the trees of $F$, which is homotopy equivalent to $Z$ as $F$ is a forest.  In this proof it will be convenient to identify the vertices of $\Gamma'$ (from Theorem~\ref{t:compress.cayley.graph.general}) with $Mx$, as described
in
Remark~\ref{rmk:before:thm}.
Notice that contracting $F$ does not affect  $\til Y$ (since the vertices of $Mx$ are in distinct equivalence classes). So
by Theorem~\ref{t:compress.cayley.graph.general},
$Z/F$ can be identified with the complex obtained from $\til Y$ by adding a new vertex $\ast$ with an edge from $\ast$ to $b\otimes y_0$ for each $b\in B$ (with $B$ as in Proposition~\ref{p:free.over.local}) and adjoining, for each $m\in Mx$ and $\alpha\in \Delta$, an edge from $m\otimes y_0$ to $m[\alpha]\otimes y_0 = m\otimes [\alpha]_Ly_0$ and adjoining $2$-cells corresponding to the images under the contraction of $F$ of the $2$-cells we had adjoined to form $Z$ from $X$.

Let us now see what happens to those $2$-cells we added in constructing $Z$,
with boundary paths $m\left(q_{\alpha}p_{\alpha}\inv\right)$ with $m\in Mx$ and $\alpha\in \Lambda$,
when we contract $F$.
By Theorem~\ref{t:compress.cayley.graph.general}(3) the only edge of $mq_{\alpha}$ that survives is its final edge, which is labelled by $\alpha$ and goes in the quotient from $m\otimes y_0$ to $m[\alpha]\otimes y_0=m\otimes [\alpha]_Ly_0$. (Recall that $q_{\alpha}$ is the path labelled by $\alpha$ from $m$ to $m[\alpha]$ in $\Gamma$ and each proper prefix of $\alpha$ belongs to $P_{\Delta}$.)  Nothing happens to $mp_{\alpha}\inv$. Since no higher dimensional cell maps into the corresponding open $2$-cell under its characteristic map, nor into the edge labelled by $\alpha$ obtained from $mq_{\alpha}$ under contraction except perhaps at its endpoints, and the edge labelled by $\alpha$ is only on the boundary of this one $2$-cell (and is not used in $mp_{\alpha}\inv$), it is a free face of the image of the original $2$-cell in $Z/F$ and so we can perform an elementary collapse of the $2$-cell along this face.   Thus we can remove each of these $2$-cells from $Z/F$ and at the same time each edge labelled by an element of $\Lambda$ without changing the homotopy type of $Z/F$.  Thus we get a new CW complex $Z'$, homotopy equivalent to $Z$ and which consists of $\til Y$, a vertex $\ast$, an edge from $\ast$ to $[t v]\otimes y_0$ for each $t\in \mathcal B$ (with $\mathcal B$ as in Proposition~\ref{p:free.over.local}), which we view as labelled by $t$, and, for each $\alpha \in \Psi$ and $m\in Mx$, an edge labelled by $\alpha$ from $m\otimes y_0$ to $m[\alpha]\otimes y_0$.  For the rest of the proof it will be convenient to identify $B$ with $\mathcal Bv$
as per Proposition~\ref{p:free.over.local}, with $r=v$.
Since $Y$ is contractible, we can now contract each connected component of $\til Y =\coprod_{c\in C}c\otimes Y\cong \coprod_{Mx/S} Y$ in $Z'$ to get a new CW complex $Z''$ that is homotopy equivalent to $Z$ and which, in fact, is a graph since we have contracted all higher dimensional cells.  We shall see that $Z''$ is a tree,
and hence is contractible.  It will then follow that $Z$ is contractible.

For each $c\in C$, we view the component $c\otimes Y$ as being contracted to $c\otimes y_0$.  We then have that the vertices of $Z''$ can be identified with the set consisting of $\ast$ and the elements of $C$.  For each $t\in \mathcal B$, there is an edge from $\ast$ to $[t v]\in C$ labelled by $t$ (and note that $[tv]$ uniquely determines $t$ by Corollary~\ref{c:free.action}). The remaining edges are labeled by elements of $\Psi$.  Suppose in $Z'$ we have an edge from $m\otimes y_0$ to $m[\alpha]\otimes y_0$ labelled by $\alpha\in \Psi$.  We can write $m$ uniquely as $cs$ with $c\in C$ and $s\in S$.  Then the corresponding edge in $Z''$ goes from $c$ to $cs[\alpha]\in C$ and is still considered labelled by $\alpha$. Note that $c$, $\alpha$ and $s$ are uniquely determined by $cs[\alpha]$, by Lemma~\ref{lem:lalle3.2}, Corollary~\ref{c:free.action} and the definition of $C$ since $L=\Psi^*\ast S$.  Conversely, if $c\in C$, $s\in S$ and $\alpha\in \Psi$, then $cs\otimes y_0$ has an edge labelled by $\alpha$ to $cs[\alpha]\otimes y_0$ in $Z'$ and so there is an edge labelled $\alpha$ from $c$ to $cs[\alpha]$ in $Z''$. In other words, $Z''$ is isomorphic to the following tree.  Consider the alphabet $\Psi\cup (S\setminus \{1\})\cup \mathcal Bv$. Consider the set of words over this alphabet $U=\{\varepsilon\}\cup \mathcal Bv\Psi^*((S\setminus \{1\})\Psi^+)^*$. Then $Z''$ is isomorphic to the Hasse diagram of $U$ under the prefix order, which is a tree.

This completes the proof that $Z$ is contractible.
\end{proof}

We thus have the following theorem which, together with Lemma~\ref{lem:max:subgroups}, proves Theorems~\ref{theorem:A:full} and~\ref{theorem:B:full} in the subspecial case.

\begin{theorem}\label{t:subspecial}
Let $M$ be a subspecial one-relator monoid.
 Then $M$ is of type left- and right-$\F_\infty$, and thus also of type left- and right-$\FP_\infty$.
 Moreover,
\[
\mathrm{cd}^{\mathrm{(l)}}(M) \leq  \mathrm{gd}^{\mathrm{(l)}}(M) \leq 2,
\]
and
\[
\mathrm{cd}^{\mathrm{(r)}}(M) \leq  \mathrm{gd}^{\mathrm{(r)}}(M) \leq 2,
\]
unless $M$ has
a maximal subgroup with torsion, in which case
\[
\mathrm{cd}^{\mathrm{(l)}}(M) =
\mathrm{cd}^{\mathrm{(r)}}(M) =
\mathrm{gd}^{\mathrm{(l)}}(M) =
\mathrm{gd}^{\mathrm{(r)}}(M) =
 \infty.
\]
\end{theorem}
\begin{proof}
Without loss of generality we may assume that $M$ is subspecial but not special as the special case was handled in~\cite[Section~3]{GraySteinberg2}.  We retain the above notation.
By Lemma~\ref{lem:max:subgroups}, $M$ has a maximal subgroup with torsion if and only if $u'$ is a proper power, which occurs precisely when the group of units $G$ of $S$ has torsion.

The claims made in the statement of the theorem now all follow from Theorem~\ref{t:subspecial.equiv} except the claim that if $M$ has a maximal subgroup which is a one-relator group with torsion, then it has infinite left cohomological dimension.
To prove this claim, suppose that $G$ has torsion and that
\[P_n\rightarrow P_{n-1}\rightarrow \cdots\rightarrow P_0\rightarrow \mathbb Z\] is a finite projective resolution of $\mathbb Z$ over $\mathbb ZM$.
Then by Proposition~\ref{p:nat.iso}, the dual of Corollary~\ref{c:free.action}, and the observation that $x\mathbb Z=\mathbb Z$ as a $\mathbb ZS$-module, we obtain a finite projective resolution
\[xP_n\rightarrow xP_{n-1}\rightarrow \cdots\rightarrow xP_0\rightarrow \mathbb Z\] over $\mathbb ZS$.
 This contradicts the result of~\cite[Section~3]{GraySteinberg2} that $S$ has infinite left  cohomological dimension.
 Thus $M$ has infinite left cohomological dimension.
The argument for the right cohomological dimension is dual.
\end{proof}

\begin{remark}\label{rem:subspecial:genealised}
While it is not needed for the main results of this paper, it is worth noting that the construction and results given in this section (and Section~\ref{s:compression}) may be generalised in a natural way to non-one-relator monoids which are subspecial, in the following sense.
Let $M = \lb A \mid u_i = v \; (i \in I) \rb$, with $I$ finite, and where $u_i \in vA^* \cap A^*v$ for all $i \in I$.
Define $T(v)$ and $\Delta$ as above.
For each $i \in I$ we can then decompose $u_i = v w_i$ with $w_i = \alpha_{i, 1} \ldots \alpha_{i, k_i}$ where $\alpha_{i,j} \in \Delta$ for $1 \leq j \leq k_i$.
Set
\[L=\lb \Delta\mid \alpha_{i,1}\cdots \alpha_{i,k_i}=1 \; (i \in I) \rb\]
and
\[S=\lb \Lambda\mid \alpha_{i,1}\cdots \alpha_{i,k_i}=1 \; (i \in I) \rb\]
where $\Lambda \subseteq \Delta$ is the finite subset of elements appearing in the relations in the presentation of $L$.
Let $G$ be the Sch\"{u}tzenberger group of the $\gh$-class of $v$ in $M$.
It may be shown that the group of units of the special monoid $S$ is isomorphic to $G$.
Then by combining the arguments given in this section with the results proved in~\cite[Section~3]{GraySteinberg2} it may be shown that:

\begin{enumerate}
\item If $G$ is of type $\FP_n$ with $1 \leq n \leq \infty$, then $M$ is of type left- and right-$\FP_n$; and
\item We have \[\mathrm{cd}(G) \leq  \mathrm{cd}^{\mathrm{(l)}}(M) \leq \mathrm{max}\{2, \mathrm{cd}(G)\}\]
and
\[\mathrm{cd}(G) \leq  \mathrm{cd}^{\mathrm{(r)}}(M) \leq \mathrm{max}\{2, \mathrm{cd}(G)\}.\]
  \end{enumerate}
The analogous statements also hold for the topological finiteness property $\F_n$ and the geometric dimension.
\end{remark}

\section{The relation module}
\label{s:relmodule}

In this section we apply the results of Section~\ref{s:compression} to compute the relation module, in the sense of Ivanov~\cite{Ivanov}, of a torsion-free one-relator monoid.
This will be a key stepping stone
towards proving that one-relator monoids are of type left- and right-$\FP_\infty$ in the non-subspecial case.
It will also lead to another proof that if $M$ is subspecial, but not special, and torsion-free, then $M$ is of type right- and left-$\F_{\infty}$ and $\FP_\infty$,
and $\mathrm{cd}^{\mathrm{(l)}}(M) \leq  \mathrm{gd}^{\mathrm{(l)}}(M) \leq 2$, by showing that the $2$-complex $K$ in Theorem~\ref{t:compress.cayley.graph.general}, with $r=v$, is an equivariant classifying space for $M$.  We begin by giving a topological interpretation of the relation module in terms of the homology of the Cayley graph of the monoid.  This is completely analogous to a well-known result concerning group relation modules.

The following lemma is well known and will be useful several times.
The proof is straightforward and is omitted.

\begin{lemma}\label{l:map.of.set}
Let $f\colon X\to Y$ be a mapping of sets.  Let $\varphi\colon \mathbb ZX\to \mathbb ZY$ be the $\mathbb Z$-linear extension of $f$.  Then $\ker \varphi$ is spanned as an abelian group by the elements $x-x'$ with $x,x'\in X$ and $f(x)=f(x')$.
\end{lemma}

If $M$ is a monoid, then the \emph{augmentation ideal} of $\mathbb ZM$ is the two-sided ideal $\omega(\mathbb ZM)$
generated as an abelian group by all differences $m-n$ with $m,n\in M$, i.e., it is the kernel of the natural homomorphism $\mathbb ZM\to \mathbb Z$ mapping each element of $m$ to $1$ by Lemma~\ref{l:map.of.set}.  It is well known that $\omega(\mathbb ZM)$ can be generated as a left $\mathbb ZM$-module by the elements of the form $a-1$ where $a$ ranges over a generating set of $M$, cf.~\cite{Pride95}.

Let $\lb A \; | \; R \rb$ be a presentation of a monoid $M$.  Let $I$ be the kernel of the natural homomorphism $\gamma\colon \mathbb ZA^*\to \mathbb ZM$.  Note that $I$ is generated as an abelian group by all differences $u-v$ with $u,v\in A^*$ representing the same element of $M$ by Lemma~\ref{l:map.of.set}.
Ivanov~\cite{Ivanov} defined the (left) relation module $\mathcal R$ of the presentation to be the quotient $I/I\omega(\mathbb ZA^*)$.  Since $I\subseteq \omega(\mathbb ZA^*)$, we have $I^2\subseteq I\omega(\mathbb ZA^*)$ and so $\mathcal R$ is a left $\mathbb ZM$-module in a natural way. The right relation module is defined dually.  He showed this notion generalises the notion of the relation module of a group presentation and initiated the study of the relation module.  We shall prove here that if the Cayley complex of a presentation is contractible, then the relation module is freely generated by a set in bijection with the relations. 
 We shall also identify the relation module of a torsion-free one-relator monoid.

Let $\Gamma$ be the Cayley graph of $M$ with respect to $A$.  Our first goal is to identify the relation module with $H_1(\Gamma)$ with respect to the natural module structure coming from the left action of $M$ on $\Gamma$; the corresponding result for groups is well known (it can be found on~\cite[Page~43]{BrownCohomologyBook}).  For monoids, the result has appeared in other guises in various sources, e.g.~\cite{Pride95}.  Since $\Gamma$ is a $1$-complex, we can identify $H_1(\Gamma)$ with the group $Z_1(\Gamma)$ of $1$-cycles.  

\begin{Prop}\label{p:Cayley.identification}
Let  $\lb A \; | \; R \rb$ be a presentation of a monoid $M$. Let $\Gamma$ be the Cayley graph of $M$ with respect to $A$ and let $\mathcal R$ be the relation module.  Then $\mathcal R\cong H_1(\Gamma)$ as $\mathbb ZM$-modules. 
\end{Prop}
\begin{proof}
Note that $C_1(\Gamma)$ can be identified with a free $\mathbb ZM$-module with basis $\{e_a\mid a\in A\}$ by sending the edge $m\xrightarrow{\,\,a\,\,}[ma]_M$ to $me_a$  and that $C_0(\Gamma)\cong \mathbb ZM$.  The boundary map $\partial_1$, under these identifications, takes the basis element $e_a$ to $[a]_M-1$, and hence $\partial_1$ has image the augmentation ideal $\omega(\mathbb ZM)$.  The kernel of this homomorphism is $\mathcal R$ according to~\cite[Theorem~3.3]{Pride95} and so $H_1(\Gamma)=\ker \partial_1\cong \mathcal R$.
\end{proof}

\begin{Cor}\label{c:free.relation}
Let $\lb A \; | \; R \rb$ be a presentation of $M$ such that the corresponding Cayley complex is contractible.  Then the relation module $\mathcal R$ of the presentation is a free $\mathbb ZM$-module on $|R|$-generators.
\end{Cor}
\begin{proof}
Let $K=\Gamma(M,A)^{(2)}$ be the Cayley complex of $M$. By Proposition~\ref{p:Cayley.identification}, we can identify $\mathcal R$ as a $\mathbb ZM$-module with $H_1(K^1)$, which is just the group of $1$-cycles $Z_1(K)$.  Since $K$ is a contractible $2$-complex, we have that $Z_1(K)=B_1(K)\cong  C_2(K)$.  But $C_2(K)$ is a free $\mathbb ZM$-module with basis the cells with boundary path $p_up_v\inv$ with $u=v\in R$ where $p_w$ is the path starting at $1$ labelled by $w$.
\end{proof}

We shall need that the Cayley complex of a strictly aspherical presentation is contractible.

\begin{lemma}\label{lem:9}
Let $N$ be the monoid defined by the presentation
$\lb A \mid R \rb$.
If   $\lb A \mid R \rb$ is strictly aspherical, then the Cayley complex $\Gamma(N,A)^{(2)}$ is contractible, and thus is a left equivariant classifying space for $N$.
\end{lemma}
\begin{proof}
Let $X = \Gamma(N,A)^{(2)}$.
It follows from the proof of~\cite[Theorem~6.14]{GraySteinberg1} that $X$ is an $N$-finite simply connected free $N$-CW complex of dimension at most 2.
Since the presentation $\lb A \mid R \rb$ is strictly aspherical it follows that the cellular chain complex of $X$ gives the free resolution displayed in Equation~(7.2) in~\cite[Theorem~7.2]{Kobayashi1998}.
This resolution was originally discovered in the papers
\cite{Cremanns94,Lafont95,Pride95}.
This shows that $X$ is acyclic.
Since $X$ is acyclic and simply connected it follows from the Whitehead and Hurewicz theorems that $X$ is contractible, and hence $X$ is a left equivariant classifying space for the monoid $N$.
\end{proof}

Next we discuss the topology of the Cayley complex of a free product.

\begin{lemma}\label{l:free.contract}
Let $L=\Psi^*\ast S$ where $S=\lb \Lambda\mid R\rb$ with $\Psi\cap \Lambda=\emptyset$.  Suppose that the Cayley complex $\Gamma(S,\Lambda)^{(2)}$ of $S$ with respect to this presentation is contractible. Put $\Sigma=\Psi\cup \Lambda$.
Then the Cayley complex $\Gamma(L, \Sigma)^{(2)}$ with respect to the presentation $\lb \Sigma\mid R\rb$ is contractible.
\end{lemma}
\begin{proof}
Let $X= \Gamma(L, \Sigma)^{(2)}$ and put $\mathcal B=\{1\}\cup \Sigma^*\Psi$.  Then $\mathcal B$ is a basis for $L$ as a free right $S$-set by the free product normal form.  It follows easily from this that if $b\in \mathcal B$ and we consider the induced subcomplex $K_b$ of $X$ on the vertex set $bS$, then we obtain an isomorphic copy of $\Gamma(S, \Lambda)^{(2)}$ via the map defined on vertices by $s\mapsto bs$ with the obvious extension on edges and $2$-cells.  Moreover, if $b\neq b'$, then $K_b\cap K_{b'}=\emptyset$.  By assumption, each $K_b$ is contractible and hence if we contract each of these subcomplexes to a point, we obtain a $2$-complex $X'$ that is homotopy equivalent to $X$. Moreover, each $2$-cell of $X$ belongs to one of the $K_b$ and so $X'$ is a graph.  In fact, it is a tree.  If we view $K_b$ as being contracted to the vertex $b$, then we see that $X'$ has vertex set $\mathcal B$.  The only edges that survive the contraction of $\coprod_{b\in \mathcal B} K_b$ are the edges labeled by elements of $\Psi$.  For each $b\in \mathcal B$, $s\in S$ and $x\in \Psi$, there is an edge $b\to bsx$ (the image of the edge $bs\xrightarrow{\,\,x\,\,}bsx$) and this edge is uniquely determined by its endpoints.  Thus $X'$ is the Hasse diagram for the free product normal forms of the elements represented by $\mathcal B$ over the alphabet $\Psi\cup S$ under the prefix ordering, which is a rooted tree.
\end{proof}

We recall that if $M=\lb A\mid u=v\rb$ is an incompressible non-special one-relator monoid or if $M=\lb A\mid u=1\rb$ and $u$ is not a proper power, then the one-relator presentation is strictly aspherical by the results of
Kobayashi~\cite{Kobayashi1998,Kobayashi2000}.  It follows that in either of these cases the relation module is $\mathbb ZM$ by Corollary~\ref{c:free.relation}.

\begin{Thm}\label{thm:not:the:clas:space}
Let $M = \lb A \mid u=v \rb$ be torsion-free,
let   $z \in A^*$ be the longest word which is a prefix and a suffix of both $u$ and $v$, and write $u = zu'$ and $v=zv'$.
Let $K$ be the $2$-complex with $1$-skeleton the Cayley graph $\Gamma$ of $M$, and a $2$-cell adjoined at every vertex in $m\in M[z]_M$ with boundary path $p_{m,u'}p_{m,v'}\inv$ where $p_{m,w}$ is the path labelled by $w\in A^*$ beginning at $m$.
Then $K$ is contractible.
\end{Thm}
\begin{proof}
Note that the boundary paths of the adjoined cells are indeed closed paths since $m[u']_M=m[v']_M$ for $m\in M[z]_M$.
If the presentation is incompressible (i.e., $z$ is empty), then $M$ is strictly aspherical by Kobayashi's results~\cite[Corollary~5.6]{Kobayashi1998} and~\cite[Corollary~7.5]{Kobayashi2000} and hence $K$, which is the Cayley complex in this case, is contractible by Lemma~\ref{lem:9}.  So assume that $M$ is compressible.
As usual put $L=\lb \Delta_z\mid u'=v'\rb$ and $S=\lb \Lambda_z\mid u'=v'\rb$ where $u=zu'$ and $v=zv'$.
The presentation for $S$ is incompressible by Corollary~\ref{c:get.to.lalle} and is either not special or it is special with defining relation $u'=1$ with $u'$ not a proper power by Lemma~\ref{lem:max:subgroups}.  Thus the presentation of $S$ is strictly aspherical by Kobayashi's results~\cite{Kobayashi1998,Kobayashi2000} and so the Cayley complex $\Gamma(S,\Lambda_z)^{(2)}$ is contractible by Lemma~\ref{lem:9}.  Therefore, the Cayley complex   $\Gamma(L,\Delta_z)^{(2)}$ of $L$ is contractible by Lemma~\ref{l:free.contract}.  Theorem~\ref{t:compress.cayley.graph.general} implies that $K$ (which is the $2$-complex from that theorem with $r=z$) is homotopy equivalent to a wedge of copies of the Cayley complex $\Gamma(L,\Delta_z)^{(2)}$ and so we deduce that $K$ is contractible
\end{proof}

Notice that if $z$ is non-empty, then $K$ in the theorem above is not the Cayley complex of $M$.  Recall that $C_2(K)\cong \mathbb ZM[z]_M$ as a left $\mathbb ZM$-module by Proposition~\ref{prop:chains}.

\begin{Thm}\label{t:relation.module.compress}
Let $M=\lb A\mid u=v\rb$ be a one-relator presentation of a torsion-free one-relator monoid.  Let $z\in A^*$ be the longest word with $u,v\in zA^*\cap A^*z$.  Then the relation module of the presentation is isomorphic to $\mathbb ZM[z]_M$.  Moreover, there is an exact sequence
\[0\longrightarrow \mathbb ZM[z]_M\longrightarrow \mathbb ZM^{|A|}\longrightarrow \mathbb ZM\longrightarrow \mathbb Z\longrightarrow 0\] of $\mathbb ZM$-modules.
\end{Thm}
\begin{proof}
Let $K$ be the $2$-complex from Theorem~\ref{thm:not:the:clas:space} and $\Gamma$ the Cayley graph of $M$.  Then $K$ is contractible by Theorem~\ref{thm:not:the:clas:space} and so the boundary mapping $\partial_2\colon C_2(K)\to C_1(K)$ is injective and has image $\ker \partial_1$ where $\partial_1\colon C_1(K)\to C_0(K)$ is the boundary mapping.  But $\ker \partial_1\cong H_1(\Gamma)$ and $C_2(K)\cong \mathbb ZM[z]_M$ and so we deduce that $H_1(\Gamma)\cong \mathbb ZM[z]_M$ as $\mathbb ZM$-modules.  This completes the proof of the first statement by Proposition~\ref{p:Cayley.identification}.

The sequence is exact because it is the augmented cellular chain complex of $K$, which is contractible.
\end{proof}

Note that when $z$ is empty, i.e., when the presentation is incompressible then the conclusion of Theorem~\ref{t:relation.module.compress} can also be deduced from Corollary~\ref{c:free.relation}.

We now deduce that $K$ in the above theorem is an equivariant classifying space when $M$ is subspecial and torsion-free.  This provides another proof of Theorem~\ref{t:subspecial} in the torsion-free case.

\begin{Cor}\label{c:subspecial.case.relation.module.complex}
Let $M=\lb A\mid u=v\rb$ be a subspecial one-relator monoid such that $u'$ is not a proper power where $u=vu'$.  Let $x=[v]_M$.  Then the $2$-complex $K$ obtained by adjoining at each vertex of $Mx$ a $2$-cell with boundary path labelled by $u'$ is an equivariant classifying space for $M$.
\end{Cor}
\begin{proof}
The $2$-complex $K$ is contractible by Theorem~\ref{thm:not:the:clas:space}.  It just suffices to verify that it is a projective $M$-CW complex. Since $\Gamma$ is a free $M$-CW complex, we need only observe that the $2$-cells are obtained by attaching the $M$-cell $Mx\times B^2$, which is projective by Proposition~\ref{p:is.regular}.
\end{proof}

Note that the $2$-complex
constructed in the above corollary is not the Cayley complex of the one-relator presentation defining the monoid $M$ unless the presentation is special. Indeed, in general for a torsion-free subspecial one-relator monoid $M$ the Cayley complex of $M$ is not an equivariant classifying space for the monoid.  For example, the one-relator subspecial presentation $\lb a\mid a^2=a\rb$ presents the two-element monoid $M=\{0,1\}$ under multiplication.  This monoid has a zero element and hence, by the results of~\cite[Section~6]{GraySteinberg1}, has no finite dimensional, contractible free $M$-CW complex.  In particular, the Cayley complex is not contractible, as is easy to verify directly.

It follows from Theorem~\ref{t:relation.module.compress} that to build a free resolution of $\mathbb Z$ for a compressible, but not subspecial, one-relator presentation, it suffices to build a free resolution of $\mathbb ZM[z]_M$, which we shall proceed to do algebraically.  To make this precise, recall that if $R$ is a ring and
\begin{align*}
E_1\colon &\cdots\longrightarrow C_1\longrightarrow C_0\xrightarrow{\,\,f\,\,} A\longrightarrow 0,\\
E_2\colon &  0\longrightarrow A\xrightarrow{\,\,g\,\,} B_0\longrightarrow B_1\longrightarrow \cdots
\end{align*}
are exact sequences of $R$-modules (finite or infinite in length), then the \emph{Yoneda splice} $E_1E_2$ of these exact sequences is the exact sequence
\[\cdots\longrightarrow C_1\longrightarrow C_0\xrightarrow{\,\,gf\,\,} B_0\longrightarrow B_1\longrightarrow\cdots.\]
Although the Yoneda splice is typically used to define the bilinear mapping $\mathrm{Ext}^m(A,B)\times \mathrm{Ext}^n(B,C)\to \mathrm{Ext}^{m+n}(A,C)$, we shall use it here to build resolutions.
Given a free or projective resolution of $\mathbb ZM[z]_M$, we can splice it onto the exact sequence of Theorem~\ref{t:relation.module.compress} to obtain a free or projective resolution of $\mathbb Z$.

The Yoneda splice can also be applied to  an exact sequence of modules
\[0\longrightarrow A\longrightarrow F_n\longrightarrow F_{n-1}\longrightarrow\cdots\longrightarrow F_0\longrightarrow A\longrightarrow 0\] via infinite iteration to yield a periodic resolution
\[\cdots\longrightarrow F_n\longrightarrow\cdots\longrightarrow F_0\longrightarrow F_n\longrightarrow\cdots\longrightarrow F_0\longrightarrow A\longrightarrow 0,\]
cf.~\cite[Section~I.6]{BrownCohomologyBook}.

\section{An injectivity lemma}
\label{s:injectivity}

The following lemma is crucial in order to prove our main results;  its proof is a surprisingly intricate use of monoid pictures.

\begin{Lemma}\label{l:inj.cut.letter}
Let $N=\langle A\mid u=v\rangle$ be an incompressible one-relator monoid with $u=u'a$ and $v=v'a$ with $a\in A$.  Then the mapping $\psi\colon \mathbb ZN\to \mathbb ZN$ given by $\psi(\alpha) = \alpha([u']-[v'])$ for $\alpha \in \ZN$ is injective.
\end{Lemma}

The rest of this section will be devoted to proving this lemma.

Throughout the rest of this section $N$ will denote the monoid defined by an incompressible one-relator presentation $\langle A\mid u=v\rangle$ where $u=u'a$ and $v=v'a$ with $a\in A$. 
We use $[w]$ to denote the image of the word $w\in A^*$ in the monoid $N$. Note that since $u=v$ is incompressible, $u'\neq v'$.
Let us denote the presentation $\lb A \mid u=v \rb$ by $\gp$ and use $\Gamma(\gp)$ to denote the Squier complex of this presentation. Recall  that given two paths $\bbp$ and $\bbq$ in $\Gamma(\gp)$ we shall write $\bbp \sim_0 \bbq$ to mean that $\bbp$ and $\bbq$ are homotopic in the Squier complex. The presentation $\gp$ is strictly aspherical by~\cite[Corollary~5.6]{Kobayashi1998}), a fact that we shall use throughout the proof of Lemma~\ref{l:inj.cut.letter}.
Recall that this means that for every closed path $\bbp$ in $\Gamma(\gp)$ we have $\bbp \sim_0 1_{\iota \bbp}$. 

The following is a reformulation of Lemma~\ref{l:inj.cut.letter}.

\begin{proposition}\label{prop:injective}
Let $N=\langle A\mid u=v\rangle$ be an incompressible one-relator monoid with $u=u'a$ and $v=v'a$ with $a\in A$.  
For every finite non-empty subset $F = \{ m_1, \ldots, m_k \}$ of $N$ if $z_1, z_2, \ldots, z_k \in \Z \setminus \{ 0 \}$ is a list of (not necessarily distinct) integers, then in $\ZN$ we have
\[
z_1m_1[u'] +
z_2m_2[u'] +
\ldots +
z_km_k[u']
\neq
z_1m_1[v'] +
z_2m_2[v'] +
\ldots +
z_km_k[v'].
\]
\end{proposition}

The rest of the discussion in this section will be devoted to proving this proposition, from which we can then deduce Lemma~\ref{l:inj.cut.letter}.

Choose and fix words $w_i$ such that $m_i = [w_i]$ for $1 \leq i \leq k$. By assumption, for all $i$ and $j$, if $i \neq j$, then this implies that $[w_i] \neq [w_j]$. Seeking a contradiction suppose that we do have
\begin{equation}
z_1[w_1u'] +
\ldots +
z_k[w_ku']
=
z_1[w_1v'] +
\ldots +
z_k[w_kv']
\label{eqn:injective}
\end{equation}
in $\ZN$.

As already observed in Section~\ref{sec:preliminaries}, every edge $\mathbb{E}$ in $\Gamma(\gp)$ can be written uniquely in the form $\alpha \cdot \mathbb{A} \cdot \beta$ with $\alpha, \beta \in A^*$ where $\mathbb{A}$ is an elementary edge. This means that $\mathbb{A}$ is equal to $(1, (u=v), \epsilon, 1)$ for some $\epsilon \in \{+1, -1\}$. We use this unique decomposition to define a map $\lambda$ from the set of edges of the Squier complex to $A^*$ where for $\mathbb{E} = \alpha \cdot \mathbb{A} \cdot \beta$,  with $\bba$ elementary, we define $\lambda(\mathbb{E}) = \alpha$. So $\lambda$ maps each edge to the word to the left of the transistor in the picture of the edge. For any edge $\bbe$ define $L(\bbe) = [\lambda(\bbe)]$. So $L$ defines a mapping from the set of edges of $\Gamma(\gp)$ into the monoid $N$.

Let us call an edge $\mathbb{E}$ in $\Gamma(\gp)$  a \emph{rightmost edge} if it is not of the form $\mathbb{F} \cdot x$ for some edge $\mathbb{F}$ and some letter $x \in A$. So a rightmost edge $\mathbb{E}$ (since the defining relation is $u=v$ with $u$ and $v$ both non-empty) has the property that its transistor has a wire connected to the rightmost letter of $\iota \bbe$ and to the rightmost letter of $\tau \bbe$. 
It is immediate from the definition that $\bbe$ is a rightmost edge if and only if its inverse $\bbe^{-1}$ is a rightmost edge.

In addition, for each $m \in N$, we define a function $f_m$ from the set of paths in $\Gamma(\gp)$ to $\mathbb Z$ where $f_m(\bbp)$ is defined to be the number of rightmost edges $\bbf$ in $\bbp$ such that $L(\bbf) = m$.
Note that  $f_m(\bbp)\geq 0$ for every path $\bbp$ in $\Gamma(\gp)$ and every element $m \in N$.   

\begin{lemma}\label{lem:parity:lemma} For every $m \in N$ and every pair of paths $\bbp$ and $\bbq$ in $\Gamma(\gp)$, if $\bbp \sim_0 \bbq$, then $f_m(\bbp) \equiv f_m(\bbq)    \pmod{2}$.  \end{lemma}
\begin{proof}
By Lemma~\ref{lem:pupd} it suffices to prove the result for the case that $\bbp$ can be transformed into $\bbq$ either by the deletion of a single cancelling pair of edges, or by a single application of pull-up push-down.

First suppose that $\bbp$ can be transformed into $\bbq$ by deletion of a cancelling pair of edges $\bbe \circ \bbe^{-1}$. Then since $[\lambda(\bbe)] = [\lambda(\bbe^{-1})]$ in $N$ it is immediate that $f_m(\bbp) \equiv f_m(\bbq)$ modulo $2$.

Now suppose that $\bbp$ can be transformed into $\bbq$ by a single application of pull-up push-down. So up to symmetry we have
\[\bbp = \bbr \circ (\bbe_1 \cdot \iota \bbe_2) \circ (\tau \bbe_1 \cdot \bbe_2) \circ \bbs\]
and
\[\bbq = \bbr \circ (\iota \bbe_1 \cdot \bbe_2) \circ (\bbe_1 \cdot \tau \bbe_2) \circ \bbs.\]
Observe that $\lambda (\bbe_1 \cdot \iota \bbe_2) = \lambda(\bbe_1 \cdot \tau \bbe_2)= \lambda(\bbe_1)$ by definition. Also by definition we have $\lambda(\tau \bbe_1 \cdot \bbe_2) = (\tau \bbe_1)(\lambda(\bbe_2))$  and $\lambda(\iota \bbe_1 \cdot \bbe_2) = (\iota \bbe_1)(\lambda(\bbe_2))$ and hence it follows that
\[
[\lambda(\tau \bbe_1 \cdot \bbe_2)] =
[\lambda(\iota \bbe_1 \cdot \bbe_2)]
\]
in the monoid $N$.

Furthermore, it is clear that $\tau \bbe_1 \cdot \bbe_2$ is a rightmost edge if and only if $\iota \bbe_1 \cdot \bbe_2$ is a rightmost edge (if and only if $\bbe_2$ is a rightmost edge), and  neither of the edges $\bbe_1 \cdot \iota \bbe_2$ nor $\bbe_1 \cdot \tau \bbe_2$ is a rightmost edge, since $u$ and $v$ are both non-empty words by assumption which implies that both $\iota \bbe_2$ and $\tau \bbe_2$ are non-empty words.

It follows from these observations that $f_m(\bbp) \equiv f_m(\bbq) \pmod{2}$, which completes the proof of this case, and of the lemma.
\end{proof}

\begin{lemma}\label{lem:TheMiltisetProof}
For every word $w \in A^{*}$ we have $[w u'] \neq [w v']$.
\end{lemma}
\begin{proof}
  Suppose, seeking a contradiction, that we have $[wu'] = [wv']$. Right multiplying by $a$ gives $[wu'a]=[wv'a]$. Since $u'a = v'a$ is the defining relation $u=v$ this gives a pair of parallel paths from $wu'a$ to $wv'a$ in $\Gamma(\gp)$: one given by a single edge corresponding to applying the relation $u'a = v'a$, and the other is given by taking a path in $\Gamma(\gp)$ from $wu'$ to $wv'$ and right multiplying by $a$. More precisely, choose and fix a path $\mathbb{P}$ from $wv'$ to $wu'$. Define $\mathcal{C} = (\mathbb{P} \cdot a) \circ (w \cdot (1, (u=v), +1, 1))$ which is a closed path in $\Gamma(\gp)$ based at $(\iota \bbp \cdot a) = wv'a$. Observe that $\mathcal{C}$ has exactly one rightmost edge, namely the edge $\bbe = w \cdot (1, (u=v), +1, 1)$. Clearly this edge satisfies $L(\bbe)=m$ where $m = [w] \in N$. It follows that $f_m(\mathcal{C})=1$. On the other hand, $f_m(1_{\iota \mathcal{C}})=0$ since the empty path has no edges. It then follows from Lemma~\ref{lem:parity:lemma} that $\mathcal{C} \not\sim_0 1_{\iota \mathcal{C}}$. But this contradicts strict asphericity of the presentation $\lb A \mid u=v \rb$.
\end{proof}

\begin{lemma}
\label{lem:aspherical:bipartite}
Assume that \eqref{eqn:injective} holds. 
\begin{enumerate}
\item[(i)] For every $i \in \{1, \ldots, k \}$ there exists a $j \in \{1, \ldots, k \}$ with $j \neq i$ such that either
$[w_i u'] = [w_j u']$ or $[w_i u'] = [w_j v']$. 
\item[(ii)] For every $i \in \{1, \ldots, k \}$ there exists a $j \in \{1, \ldots, k \}$ with $j \neq i$ such that either
$[w_i v'] = [w_j v']$ or $[w_i v'] = [w_j u']$.
  \end{enumerate}
\end{lemma}
\begin{proof}
First note that we have already proved (see Lemma~\ref{lem:TheMiltisetProof} above) that for all $i \in \{ 1, \ldots, k \}$ we have $[w_i u'] \neq [w_i v']$. Now let $i \in \{ 1, \ldots, k \}$. If (i) does not hold this would mean that in equation \eqref{eqn:injective} the expression $z_i [w_i u']$ is the only term of the form $zm$ with with $z \in \Z \setminus \{ 0 \}$ and $m = [w_i u']$, which would mean \eqref{eqn:injective} cannot hold, a contradiction. The proof for (ii) is similar.
\end{proof}

Define $\varphi\colon \{ u', v' \} \rightarrow \{u', v' \}$ by $\varphi(u') = v'$ and $\varphi(v') = u'$. Applying Lemma~\ref{lem:aspherical:bipartite} repeatedly starting at $[w_1 u']$ gives an infinite sequence of equalities
\begin{align*}
[w_1 u'] &= [w_{i_1} \beta_{i_1}]
& \mbox{ where
$i_1 \neq 1$ and $\beta_{i_1} \in \{u', v'\}$ } \\
[w_{i_1} \varphi(\beta_{i_1}) ] &= [w_{i_2} \beta_{i_2}]
& \mbox{ where
$i_2 \neq i_1$ and $\beta_{i_2} \in \{u', v'\}$ } \\
[w_{i_2} \varphi(\beta_{i_2}) ] &= [w_{i_3} \beta_{i_3}]
& \mbox{ where
$i_3 \neq i_2$ and $\beta_{i_3} \in \{u', v'\}$ } \\
\vdots & \\
\end{align*}

Since $\{w_1, \ldots, w_k \}$ is a finite set it follows that some repetition of indices must occur.
 That is, there must exist some $p, q \in \mathbb{N}$ with $p< q$ such that $i_p=i_q$.
It follows, by considering a first repetition, that there is a subsequence of equalities:
\begin{align*}\label{eq_graph}
[w_{j_1} \varphi(\beta_{j_1}) ] &= [w_{j_2} \beta_{j_2}]
\\ \tag{$\ast$}
[w_{j_2} \varphi(\beta_{j_2}) ] &= [w_{j_3} \beta_{j_3}]
\\
\; \; \vdots \\
[w_{j_{n-1}} \varphi(\beta_{j_{n-1}}) ] &= [w_{j_n} \beta_{j_n}]
\\
 [w_{j_n} \varphi(\beta_{j_n}) ] &= [w_{j_1} \varphi^{e}(\beta_{j_1})]
\end{align*}
with $\beta_{j_t} \in \{u', v'\}$  for all $1 \leq t \leq n$,  where all $j_l$ are distinct elements of $\{1,\ldots,k\}$ for $1 \leq l \leq n$, $e\in \{0,1\}$ and $n\geq 2$.
Here $\varphi^{0}(\beta_{j_1}) = \beta_{j_1}$ while $\varphi^{1}(\beta_{j_1}) = \varphi(\beta_{j_1})$.

Right multiplying all of these equations by $a$, and using the relation $u' a = v' a$, gives the following cycle of equalities in the monoid $N$.
\begin{align*}
 & \; \; \; \; \; [w_{j_1} \varphi(\beta_{j_1}) a] = [w_{j_2} \beta_{j_2}a]
\\
& = [w_{j_2} \varphi(\beta_{j_2}) a] = [w_{j_3} \beta_{j_3}a]
\\
& \; \; \vdots \\
& = [w_{j_{n-1}} \varphi(\beta_{j_{n-1}}) a] = [w_{j_n} \beta_{j_n}a]
 \\
& = [w_{j_n} \varphi(\beta_{j_n}) a] = [w_{j_1} \varphi^{e}(\beta_{j_1})a]
\end{align*}
with $\beta_{j_t} \in \{u', v'\}$ for all $1 \leq t \leq n$,  where all $j_l$ are distinct elements of $\{1,\ldots,k\}$ for $1 \leq l \leq n$,  $e\in\{0,1\}$ and $n\geq 2$.

In the Squier complex this corresponds to a closed path $\mathcal{C}$ from $[w_{j_1} \varphi(\beta_{j_1}) a]$ to itself. To describe this closed path,
it will be helpful to introduce some notation. For each $l \in \{1, \ldots, n \}$ choose and fix a path $\bbp_l$ in the Squier complex with initial vertex $w_{j_l} \varphi(\beta_{j_l})$ to $w_{j_{1+1}} \beta_{j_{l+1}}$, for $1\leq l\leq n-1$ and a path $\bbp_n$ from $w_{j_n} \varphi(\beta_{j_n})$ to $w_{j_{1}} \varphi^e(\beta_{j_{1}})$.  Such paths exist because of the first list of equalities above in ($\ast$).
The resulting closed path $\mathcal{C}$ in the Squier complex arising from the above cycle of equalities is:
\[\mathcal{C} = (\bbp_1 \cdot a) \circ (w_{j_2} \cdot \bbe_2) \circ (\bbp_2 \cdot a) \circ (w_{j_3} \cdot \bbe_3) \circ \ldots \circ (\bbp_n \cdot a) \circ (w_{j_1} \cdot \bbe_1) \; (\dagger) \] if $e=0$ and
\[\mathcal{C} = (\bbp_1 \cdot a) \circ (w_{j_2} \cdot \bbe_2) \circ (\bbp_2 \cdot a) \circ (w_{j_3} \cdot \bbe_3) \circ \ldots \circ (\bbp_n \cdot a)\; (\ddagger) \] if $e=1$, where $\bbe_l=(1,(u=v),\epsilon_l,1)$, for all $l$, with
$\epsilon_l=1$ if $\beta_{j_l}=u'$ and $\epsilon_l=-1$ if $\beta_{j_l}=v'$.
See Figure~\ref{fig_C} for an illustration of this closed path in the case $e=0$.
Note that by assumption all the elements $[w_{j_r}]$ are distinct elements of $N$ for $1 \leq r \leq n$. Clearly the closed path $\mathcal{C}$ has exactly $n$ rightmost edges if $e=0$, namely the edges $w_{j_r} \cdot \bbe_r$ for $1 \leq r \leq n$; when $e=1$, there are $n-1$ rightmost edges, namely the edges $w_{j_r} \cdot \bbe_r$ for $1 \leq r \leq n-1$.  Since $n\geq 2$, there is always at least one rightmost edge.

\begin{figure}[tb]
\begin{center}
\scalebox{0.8}{
\begin{tikzpicture}
[blackbox/.style={draw, fill=gray!20, rectangle, minimum height=2cm, minimum width=9cm},
transistorbox/.style={draw, fill=blue!20, thick, rectangle, minimum height=0.8cm, minimum width=2.5cm},
dottedbox/.style={draw, rectangle, loosely dashed, minimum height=3.5cm, minimum width=8cm},
dottedboxsmall/.style={draw, rectangle, loosely dashed, minimum height=3.5cm, minimum width=4cm},
clearbox/.style={draw, rectangle, minimum height=2cm, minimum width=7cm},
clearboxsmall/.style={draw, rectangle, minimum height=2cm, minimum width=1cm}]
%
%
%
%
%
%
\draw (1,1) -- (1,3);
\draw (2,1) -- (2,3);
\draw (3,1) -- (3,3);
\draw (4,1) -- (4,3);
\draw (5,1) -- (5,3);
\draw (6,1) -- (6,3);
\draw (7,1) -- (7,3);
\draw (8,1) -- (8,3);
%
%
\draw (1,1-10) -- (1,3-10);
\draw (2,1-10) -- (2,3-10);
\draw (3,1-10) -- (3,3-10);
\draw (4,1-10) -- (4,3-10);
\draw (5,1-10) -- (5,3-10);
\draw (6,1-10) -- (6,3-10);
\draw (7,1-10) -- (7,3-10);
\draw (8,1-10) -- (8,3-10);
%
%
\draw (1,-3) -- (1,-1);
\draw (2,-3) -- (2,-1);
\draw (3,-3) -- (3,-1);
\draw (4,-3) -- (4,-1);
\draw (5,-3) -- (5,-1);
\draw (6,-3) -- (6,-1);
\draw (7,-3) -- (7,-1);
\draw (8,-3) -- (8,-1);
%

%
%
\draw[dotted] (1,-5) -- (1,-3);
\draw[dotted] (3,-5) -- (3,-3);
\draw[dotted] (5,-5) -- (5,-3);
\draw[dotted] (7,-5) -- (7,-3);
\draw[dotted] (9,-5) -- (9,-3);
\draw[dotted] (11,-5) -- (11,-3);
%
%
%
%
\draw (10.5,3) -- (10.5,5);
\draw (10.5,3-10) -- (10.5,5-10);
%
%
\draw (10.5,-1) -- (10.5,1);
%
%
%
\draw (10.5,3) -- (10.5,2.4);
\draw (9.5,3) -- (9.5,2.4);
\draw (8.5,3) -- (8.5,2.4);
%
%
\draw (10.5,3-10) -- (10.5,2.4-10);
\draw (9.5,3-10) -- (9.5,2.4-10);
\draw (8.5,3-10) -- (8.5,2.4-10);
%
%
\draw (10.5,-1) -- (10.5,-1.6);
\draw (9.5,-1) -- (9.5,-1.6);
\draw (8.5,-1) -- (8.5,-1.6);
%
%
%
\draw (10.5,1.6) -- (10.5,1);
\draw (9.5,1.6) -- (9.5,1);
\draw (8.5,1.6) -- (8.5,1);
%
%
\draw (10.5,1.6-10) -- (10.5,1-10);
\draw (9.5,1.6-10) -- (9.5,1-10);
\draw (8.5,1.6-10) -- (8.5,1-10);
%
%
\draw (10.5,1.6-4) -- (10.5,1-4);
\draw (9.5,1.6-4) -- (9.5,1-4);
\draw (8.5,1.6-4) -- (8.5,1-4);
%
%
%
%
%
\node[blackbox] (midleft) at (5.5,4)  {$\bbp_1$};
\node[blackbox] (midleft) at (5.5,4-10)  {$\bbp_n$};
\node[blackbox] (midleft2) at (5.5,0)  {$\bbp_2$};
\node[clearboxsmall] (midright) at (10.5,4)  {};
\node[clearboxsmall] (bottom) at (10.5,4-10)  {};
\node[clearboxsmall] (secondright) at (10.5,0)  {};
\node (label) at (10.7,4)  {$a$};
\node (label) at (10.7,4-10)  {$a$};
\node (label) at (10.7,0)  {$a$};
\node[transistorbox] (firsttransistor) at (9.5,2)  {};
\node[transistorbox] (lasttransistor) at (9.5,2-10)  {};
\node[transistorbox] (2ndtransistor) at (9.5,2-4)  {};
%
%
%
%
\draw[red, very thick]   (1,5) -- node[above]  {$w_{j_1}$} (8,5);
\draw[red, very thick]   (1,5-10) -- node[above]  {$w_{j_n}$} (8,5-10);
\draw[red, very thick]   (1,3) -- node[above]  {$w_{j_2}$} (8,3);
\draw[red, very thick]   (1,3-10) -- node[above]  {$w_{j_1}$} (8,3-10);
\draw[very thick]        (8,5) -- node[above]  {$\varphi(\beta_{j_1})$} (10,5);
\draw[very thick]        (8,5-10) -- node[above]  {$\varphi(\beta_{j_n})$} (10,5-10);
\draw[very thick]                    (10,5) -- node[above]  {$a$} (11,5);
\draw[very thick]                    (10,5-14) -- node[below]  {$a$} (11,5-14);
\draw[red, very thick]   (1,1) -- node[below]  {$w_{j_2}$} (8,1);
\draw[red, very thick]   (1,1-10) -- node[below]  {$w_{j_1}$} (8,1-10);
\draw[red, very thick]   (1,-1) -- node[above]  {$w_{j_3}$} (8,-1);
\draw[red, very thick]   (1,-3) -- node[below]  {$w_{j_3}$} (8,-3);
\draw[very thick]        (8,1) -- node[below]  {$\varphi(\beta_{j_2})$} (10,1);
\draw[very thick]        (8,1-10) -- node[below]  {$\varphi(\beta_{j_1})$} (10,1-10);
\draw[very thick]        (8,-1) -- node[above]  {$\beta_{j_3}$} (10,-1);
\draw[very thick]        (8,-3) -- node[below]  {$\varphi(\beta_{j_3})$} (10,-3);
\draw        (10,-3) -- node[below]  {} (11,-3);
\draw        (10,-3-6) -- node[below]  {} (11,-3-6);
\draw[very thick]        (8,3) -- node[above]  {$\beta_{j_2}$} (10,3);
\draw[very thick]        (8,3-10) -- node[above]  {$\beta_{j_1}$} (10,3-10);
%
%
\end{tikzpicture}
}
\end{center}
\caption{The closed path $\mathcal{C}$ in the Squier complex.}
\label{fig_C}
\end{figure}
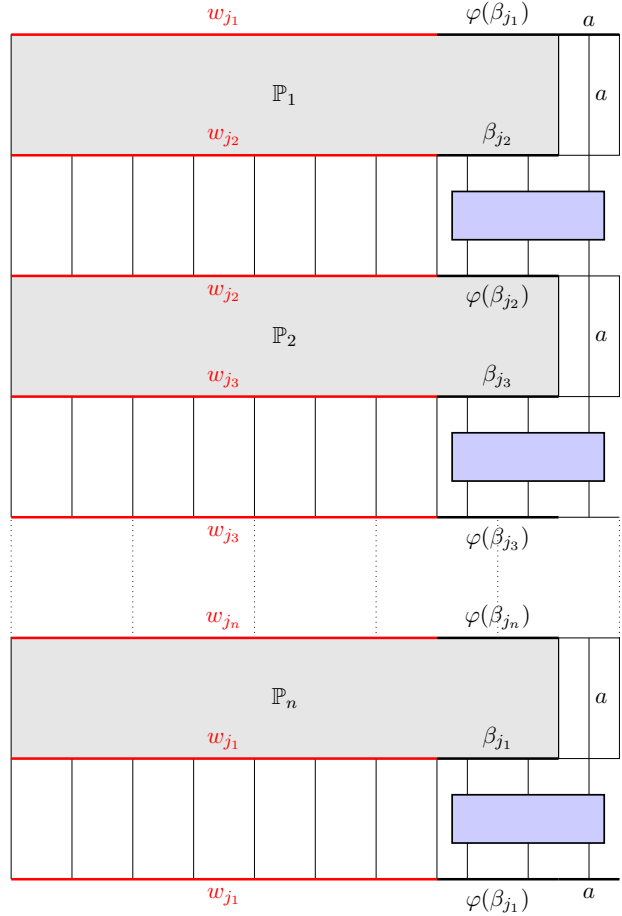

\begin{lemma}\label{lem:not:nul:hom} The closed path $\mathcal{C}$ is not null-homotopic, that is, $\mathcal{C} \not\sim_0 1_{\iota \mathcal{C}}$.    \end{lemma}
\begin{proof}
Choose some rightmost edge $\bbe$ in $\mathcal{C}$ and let $m = L(\bbe) \in N$. By the definition of $\mathcal{C}$ (as the $[w_{j_r}]$ are distinct) we have $f_m(\mathcal{C}) = 1$, that is, $\bbe$ is the unique rightmost edge in $\mathcal{C}$ with image $m$ under the function $L$. On the other hand, the empty path $1_{\iota \mathcal{C}}$ satisfies $f_m(1_{\iota \mathcal{C}}) = 0$ since it has no edges. Now it follows from Lemma~\ref{lem:parity:lemma} that $\mathcal{C} \not\sim_0 1_{\iota \mathcal{C}}$, which completes the proof of the lemma.
\end{proof}

But Lemma~\ref{lem:not:nul:hom} is a contradiction since, by assumption, the presentation $\lb A \mid u=v \rb$ is strictly aspherical. This completes the proof of Proposition~\ref{prop:injective} and hence of Lemma~\ref{l:inj.cut.letter}.

\section{Resolving the relation module
and completing the proof in the non-subspecial case
}
\label{sec:resolving:relation:module}
In this section we will bring together several of the results from previous sections to complete the proof that all one-relator monoids are of type left- and right-$\F_\infty$ and $\FP_\infty$.

Recall that a module $V$ over a ring $R$ is of \emph{type $\FP_n$} if it has a projective resolution
\[\cdots\longrightarrow  P_j\longrightarrow  P_{j-1}\longrightarrow \cdots \longrightarrow P_0\longrightarrow V\longrightarrow 0 \] with $P_j$ finitely generated for $0\leq j\leq n\leq \infty$.  
It is well known that replacing the word ``projective'' by ``free'' in the definition of $\FP_n$ gives the same notion; see~\cite[Proposition~VIII.4.5]{BrownCohomologyBook}.
A useful fact, cf.~\cite{Brown1982} or~\cite[Proposition~2.2]{Strebel1983}, is the following.

\begin{Lemma}\label{l:brown}
Suppose that $R$ is a ring and
\[C_n\longrightarrow C_{n-1}\longrightarrow\cdots\longrightarrow C_0\longrightarrow V\] is a partial resolution of an $R$-module $V$.
 If $C_i$ is of type  $\FP_{n-i}$, for $0\leq i\leq n$, then $V$ is of type $\FP_n$.
\end{Lemma}

Fix now a non-subspecial compressible presentation $M=\langle A\mid u=v\rangle$. Let $y$ be the shortest word compressing $u=v$
(noting that $y$ is
necessarily SOF) and $z$ be the longest word compressing $u=v$.  Assume that $u=zu'=u''y$ and $v=zv'=v''y$.  Put $\Delta=\Delta_z$ and $\Lambda=\Lambda_z$.
Write $z=z'y$. Put $u'=\alpha_1\cdots \alpha_k$ and $v'=\beta_1\cdots\beta_\ell$ with $\alpha_i,\beta_j\in \Lambda$.  Since the presentation is not subspecial, we have that $k,\ell\geq 1$.  Let $S=\lb \Lambda\mid u'=v'\rb$ be the Lallement compression of $M$.

\begin{Prop}\label{p:finding.y.in.delta}
We have $\Delta\subseteq A^*y$.
\end{Prop}
\begin{proof}
Let $\alpha\in \Delta$ and note that $\alpha$ is
not the empty word.
Then $z'y\alpha=z\alpha = \alpha'z=\alpha'z'y$ with $\alpha'\in A^*$.  Since $|\alpha'|=|\alpha|>0$ and $y$ is SOF, the occurrence of $y$ as a suffix of $z'y\alpha$ does not overlap the displayed occurrence of $y$ and so $\alpha\in A^*y$, as required.
\end{proof}

As a consequence of Proposition~\ref{p:finding.y.in.delta} we can write $\alpha_k = \overline{\alpha_k}y$ and $\beta_{\ell} = \overline{\beta_\ell}y$. It follows that $u'' = z\alpha_1\cdots \alpha_{k-1}\overline{\alpha_k}$ and $v'' = z\beta_1\cdots \beta_{\ell-1}\overline{\beta_\ell}$.

\begin{Prop}\label{p:first.kernel}
Let $\varphi\colon \mathbb ZM\to \mathbb ZM[y]_M$ be the epimorphism given by $\varphi(b) = b[y]_M$.  Then $\ker \varphi =   \mathbb ZM([u'']_M-[v'']_M)$.
\end{Prop}
\begin{proof}
Clearly $\mathbb ZM([u'']_M-[v'']_M)\subseteq \ker \varphi$ since $[u''y]_M-[v''y]_M =[u]_M-[v]_M=0$.
Since $\varphi$ is the $\mathbb Z$-linear extension of the mapping of sets $M\to M[y]_M$ given by $m\mapsto m[y]_M$, it follows that $\ker \varphi$ is spanned by all differences $[w]_M-[w']_M$ with $[wy]_M=[w'y]_M$ by Lemma~\ref{l:map.of.set}.  If we perform compression with respect to $y$, then the right canonical form of $wy$ is $wy\varepsilon$ and so if $[wy]_M=[s]_M$, then $s=s'y$ for some $s'\in A^*$ by Lemma~\ref{lem:lalle3.1}.  Therefore, $[wy]_M=[w'y]_M$ implies that there is a sequence $w_1,\cdots, w_n\in A^*$ with $w_1=w$, $w_n=w'$ and $w_{i}y$ transformable to $w_{i+1}y$ with one application of the defining relation $u=v$.  Then we have that \[[w]_M-[w']_M = ([w_1]_M-[w_2]_M)+([w_2]_M-[w_3]_M)+\cdots + ([w_{n-1}]_M-[w_n]_M)\] and so $\ker \varphi$ is spanned by those differences $[s]_M-[s']_M$ so that $sy$ can be transformed into $s'y$ by one application of the relation.  Without loss of generality we may assume that $sy=guh$ and $s'y=gvh$.  Since $u,v\in A^*y$ and $y$ is SOF, either $y$ is a suffix of $h$, and $[s]_M=[s']_M$, or $h=\varepsilon$ and so $s=gu''$ and $s'=gv''$.  In the first case, $[s]_M-[s']_M=0\in \mathbb ZM([u'']_M-[v'']_M)$ and in the second case \[[s]_M-[s']_M = [g]_M([u'']_M-[v'']_M)\in \mathbb ZM([u'']_M-[v'']_M),\] as required.
\end{proof}

Our next goal will be to show that $\mathbb ZM([u'']_M-[v'']_M)\cong \mathbb ZM[z]_M$.  First we need to relate the algebra of a local divisor to that of the monoid.

\begin{Prop}\label{p:to.divisor.alg}
Let $r$ compress $u=v$ and let $N$ be the compression of $M$ with respect to $r$.
Let $w_1,\ldots, w_n\in \Lambda_r^*$ and $c_1,\ldots, c_n\in \mathbb Z$.  Let $a=\sum_{i=1}^n c_i[w_i]_M\in \mathbb ZM$ and let $a'=\sum_{i=1}^n c_i[w_i]_N\in \mathbb ZN$.  Note that $\mathbb ZM[r]_Ma\subseteq \mathbb ZM[r]_M$ as $\Lambda_r^*\subseteq T(r)$.  Suppose that $\gamma\colon \mathbb ZN\to \mathbb ZN$ is given by $\gamma(b')= b'a'$.  Let
\[\eta\colon \mathbb ZM[r]_M\otimes_{\mathbb ZN} \mathbb ZN\to \mathbb ZM[r]_M,\]
given by
\[m\otimes [w]_N\longmapsto m[w]_M\]
for $w \in \Lambda_r^*$
on basic tensors, be the natural isomorphism.  Let $\delta\colon \mathbb ZM[r]_M\to \mathbb ZM[r]_M$  be given by $\delta(b)= ba$.
Then $\delta = \eta(1_{\mathbb ZM[r]_M}\otimes \gamma)\eta\inv$.   In particular, $\delta$ is injective if $\gamma$ is injective and $\mathbb ZM[r]_Ma\cong \mathbb ZM[r]_M\otimes_{\mathbb ZN} \mathbb ZNa'$.
\end{Prop}
\begin{proof}
Recall that $M[r]_M$ is a free right $N$-set by Corollary~\ref{c:free.action} via the action \[(m,[w]_N)\mapsto m[w]_M\]
where $w \in \Lambda_r^*$.
The action of the functor $\mathbb ZM[r]_M\otimes_{\mathbb ZN}(-)$ on the mapping $\gamma$ is as follows
\[(1_{\mathbb ZM[r]_M}\otimes \gamma)(m\otimes [w]_N)= m\otimes\sum_{i=1}^nc_i[ww_i]_N= \sum_{i=1}^nc_im[ww_i]_M\otimes 1\] which maps under the isomorphism $\eta$ to $m[w]_Ma$, whereas $\eta(m\otimes [w]_N) = m[w]_M$.  Thus we have $\delta = \eta (1_{\mathbb ZM[r]_M}\otimes \gamma)\eta\inv$.  Since $\mathbb ZM[r]_M$ is a free right $\mathbb ZN$-module, the functor $\mathbb ZM[r]_M\otimes_{\mathbb ZN}(-)$ is exact and so we deduce that $\delta$ is injective whenever $\gamma$ is injective. Clearly, the image of $1_{\mathbb ZM[r]_M}\otimes \gamma$ is isomorphic to the image $\mathbb ZM[r]_Ma$ of $\delta$ via $\eta$.  But $1_{\mathbb ZM[r]_M}\otimes \gamma$ has image isomorphic to $\mathbb ZM[r]_M\otimes_{\mathbb ZN}\mathbb ZNa'$ by right exactness of the tensor product functor.
\end{proof}

Our next proposition allows us to relate the homological finiteness properties of
the $\ZM$-modules
$\mathbb ZM[y]_M$ and $\mathbb ZM[z]_M$, and is the key technical tool in this section.

\begin{Prop}\label{p:key.exact}
Let $M=\lb A\mid u=v\rb$ be a compressible non-subspecial one-relator monoid presentation.
Let $y$ be the shortest word compressing $u=v$, let $z$ be the longest word compressing $u=v$, and write $u=u''y$ and $v=v''y$.
Then there is an isomorphism \[\mathbb ZM([u'']_M-[v'']_M)\cong \mathbb ZM[z]_M\] and hence there is an exact sequence
\[0\longrightarrow \mathbb ZM[z]_M\longrightarrow \mathbb ZM\longrightarrow \mathbb ZM[y]_M\longrightarrow 0\] of left $\mathbb ZM$-modules.
\end{Prop}
\begin{proof}
We retain the notation introduced earlier in this section.
Let $\widetilde{u} = \alpha_1\cdots \alpha_{k-1}\overline{\alpha_k}$ and $\widetilde{v}=\beta_1\cdots \beta_{\ell-1}\overline{\beta_{\ell}}$.  Then $u''=z\widetilde{u}$ and $v''=z\widetilde{v}$ and so $\mathbb ZM([u'']-[v'']) = \mathbb ZM[z]_M([\widetilde{u}]_M-[\widetilde{v}]_M)$.
Hence it suffices to show that
$\psi\colon \mathbb ZM[z]_M\to \mathbb ZM([u'']-[v''])$
given by $\psi(b)=b([\widetilde{u}]_M-[\widetilde{v}]_M)$ is injective.
We prove this using local divisor and compression techniques
and some results from earlier sections.    Put $\widetilde{u}' = \alpha_1\cdots\alpha_{k-1}\in \Delta^*=T(z)$ and $\widetilde{v}'=\beta_1\cdots\beta_{\ell-1}\in \Delta^*=T(z)$.

The first case we consider is when $\alpha_k\neq \beta_\ell$.  Note that $\alpha_k=\overline{\alpha_k}y$ and $\beta_\ell=\overline{\beta_{\ell}}y$ implies that $\overline{\alpha_k}\neq \overline{\beta_{\ell}}$ and both these elements belong to the set $P_{\Delta}$ of proper prefixes of
words from
$\Delta$.

 Suppose that $b=\sum_{m\in M[z]_M} c_mm$ with $\psi(b)=0$, where $c_m\in \mathbb Z$ and only finitely many are non-zero.  Then we have that
\begin{equation}\label{eq:help.easy.case}
b[\widetilde{u}'\overline{\alpha_{k}}]_M=b[\widetilde{v}'\overline{\beta_{\ell}}]_M.
\end{equation}
Since $\widetilde{u}',\widetilde{v}'\in T(z)$ and $b\in \mathbb ZM[z]_M$ we have that the left hand side belongs to $\mathbb ZM[z]_M[\overline{\alpha_{k}}]_M$ and the right hand side belongs to $\mathbb ZM[z][\overline{\beta_{\ell}}]_M$.  Since $\overline{\alpha_k},\overline{\beta_{\ell}}\in P_{\Delta}$ are distinct, we deduce that $\mathbb ZM[z][\overline{\alpha_{k}}]_M\cap \mathbb ZM[z]_M[\overline{\beta_{\ell}}]_M=0$ by Lemma~\ref{l:embed.prefix}.  Thus we deduce from \eqref{eq:help.easy.case} that $b[\widetilde{u}'\overline{\alpha_k}]_M=0$ and $b[\widetilde{v}\overline{\beta_\ell}]_M=0$.
We now show that $b[\widetilde{u}'\overline{\alpha_k}]_M=0$ implies $b=0$.

 Note that $M[z]_M\to M$ given by $b'\mapsto b'[\overline{\alpha_k}]_M$ is injective by Lemma~\ref{l:embed.prefix}, as $\overline{\alpha_k}\in P_{\Delta}$, and hence so is its linear extension $\mathbb ZM[z]_M\to \mathbb ZM$.  Thus it suffices to show that $\delta\colon \mathbb ZM[z]_M\to \mathbb ZM[z]_M$ given by $\delta(b) = b[\widetilde{u}']_M$ is injective.  Proposition~\ref{p:to.divisor.alg} shows that it is enough to prove that right multiplication by $[\widetilde{u}']_S$ is an injective mapping $\mathbb ZS\to \mathbb ZS$ where $S$ is the Lallement compression of $M$.  But the one-relator presentation $S=\lb \Lambda\mid u'=v'\rb$ has the property that $u'$ and $v'$ end in distinct symbols from $\Lambda$ (as $\alpha_k\neq \beta_{\ell}$).
It is a well-known result of Adjan (cf.~\cite[Corollary~1.7.8]{Higginsbook}) that such one-relator monoids are right cancellative and so right multiplication by  $[\widetilde{u}']_S$  is injective on $S$ and hence its linear extension to $\mathbb ZS$ is also injective.
This completes the proof
that $b[\widetilde{u}'\overline{\alpha_k}]_M=0$ implies $b=0$.

 Next we consider the case $\alpha_k=\beta_\ell$.  In this case, we deduce from $\overline{\alpha_k}y = \alpha_k=\beta_{\ell}=\overline{\beta_{\ell}}y$ that $\overline{\alpha_k}=\overline{\beta_\ell}$; let us denote this element of $P_{\Delta}$ by $s$.  So $[\widetilde{u}]_M-[\widetilde{v}]_M = ([\widetilde{u}']_M-[\widetilde{v}']_M)[s]_M$.  It follows from Lemma~\ref{l:embed.prefix} that right multiplication by $[s]_M$ gives an injective mapping $\mathbb ZM[z]_M\to \mathbb ZM$ and so it suffices to show that right multiplication by $[\widetilde{u}']_M-[\widetilde{v}']_M$ is an injective mapping $\mathbb ZM[z]_M\to \mathbb ZM[z]_M$.  But by Proposition~\ref{p:to.divisor.alg} to prove this, it suffices to show that right multiplication by $[\widetilde{u}']_S-[\widetilde{v}']_S$ is an injective mapping $\mathbb ZS\to \mathbb ZS$.  But this is a consequence of Lemma~\ref{l:inj.cut.letter} since $S$ is incompressible with presentation $\widetilde{u}'\alpha=\widetilde{v}'\alpha$ where $\alpha$ is the common element $\alpha_k=\beta_\ell$ of $\Lambda$.  This completes the proof that $\psi$ is injective.

It now follows that $\mathbb ZM[z]_M\cong \mathbb ZM([u'']_M-[v'']_M)$.
The exact sequence is then a consequence of Proposition~\ref{p:first.kernel}.
\end{proof}

We are nearly ready to prove that $M$ is of type left-$\FP_{\infty}$ (and, of course, right-$\FP_{\infty}$).  We need one last observation.
\begin{Prop}\label{p:iter.the.long}
Let $N$ be the Adjan-Oganesyan compression of $M$.
Write $z=yz''$ where $z$ and $y$ are the longest and shortest words compressing $u=v$, respectively.
Then $z''\in \Lambda^*_y$ and is the longest word in $\Lambda_y^+$ that compresses $N$,
and there is an isomorphism \[\mathbb ZM[y]_M\otimes_{\mathbb ZN}\mathbb ZN[z'']_{N}\cong \mathbb ZM[z]_M\] of $\mathbb ZM$-modules.
\end{Prop}
\begin{proof}
The first claim is a consequence of  Proposition~\ref{p:compress.transitivity}.
For the isomorphism, apply the final statement of Proposition~\ref{p:to.divisor.alg}  with $r=y$, $a=[z'']_M$ and $a'=[z'']_N$ noting that $z=yz''$.
\end{proof}

\begin{Thm}\label{t:complete.fp.infinity}
Let $M=\lb A\mid u=v\rb$ be a compressible non-subspecial one-relator monoid.  Let $z$ be the longest word compressing $u=v$.  Then the relation module $\mathbb ZM[z]_M$ is of type $FP_{\infty}$.
\end{Thm}
\begin{proof}
We prove the result by induction on the number of words compressing $u=v$.  If there is just one, then the Adjan-Oganesyan and Lallement compressions coincide and so $y=z$ in Proposition~\ref{p:key.exact}.  Thus we have an exact sequence
\[0\longrightarrow \mathbb ZM[z]_M\longrightarrow \mathbb ZM\longrightarrow  \mathbb ZM[z]_M\longrightarrow 0.\]
Therefore, we can build a periodic resolution of the form \[\cdots\longrightarrow \mathbb ZM\longrightarrow \mathbb ZM\longrightarrow \mathbb ZM[z]_M\longrightarrow 0,\] where at each stage the kernel and image of the map $\mathbb ZM\to \mathbb ZM$ are both isomorphic to $\mathbb ZM[z]_M$, by splicing the previous exact sequence onto its left end repeatedly.

Next assume that the theorem is true for all compressible non-subspecial presentations that can be compressed by fewer words and let $y\neq z$ be the, respective, shortest and longest words compressing $u=v$ and write $z=yz''$.  Let $N$ be the
Adjan-Oganesyan compression of $M$.  Note that $N$ has fewer words compressing it than $M$ and that $z''\in \Lambda_y^+$ is the longest word compressing the defining relation of $N$ by Proposition~\ref{p:compress.transitivity} and Proposition~\ref{p:iter.the.long}.  By Corollary~\ref{c:stay.sub}, $N$ is a non-subspecial compressible one-relator monoid.
 Thus $\mathbb ZN[z'']_N$ has a free resolution $F_\bullet\to \mathbb ZM[z'']_N$ by finitely generated free $\mathbb ZN$-modules by induction.  Since $\mathbb ZM[y]_M$ is a free right $\mathbb ZN$-module by Corollary~\ref{c:free.action}, and hence flat, tensoring this resolution with $\mathbb ZM[y]_M$ yields a resolution
\begin{equation}\label{eq:resol.by.y}
\mathbb ZM[y]_M\otimes_{\mathbb ZN} F_\bullet \to \mathbb ZM[z]_M
\end{equation}
by Proposition~\ref{p:iter.the.long}.
Moreover, each module $\mathbb ZM[y]_M\otimes_{\mathbb ZN} F_n$, with $F_n$ a finitely generated free $\mathbb ZN$-module, is isomorphic to a finite direct sum of copies of $\mathbb ZM[y]_M$, since we recall if $V$ is an $S$-$R$-bimodule, then $V\otimes_R R^k\cong (V\otimes_R R)^k\cong V^k$ as an $S$-module for any $k\geq 0$.

We now prove by induction on $n$ that $\mathbb ZM[z]_M$ is of type $\FP_n$.  If $n=0$, this is clear since $\mathbb ZM[z]_M$ is a  principal left ideal and hence a finitely generated (in fact, cyclic) module.
Now assume inductively that $\mathbb ZM[z]_M$ is of type $\FP_n$ for some $n\geq 0$.  Then the exact sequence in Proposition~\ref{p:key.exact} and Lemma~\ref{l:brown} imply that $\mathbb ZM[y]_M$ is of type $\FP_{n+1}$ and hence so is any finite direct sum of copies of $\mathbb ZM[y]_M$.
But \eqref{eq:resol.by.y} is a resolution of $\mathbb ZM[z]_M$ by the modules $\mathbb ZM[y]_M\otimes_{\mathbb ZN} F_n$, which are finite direct sums of isomorphic copies of $\mathbb ZM[y]_M$,
and hence another application of Lemma~\ref{l:brown} yields that $\mathbb ZM[z]_M$ if of type $\FP_{n+1}$.
This proves by induction that $\mathbb ZM[z]_M$ is of type $\FP_n$ for all $n\geq 0$ and hence
by standard general results (see e.g.~\cite[Proposition~VIII.4.5]{BrownCohomologyBook})
it
then follows that $\mathbb ZM[z]_M$
is of type $\FP_{\infty}$.
\end{proof}

We now prove Theorem~\ref{theorem:A:full} for non-subspecial one-relator monoids, thereby completing its proof.

\begin{Thm}\label{t:nonsub}
Let $M$ be a non-subspecial one-relator monoid.  Then $M$ is of type left- and right-$\FP_{\infty}$ and of type left- and right-$\F_\infty$.
\end{Thm}
\begin{proof}
Recall that for finitely presented monoids, the properties left-$\F_\infty$ and left-$\FP_{\infty}$ are equivalent~\cite{GraySteinberg1}.
If the presentation is incompressible, then it is strictly aspherical by~\cite[Corollary~5.6]{Kobayashi1998} and hence of type left-$\F_\infty$ since it has a contractible Cayley complex by Lemma~\ref{lem:9}. If the presentation is compressible, this is an immediate consequence of Theorem~\ref{t:complete.fp.infinity} since we can splice
a free resolution of $\mathbb ZM[z]_M$ by finitely generated free modules onto the exact sequence of Theorem~\ref{t:relation.module.compress}.
\end{proof}

\section{Torsion-free one-relator monoids of infinite cohomological dimension}
\label{s:cd}
The proof of Theorem~\ref{t:nonsub} would suggest that a compressible non-subspecial one-relator monoid could have infinite cohomological dimension as the free resolution implicit in the proof is infinite.
We prove that this is the case for a large family of examples thereby completing the proof of Theorem~\ref{theorem:B:full} and of Theorem~B.

Let us say that a one-relator presentation $\lb A\mid u=v\rb$ is \emph{one-step compressible} if its Adjan-Oganesyan and Lallement compressions coincide, that is, there is a unique word $y\in A^+$ compressing $u=v$.  We will show that any one-step compressible non-subspecial monoid has infinite cohomological dimension.
A typical example is $M=\lb a,b,c\mid aba=aca\rb$.
In particular this shows that, in contrast to groups, there are torsion-free one-relator monoids with infinite (left and right) cohomological dimension.

First we recall the basic fact that a principal left ideal $Ra$ of a ring $R$ is a projective module if and only if the left annihilator \[\mathrm{ann}_L(a) =\{r\in R\mid ra=0\}\] of $a$ is of the form $Re$ for some idempotent $e\in R$.  Indeed, there is a short exact sequence of $R$-modules
\[0\longrightarrow \mathrm{ann}_L(a)\longrightarrow R\xrightarrow{\,\,f\,\,} Ra\longrightarrow 0\] where $f(r)=ra$.  Hence if $Ra$ is projective, then the sequence splits and thus $\mathrm{ann}_L(a)$ is a direct summand in $R$.  But all direct summands in $R$ are of the form $Re$ for some idempotent $e\in R$ because every endomorphism of the regular $R$-module $R$ is given via right multiplication by an element of $R$.  Conversely, if $\mathrm{ann}_L(a)=Re$, then $Ra\cong R/Re\cong  R(1-e)$ and hence is  projective as $1-e$ is an idempotent.

A ring $R$ has \emph{only trivial idempotents} if $0$ and $1$ are the only idempotents in $R$.  For example, a  famous conjecture of Kaplansky asserts that if $G$ is a torsion-free group, then the group ring $\mathbb ZG$ has only trivial idempotents. By the observation in the previous paragraph, if a ring $R$ has only trivial idempotents and $0\neq a\in R$, then $Ra$ is projective if and only if right multiplication by $a$ is an injective mapping.
We now aim to prove that if $M$ is a non-subspecial one-relator monoid, then $\mathbb ZM$ has  only trivial idempotents.  We shall first need the following result, which for the most part encapsulates~\cite[Theorem~2.3]{Lallement1974} and~\cite[Corollary~2.4]{Lallement1974}.  Recall that a monoid $M$ is \emph{$\gr$-trivial} if Green's relation $\gr$ is the equality relation, i.e., $mM=m'M$ implies $m=m'$; \emph{$\gl$-trivial} monoids are defined dually.

\begin{Prop}\label{p:lallement.rtriv}
Let $M=\lb A\mid u=v\rb$ be a one-relator monoid with $1\leq |v|\leq |u|$.
\begin{enumerate}
  \item[(i)] If $u\notin vA^*$, then $M$ is $\gr$-trivial and $[w]_M=[ww']_M$ implies $w'$ is empty for $w,w'\in A^*$.
  \item[(ii)] If $u\notin A^*v$, then $M$ is $\gl$-trivial and $[w]_M=[w'w]_M$ implies $w'$ is empty for $w,w'\in A^*$.
\end{enumerate}
\end{Prop}
\begin{proof}
If $|v|<|u|$, then the result follows from~\cite[Theorem~2.3]{Lallement1974} and~\cite[Corollary~2.4]{Lallement1974}.  If $|u|=|v|$, then there is a well-defined homomorphism $\ell\colon M\to \mathbb N$ given by $\ell([w]_M) = |w|$.  Let us prove (i), as (ii) is dual.   If $[w]_M=[ww']_M$, then $|w|=\ell([w]_M) = \ell([ww']_M) = |w|+|w'|$ and so $w'$ is empty.  If $M$ is not $\gr$-trivial, then there exist $w,w_1,w_2\in A^*$ with $[w]_M = [ww_1w_2]_M$ and $[w]_M\neq [ww_1]_M$.  But $w_1w_2$ must be empty by what we just proved and so $[w]_M=[ww_1]_M$, a contradiction.  Thus $M$ is $\gr$-trivial.
\end{proof}

Recall that the \emph{$\gr$-order} on a monoid $M$ is defined by $m\leq_{\gr} n$ if $mM\subseteq nM$.  This is a preorder on $M$ and is a partial order precisely when $M$ is $\gr$-trivial.
We can now prove the desired result.

\begin{Thm}\label{t:no.idems}
Let $M=\lb A\mid u=v\rb$ be a non-subspecial one-relator monoid.  Then $\mathbb ZM$ has only trivial idempotents.
\end{Thm}
\begin{proof}
Assume that $|v|\leq |u|$.  Since the relation is not subspecial, either $v$ is not a prefix of $u$ or $v$ is not a suffix of $u$.  By symmetry, we may assume that $v$ is not a prefix of $u$.  Hence $M$ is $\gr$-trivial by Proposition~\ref{p:lallement.rtriv}.
Let $e\in \mathbb ZM$ be a non-zero idempotent.  Let $X$ be the support of $e$, i.e., the finite set of elements of $M$ with non-zero coefficient in $e$.  First we claim that $1\in X$.  Choose $m\in X$ maximal with respect to the partial order $\leq_{\gr}$.  Then since $e^2=e$, there must exist $a,b\in X$ with $m=ab$.  By maximality of $m$, since $M$ is $\gr$-trivial, we must have $a=m$.  Let $w,w'\in A^*$ with $[w]_M=m$ and $[w']_M=b$.  Then $[w]_M=[ww']_M$ and so $w'$ is empty by Proposition~\ref{p:lallement.rtriv}(i).  Thus $b=1$ and so $1\in X$.

Let $c\neq 0$ be the coefficient of $1$ in $e$.  Since $M$ is not special, $|u|\geq |v|>0$.  Thus no non-empty word in $A^*$ represents $1$ in $M$.  It then follows that $c^2$ is the coefficient of $1$ in $e^2$ and hence $c^2=c$ as $e^2=e$.  Thus $c=1$ since $c$ is a non-zero integer.  We conclude that $e=1+r$ with $r\in \mathbb ZM$ having support not containing $1$.  Then $1-e = -r$ is an idempotent without $1$ in its support and hence $r=0$ by the previous paragraph.  Thus $e=1$, as required.
\end{proof}

\begin{Cor}\label{c:not.proj}
Let $M=\lb A\mid u=v\rb$ be a non-subspecial one-relator monoid.  Let $r\in A^+$ compress the relation $u=v$.  Then $\mathbb ZM[r]_M$ is not a projective module.
\end{Cor}
\begin{proof}
Assume that $|v|\leq |u|$.
By Theorem~\ref{t:no.idems}, the ring $\mathbb ZM$ has no non-trivial idempotents.
Hence,
by the comments preceding the statement of Proposition~\ref{p:lallement.rtriv}, we have that
$\mathbb ZM[r]_M$ is projective if and only if right multiplication by $[r]_M$ is injective on $\mathbb ZM$ or, equivalently, on $M$.  But we can write $u=u_1r$ and $v=v_1r$ and hence $[u_1]_M[r]_M = [u]_M=[v]_M=[v_1]_M[r]_M$.
It remains to show that $[u_1]_M\neq [v_1]_M$, as then right multiplication by $[r]_M$ will not be injective.  Since the relation is not subspecial, $u\neq v$ and hence $u_1\neq v_1$.  But $|v_1|<|v|\leq |u|$ and hence $v_1$ is not equivalent in $M$ to any other word.   Thus $[u_1]_M\neq [v_1]_M$ and so $\mathbb ZM[r]_M$ is not projective.
\end{proof}

As a consequence, we may now deduce that all compressible non-subspecial one-relator monoids have cohomological dimension at least $3$.

\begin{Cor}\label{c:at.least.3}
Let $M=\lb A\mid u=v\rb$ be a compressible non-subspecial one-relator monoid.  Then $\mathop{\mathrm{cd}^{\mathrm{(l)}}}(M)\geq 3$ and $\mathop{\mathrm{cd}^{\mathrm{(r)}}}(M)\geq 3$ and the same inequality applies to the left and right geometric dimensions.
\end{Cor}
\begin{proof}
We just need to handle left cohomological dimension.  By Theorem~\ref{t:relation.module.compress}, we have an exact sequence
\[0\longrightarrow \mathbb ZM[z]_M\longrightarrow \mathbb ZM^{|A|}\longrightarrow \mathbb ZM\longrightarrow \mathbb Z\longrightarrow 0\] with $z\in A^+$ the longest word compressing $u=v$.  The module $\mathbb ZM[z]_M$ is not projective by Corollary~\ref{c:not.proj} and hence the left cohomological dimension of $M$ is at least $3$ by~\cite[Lemma VIII.2.1]{BrownCohomologyBook}.
The last claim on geometric dimensions then follows from the fact that in general the cohomological dimension is a lower bound for the geometric dimension.
\end{proof}

\begin{Prop}\label{p:inf.proj.dim}
Let $M=\lb A\mid u=v\rb$ be a one-step compressible, non-subspecial one-relator monoid.   Then $\mathop{\mathrm{gd}^{\mathrm{(l)}}}(M)=\mathop{\mathrm{cd}^{\mathrm{(l)}}}(M)=\infty=\mathop{\mathrm{cd}^{\mathrm{(r)}}}(M)=\mathop{\mathrm{gd}^{\mathrm{(r)}}}(M)$.
\end{Prop}
\begin{proof}
Since the hypothesis on $M$ is left-right dual, it suffices to prove that $\mathop{\mathrm{cd}^{\mathrm{(l)}}}(M)=\infty$.  Let $y\in A^+$ be the unique word compressing $u=v$.  We have exact sequences
\begin{gather*}
0\longrightarrow \mathbb ZM[y]_M\longrightarrow \mathbb ZM^{|A|}\longrightarrow \mathbb ZM\longrightarrow \mathbb Z\longrightarrow 0\\
0\longrightarrow \mathbb ZM[y]_M\longrightarrow \mathbb ZM\longrightarrow \mathbb ZM[y]_M\longrightarrow 0
\end{gather*}
by Theorem~\ref{t:relation.module.compress} and Proposition~\ref{p:key.exact}.  By repeatedly splicing the second sequence onto the first, it follows that for any $n\geq 0$, we can find an exact sequence
\[0\longrightarrow \mathbb ZM[y]_M\longrightarrow \mathbb ZM\longrightarrow \cdots \longrightarrow \mathbb ZM\longrightarrow \mathbb ZM^{|A|}\longrightarrow \mathbb ZM\longrightarrow \mathbb Z\longrightarrow 0 \] where there are $n$ occurrences of $\mathbb ZM$ between $\mathbb ZM[y]_M$ and $\mathbb ZM^{|A|}$.  If $\mathop{\mathrm{cd}^{\mathrm{(l)}}}(M)\leq n$, then it would follow from the existence of this exact sequence
and~\cite[Lemma VIII.2.1]{BrownCohomologyBook} that $\mathbb ZM[y]_M$ is projective, contradicting Corollary~\ref{c:not.proj}.  Thus $\mathop{\mathrm{cd}^{\mathrm{(l)}}}(M)=\infty$, as required.
\end{proof}

This completes the proof of Theorem~\ref{theorem:B:full}.

For example, $M=\lb a,b,c\mid aba=aca\rb$ meets the conditions of Proposition~\ref{p:inf.proj.dim} and hence has infinite left and right cohomological dimensions.  This example was first considered by Ivanov~\cite{Ivanov}, who observed that its relation module is not freely generated by the coset of the relation.
Since $M$ has infinite left cohomological dimension, it cannot have a contractible Cayley complex. Let us see that directly.  The reader should draw the cells $c_1,c_2$ in the Cayley complex of this presentation based at $[ab]_M$ and at $[ac]_M$ and notice that they create a $2$-sphere with two thorns sticking out.  Thus this Cayley complex is not contractible, in fact $c_1-c_2$ is a $2$-cycle, and, in particular, the presentation is not strictly aspherical.
For this particular example,
the fact that the presentation is not strictly aspherical
was originally
observed by Ivanov in~\cite{Ivanov}.

It would be interesting to know whether all compressible non-subspecial one-relator monoids have infinite cohomological dimension. We suspect this to be the case.


\begin{thebibliography}{10}

\bibitem{Adjan1966}
S.~I. Adjan.
\newblock Defining relations and algorithmic problems for groups and
  semigroups.
\newblock {\em Trudy Mat. Inst. Steklov.}, 85:123, 1966.

\bibitem{Adyan1987}
S.~I. Adjan and G.~U. Oganesyan.
\newblock On the word and divisibility problems for semigroups with one
  relation.
\newblock {\em Mat. Zametki}, 41(3):412--421, 458, 1987.

\bibitem{Alonso2003}
J.~M. Alonso and S.~M. Hermiller.
\newblock Homological finite derivation type.
\newblock {\em Internat. J. Algebra Comput.}, 13(3):341--359, 2003.

\bibitem{Anick1986}
D.~J. Anick.
\newblock On the homology of associative algebras.
\newblock {\em Trans. Amer. Math. Soc.}, 296(2):641--659, 1986.

\bibitem{Bieri1987}
R.~Bieri, W.~D. Neumann, and R.~Strebel.
\newblock A geometric invariant of discrete groups.
\newblock {\em Invent. Math.}, 90(3):451--477, 1987.

\bibitem{Bieri1988}
R.~Bieri and B.~Renz.
\newblock Valuations on free resolutions and higher geometric invariants of
  groups.
\newblock {\em Comment. Math. Helv.}, 63(3):464--497, 1988.

\bibitem{BookAndOtto}
R.~V. Book and F.~Otto.
\newblock {\em String-rewriting systems}.
\newblock Texts and Monographs in Computer Science. Springer-Verlag, New York,
  1993.

\bibitem{Boone1959}
W.~W. Boone.
\newblock The word problem.
\newblock {\em Ann. of Math. (2)}, 70:207--265, 1959.

\bibitem{Brown1982}
K.~S. Brown.
\newblock Complete {E}uler characteristics and fixed-point theory.
\newblock {\em J. Pure Appl. Algebra}, 24(2):103--121, 1982.

\bibitem{Brown1989}
K.~S. Brown.
\newblock The geometry of rewriting systems: a proof of the
  {A}nick-{G}roves-{S}quier theorem.
\newblock In {\em Algorithms and classification in combinatorial group theory
  ({B}erkeley, {CA}, 1989)}, volume~23 of {\em Math. Sci. Res. Inst. Publ.},
  pages 137--163. Springer, New York, 1992.

\bibitem{BrownCohomologyBook}
K.~S. Brown.
\newblock {\em Cohomology of groups}, volume~87 of {\em Graduate Texts in
  Mathematics}.
\newblock Springer-Verlag, New York, 1994.
\newblock Corrected reprint of the 1982 original.

\bibitem{Cockcroft1954}
W.~H. Cockcroft.
\newblock On two-dimensional aspherical complexes.
\newblock {\em Proc. London Math. Soc. (3)}, 4:375--384, 1954.

\bibitem{Cohen1992}
D.~E. Cohen.
\newblock A monoid which is right {$FP\sb \infty$} but not left {$FP\sb 1$}.
\newblock {\em Bull. London Math. Soc.}, 24(4):340--342, 1992.

\bibitem{CS15}
A.~Costa and B.~Steinberg.
\newblock The {S}ch\"{u}tzenberger category of a semigroup.
\newblock {\em Semigroup Forum}, 91(3):543--559, 2015.

\bibitem{Cremanns94}
R.~Cremanns and F.~Otto.
\newblock Finite derivation type implies the homological finiteness condition
  {${\rm FP}_3$}.
\newblock {\em J. Symbolic Comput.}, 18(2):91--112, 1994.

\bibitem{Dershowitz2005}
N.~Dershowitz.
\newblock Open. {C}losed. {O}pen.
\newblock In {\em Term rewriting and applications}, volume 3467 of {\em Lecture
  Notes in Comput. Sci.}, pages 376--393. Springer, Berlin, 2005.

\bibitem{Diekert2006}
V.~Diekert and P.~Gastin.
\newblock Pure future local temporal logics are expressively complete for
  {M}azurkiewicz traces.
\newblock {\em Inform. and Comput.}, 204(11):1597--1619, 2006.

\bibitem{Diekert2016}
V.~Diekert and M.~Kufleitner.
\newblock A survey on the local divisor technique.
\newblock {\em Theoret. Comput. Sci.}, 610(part A):13--23, 2016.

\bibitem{DyerVasquez1973}
E.~Dyer and A.~T. Vasquez.
\newblock Some small aspherical spaces.
\newblock {\em J. Austral. Math. Soc.}, 16:332--352, 1973.
\newblock Collection of articles dedicated to the memory of Hanna Neumann, III.

\bibitem{GeogheganBook}
R.~Geoghegan.
\newblock {\em Topological methods in group theory}, volume 243 of {\em
  Graduate Texts in Mathematics}.
\newblock Springer, New York, 2008.

\bibitem{GraySteinberg1}
R.~D. Gray and B.~Steinberg.
\newblock Topological finiteness properties of monoids, part 1: Foundations.
\newblock {\em arXiv preprint arXiv:1706.04387}, 2017.

\bibitem{GraySteinberg2}
R.~D. Gray and B.~Steinberg.
\newblock Topological finiteness properties of monoids. part 2: special
  monoids, one-relator monoids, amalgamated free products, and hnn extensions.
\newblock {\em arXiv preprint arXiv:1805.03413}, 2018.

\bibitem{GrayRAAG}
R.~D. Gray.
\newblock Undecidability of the word problem for one-relator inverse monoids
  via right-angled Artin subgroups of one-relator groups.
\newblock {\em Invent. Math.}, (2019)
  https://doi.org/10.1007/s00222-019-00920-2.

\bibitem{Guba1998}
V.~S. Guba and S.~J. Pride.
\newblock On the left and right cohomological dimension of monoids.
\newblock {\em Bull. London Math. Soc.}, 30(4):391--396, 1998.

\bibitem{GubaSapirDirected}
V.~S. Guba and M.~V. Sapir.
\newblock Diagram groups and directed 2-complexes: homotopy and homology.
\newblock {\em J. Pure Appl. Algebra}, 205(1):1--47, 2006.

\bibitem{GubaSapir1997}
V.~Guba and M.~Sapir.
\newblock Diagram groups.
\newblock {\em Mem. Amer. Math. Soc.}, 130(620):viii+117, 1997.

\bibitem{Hatcher2002}
A.~Hatcher.
\newblock {\em Algebraic topology}.
\newblock Cambridge University Press, Cambridge, 2002.

\bibitem{Higginsbook}
P.~M. Higgins.
\newblock {\em Techniques of semigroup theory}.
\newblock Oxford Science Publications. The Clarendon Press, Oxford University
  Press, New York, 1992.
\newblock With a foreword by G. B. Preston.

\bibitem{HoltBook}
D.~F. Holt, B.~Eick, and E.~A. O'Brien.
\newblock {\em Handbook of computational group theory}.
\newblock Discrete Mathematics and its Applications (Boca Raton). Chapman \&
  Hall/CRC, Boca Raton, FL, 2005.

\bibitem{Howie}
J.~M. Howie.
\newblock {\em Fundamentals of semigroup theory}.
\newblock Academic Press [Harcourt Brace Jovanovich Publishers], London, 1995.
\newblock L.M.S. Monographs, No. 7.

\bibitem{Ivanov}
S.~V. Ivanov.
\newblock Relation modules and relation bimodules of groups, semigroups and
  associative algebras.
\newblock {\em Internat. J. Algebra Comput.}, 1(1):89--114, 1991.

\bibitem{Kobayashi1998}
Y.~Kobayashi.
\newblock Homotopy reduction systems for monoid presentations: asphericity and
  low-dimensional homology.
\newblock {\em J. Pure Appl. Algebra}, 130(2):159--195, 1998.

\bibitem{Kobayashi2000}
Y.~Kobayashi.
\newblock Finite homotopy bases of one-relator monoids.
\newblock {\em J. Algebra}, 229(2):547--569, 2000.

\bibitem{Lafont95}
Y.~Lafont.
\newblock A new finiteness condition for monoids presented by complete
  rewriting systems (after {C}raig {C}. {S}quier).
\newblock {\em J. Pure Appl. Algebra}, 98(3):229--244, 1995.

\bibitem{Lallement1974}
G.~Lallement.
\newblock On monoids presented by a single relation.
\newblock {\em J. Algebra}, 32:370--388, 1974.

\bibitem{Lallement1988}
G.~Lallement.
\newblock Some algorithms for semigroups and monoids presented by a single
  relation.
\newblock In {\em Semigroups, theory and applications ({O}berwolfach, 1986)},
  volume 1320 of {\em Lecture Notes in Math.}, pages 176--182. Springer,
  Berlin, 1988.

\bibitem{Lyndon1950}
R.~C. Lyndon.
\newblock Cohomology theory of groups with a single defining relation.
\newblock {\em Ann. of Math. (2)}, 52:650--665, 1950.

\bibitem{LyndonAndSchupp}
R.~C. Lyndon and P.~E. Schupp.
\newblock {\em Combinatorial group theory}.
\newblock Classics in Mathematics. Springer-Verlag, Berlin, 2001.
\newblock Reprint of the 1977 edition.

\bibitem{MacLaneBook}
S.~MacLane.
\newblock {\em Categories for the working mathematician}.
\newblock Springer-Verlag, New York-Berlin, 1971.
\newblock Graduate Texts in Mathematics, Vol. 5.

\bibitem{Magnus1932}
W. Magnus.
\newblock
Das {I}dentit\"atsproblem f\"ur {G}ruppen mit einer definierenden {R}elation.
\newblock
{\em  Math. Ann.} 106(1):295--307, 1932).


\bibitem{Malheiro2005}
A.~Malheiro.
\newblock Complete rewriting systems for codified submonoids.
\newblock {\em Internat. J. Algebra Comput.}, 15(2):207--216, 2005.

\bibitem{Margolis:1995qo}
S.~Margolis, J.~Meakin, and M.~Sapir.
\newblock Algorithmic problems in groups, semigroups and inverse semigroups.
\newblock In {\em Semigroups, formal languages and groups ({Y}ork, 1993)},
  volume 466 of {\em NATO Adv. Sci. Inst. Ser. C Math. Phys. Sci.}, pages
  147--214. Kluwer Acad. Publ., Dordrecht, 1995.

\bibitem{Markov1947}
A.~Markov.
\newblock On the impossibility of certain algorithms in the theory of
  associative systems.
\newblock {\em C. R. (Doklady) Acad. Sci. URSS (N.S.)}, 55:583--586, 1947.

\bibitem{Matiyasevich1967}
Y.~Matiyasevich.
\newblock Simple examples of unsolvable associative calculi.
\newblock {\em Dokl. Akad. Nauk SSSR}, 173:1264--1266, 1967.

\bibitem{Matiyasevich2005}
Y.~Matiyasevich and G.~S\'{e}nizergues.
\newblock Decision problems for semi-{T}hue systems with a few rules.
\newblock {\em Theoret. Comput. Sci.}, 330(1):145--169, 2005.

\bibitem{Novikov1952}
P.~S. Novikov.
\newblock On algorithmic unsolvability of the problem of identity.
\newblock {\em Doklady Akad. Nauk SSSR (N.S.)}, 85:709--712, 1952.

\bibitem{Otto1997}
F.~Otto and Y.~Kobayashi.
\newblock Properties of monoids that are presented by finite convergent
  string-rewriting systems---a survey.
\newblock In {\em Advances in algorithms, languages, and complexity}, pages
  225--266. Kluwer Acad. Publ., Dordrecht, 1997.

\bibitem{Post1947}
E.~L. Post.
\newblock Recursive unsolvability of a problem of {T}hue.
\newblock {\em J. Symbolic Logic}, 12:1--11, 1947.

\bibitem{Pride95}
S.~J. Pride.
\newblock Low-dimensional homotopy theory for monoids.
\newblock {\em Internat. J. Algebra Comput.}, 5(6):631--649, 1995.

\bibitem{Pride2004}
S.~J. Pride and F.~Otto.
\newblock For rewriting systems the topological finiteness conditions {FDT} and
  {FHT} are not equivalent.
\newblock {\em J. London Math. Soc. (2)}, 69(2):363--382, 2004.

\bibitem{Serre}
J.-P. Serre.
\newblock {\em Trees}.
\newblock Springer Monographs in Mathematics. Springer-Verlag, Berlin, 2003.
\newblock Translated from the French original by John Stillwell, Corrected 2nd
  printing of the 1980 English translation.

\bibitem{Shikishima1997}
K.~Shikishima-Tsuji, M.~Katsura, and Y.~Kobayashi.
\newblock On termination of confluent one-rule string-rewriting systems.
\newblock {\em Inform. Process. Lett.}, 61(2):91--96, 1997.

\bibitem{Squier1987}
C.~C. Squier.
\newblock Word problems and a homological finiteness condition for monoids.
\newblock {\em J. Pure Appl. Algebra}, 49(1-2):201--217, 1987.

\bibitem{Squier1994}
C.~C. Squier, F.~Otto, and Y.~Kobayashi.
\newblock A finiteness condition for rewriting systems.
\newblock {\em Theoret. Comput. Sci.}, 131(2):271--294, 1994.

\bibitem{Strebel1983}
R.~Strebel.
\newblock On quotients of groups having finite homological type.
\newblock {\em Arch. Math. (Basel)}, 41(5):419--426, 1983.

\bibitem{Turing1950}
A.~M. Turing.
\newblock The word problem in semi-groups with cancellation.
\newblock {\em Ann. of Math. (2)}, 52:491--505, 1950.

\bibitem{Shirshov62}
A.~I. \v{S}ir\v{s}ov.
\newblock Some algorithm problems for {L}ie algebras.
\newblock {\em Sibirsk. Mat. \v{Z}.}, 3:292--296, 1962.

\end{thebibliography}
\end{document}